\documentclass[11pt]{amsart}
\usepackage{amsmath, amsthm, thmtools, amssymb, colonequals, mathrsfs, environ,
tabu, stmaryrd, tikz, enumitem, thm-restate, fullpage}
\usepackage[T1]{fontenc}

\usepackage[
	backref,
	pdfauthor={Vishal Arul}, 
]{hyperref}
\hypersetup{pdftitle=Torsion points on Fermat quotients of the form y\^{}n\ =\
x\^{}d + 1}
\usepackage[alphabetic,backrefs,lite,nobysame]{amsrefs} 
\usepackage{nameref, cleveref}

\hypersetup{
    colorlinks,
    linkcolor={red!50!black},
    citecolor={blue!50!black},
    urlcolor={blue!80!black}
}

\declaretheorem[name=Theorem,
style=plain,
numberwithin=subsection,
refname={Theorem,Theorems},
Refname={Theorem,Theorems}]{theorem}

\declaretheorem[name=Lemma,
style=plain,
sharenumber=theorem,
refname={Lemma,Lemmas},
Refname={Lemma,Lemmas}]{lemma}

\declaretheorem[name=Corollary,
style=plain,
sharenumber=theorem,
refname={Corollary,Corollaries},
Refname={Corollary,Corollaries}]{corollary}

\declaretheorem[name=Proposition,
style=plain,
sharenumber=theorem,
refname={Proposition,Propositions},
Refname={Proposition,Propositions}]{proposition}

\declaretheorem[name=Definition,
style=definition,
sharenumber=theorem,
refname={Definition,Definitions},
Refname={Definition,Definitions}]{definition}

\declaretheorem[name=Remark,
style=remark,
sharenumber=theorem,
refname={Remark,Remarks},
Refname={Remark,Remarks}]{remark}

\makeatletter
\renewcommand{\itemautorefname}{\@gobble}
\makeatother

\crefname{enumi}{}{}

\makeatletter
\renewcommand*{\eqref}[1]{%
  \hyperref[{#1}]{\textup{\tagform@{\ref*{#1}}}}%
}
\makeatother

\numberwithin{equation}{section}

\SetSymbolFont{stmry}{bold}{U}{stmry}{m}{n}

\DeclareMathOperator{\Spl}{Spl}
\DeclareMathOperator{\End}{End}

\DeclareMathOperator{\Princ}{Princ}
\DeclareMathOperator{\Div}{Div}

\DeclareMathOperator{\characteristic}{char}

\DeclareMathOperator{\tors}{tors}
\DeclareMathOperator{\divisor}{div}
\DeclareMathOperator{\Gal}{Gal}
\DeclareMathOperator{\Fix}{Fix}
\DeclareMathOperator{\Aut}{Aut}

\DeclareMathOperator{\WM}{WM}
\DeclareMathOperator{\wt}{wt}
\DeclareMathOperator{\Frob}{Frob}
\DeclareMathOperator{\lcm}{lcm}

\DeclareMathOperator{\Spec}{Spec}

\DeclareMathOperator{\Norm}{Norm}
\DeclareMathOperator{\AJ}{AJ}

\makeatletter
\g@addto@macro\bfseries{\boldmath} 
\makeatother

\begin{document}
\author{Vishal Arul}
\title{Torsion points on Fermat quotients of the form $y^{n} = x^{d} + 1$}
\thanks{This research was supported in part by grants from the Simons Foundation
(\#402472 to Bjorn Poonen, and \#550033).}

\begin{abstract}
We study geometric torsion points on curves of the form $y^{n} = x^{d} + 1$
where $n$ and $d$ are coprime. When $n + d \ge 8$, we show that the only torsion
points on this curve are: (i) those whose $x$-coordinate is zero, (ii) those
whose $y$-coordinate is zero, (iii) the point at infinity. When $n + d = 7$,
there are more torsion points and we classify them all. In addition, we classify
all geometric torsion points on the generic superelliptic curve $y^n = (x - a_1)
\cdots (x - a_d)$ when $(n, d) = 1$ and $n, d \ge 2$.
\end{abstract}

\maketitle

\section{Introduction and summary of new results}
\label{Section:TorsionSummary}

\begin{definition}
Suppose that $X$ is a smooth proper geometrically irreducible curve defined over
a field $K$. Let $J$ be the jacobian variety of $X$ and let $B \in
X(\overline{K})$.  Define the Abel--Jacobi map with respect to $B$ by 
\[
\AJ_{B} \colon P \in X \mapsto [P - B] \in J.
\]

We say that $P \in X(\overline{K})$ is a torsion point of $X$ with respect to
$B$ if $\AJ_{B}(P) \in J(\overline{K})_{\tors}$, i.e., $k[P - B] = 0$ for some
integer $k \ge 1$. Denote by $T_{B}(X)$ the set of torsion points with respect
to $B$.
\end{definition}

We recall Raynaud's theorem (formerly the Manin-Mumford conjecture).
\begin{theorem}[Raynaud]
Suppose that the genus of $X$ is at least 2. Then for each $B \in
X(\overline{K})$, $T_{B}(X)$ is finite.
\end{theorem}

In this paper, we will study torsion points on superelliptic curves.

\begin{definition}
Let $n, d \ge 2$ be coprime integers and let $K$ be a field such that
$\characteristic(K) \nmid n$. Let $f(x) \in K[x]$ be separable with
$\deg(f) = d$ and suppose that
\[
f(x) = (x + \alpha_{1}) \dots (x + \alpha_{d})
\]
where $\alpha_{1}, \dots, \alpha_{d} \in \overline{K}$.  Let $\mathcal{C}$ be
the smooth projective model of the affine plane $K$-curve given by the equation
\[
y^{n} = f(x).
\]
Then we call $\mathcal{C}$ a superelliptic curve. When $n = 2$, we call
$\mathcal{C}$ a hyperelliptic curve.

Since $n$ and $d$ are coprime, $\mathcal{C}$ has a unique point at infinity,
denoted by $\infty$. For $i \in [1, d]$, let $\mathcal{W}_{i} \colonequals
(-\alpha_{i}, 0) \in \mathcal{C}(\overline{K})$. Let $\mathcal{W} \colonequals
\{ \mathcal{W}_{1}, \dots, \mathcal{W}_{d} \}$.  Use $g$ to denote the genus of
$\mathcal{C}$. By the Riemann-Hurwitz formula, 
\[
g = (n - 1)(d - 1) / 2.
\]

When we speak of torsion points on $\mathcal{C}$ (without specifying a
basepoint), we mean elements of $T_{\infty}(\mathcal{C})$.
\end{definition}

\begin{definition}
\label{Definition:1MinusZetaEndomorphism}
Suppose that $\mathcal{C}$ is the superelliptic curve $y^{n} = f(x)$ defined
over a field $K$ which contains a primitive $n$th root of unity $\zeta_{n}$. We
abuse notation and use $\zeta_{n}$ to denote the automorphism of $\mathcal{C}$
given by $(x, y) \mapsto (x, \zeta_{n} y)$. We also use $\zeta_{n}$ to denote
the induced automorphism on its jacobian $\mathcal{J}$. Then $1 - \zeta_{n}$ is
an endomorphism of $\mathcal{J}$. Let $Z_{n}$ be the subgroup of
$\Aut(\mathcal{C})$ generated by $\zeta_{n}$. 
\end{definition}

It is easy to check that $\mathcal{J}[1 - \zeta_{n}] \subseteq \mathcal{J}[n]$,
so since $\{ \infty \} \cup \mathcal{W}$ are points of
$\mathcal{C}(\overline{K})$ fixed by $\zeta_{n}$, 
\[
\{ \infty \} \cup \mathcal{W} \subseteq T_{\infty}(\mathcal{C}).
\]

\begin{definition}
Let $n, d \ge 2$ be coprime. Define $\mathcal{C}_{n, d}$ to be the superelliptic
``Catalan curve'' with defining equation $y^{n} = x^{d} + 1$ over $\mathbf{C}$.
Let $\mathcal{J}_{n, d}$ be the jacobian of $\mathcal{C}_{n, d}$.

Note that $\mathcal{C}_{n, d}$ is a superelliptic curve in ``two ways'': writing
the defining equation as $y^{n} = x^{d} + 1$ expresses $\mathcal{C}_{n, d}$ as
a degree $n$ cyclic branched covering of $\mathbf{P}^{1}$, and writing it as
$x^{d} = y^{n} - 1$ expresses $\mathcal{C}_{n, d}$ as a degree $d$ cyclic
branched covering of $\mathbf{P}^{1}$. Consequently, we may define $\zeta_{n}$
and $\zeta_{d}$ to be automorphisms of $\mathcal{C}_{n, d}$ and $\mathcal{J}_{n,
d}$ that scale the $y$-coordinate by $\zeta_{n}$ and the $x$-coordinate by
$\zeta_{d}$ respectively. Let $Z$ be the subgroup of $\Aut(\mathcal{C}_{n, d})$
generated by $Z_{n}$ and $Z_{d}$.
\end{definition}

\begin{definition}
Let $m \ge 1$ be an integer. Define the Fermat curve $F_{m}$ over a field of
characteristic not dividing $m$ to be the smooth plane curve given by the
equation $X^{m} + Y^{m} + Z^{m} = 0$.
\end{definition}

Note that $\mathcal{C}_{n, d}$ is a quotient of $F_{n d}$.  In
\autoref{Section:TorsionPointsSuperellipticCatalan}, we determine
$T_{\infty}(\mathcal{C}_{n, d})$.

\begin{definition}
An element of $T_{\infty}(\mathcal{C}_{n, d})$ is an \textit{exceptional}
torsion point if it is not fixed by any element of $Z$.
\end{definition}

Our main result is the following classification.
\begin{restatable*}{theorem}{MainTheoremOfPaper}
\label{Theorem:MainTheoremOfPaper}
Suppose that $n, d$ are coprime integers with $n, d \ge 2$. 

\begin{enumerate}[label=\upshape(\arabic*),
ref={\autoref{Theorem:MainTheoremOfPaper}(\arabic*)}]
\item \label{Theorem:MainTheoremOfPaperC23}
If $(n, d) = (2, 3)$, then $\mathcal{C}_{2, 3}$ is an elliptic curve, so it has
infinitely many torsion points.

\item \label{Theorem:MainTheoremOfPaperC25}
If $(n, d) = (2, 5)$, then the set of exceptional torsion points of
$\mathcal{C}_{2, 5}$ is the $Z$-orbit of $(\sqrt[5]{4}, \sqrt{5})$. Each has exact order $(1 -
\zeta_{5})^{3}$; in particular, each is killed by $5$.

\item \label{Theorem:MainTheoremOfPaperC34}
If $(n, d) = (4, 3)$, then the set of exceptional torsion points
of $\mathcal{C}_{4, 3}$ is the $Z$-orbit of $(2, \sqrt{3})$. Each has exact order $(1 -
\zeta_{4})(1 - \zeta_{3})^{2}$; in particular, each is killed by $12$.

\item \label{Theorem:MainTheoremOfPaperCdn}
If $(n, d) \in \{ (3, 2), (5, 2), (3, 4) \}$, then $\mathcal{C}_{n, d} \simeq
\mathcal{C}_{d, n}$ via $(x, y) \in \mathcal{C}_{n, d} \mapsto (\zeta_{2 n} y,
\zeta_{2 d} x) \in \mathcal{C}_{d, n}$, so one of \autoref{Theorem:MainTheoremOfPaperC23},
\autoref{Theorem:MainTheoremOfPaperC25}, \autoref{Theorem:MainTheoremOfPaperC34} applies.

\item \label{Theorem:MainTheoremOfPaperNoETP}
Otherwise, $\mathcal{C}_{n, d}$ has no exceptional torsion points.

\end{enumerate}
\end{restatable*}

The case $(n, d) = (2, 5)$ was already handled in the last two pages of
\cite{coleman1986torsion}.  The case when $n = 2$ and $d \ge 7$ is prime was
already proven as Theorem 1.1 of \cite{grant2004cuspidal}.

Similar results are proven in \cites{coleman1986torsion,coleman1998cuspidal} for
$F_{m}$. These papers show that whenever $P$ and $Q$ are points of
$F_{m}(\mathbf{C})$ such that $P - Q$ is torsion and $P$ is a cusp (a point such
that one of its coordinates is zero), then $Q$ is also necessarily a cusp.  Our
result for $y^{n} = x^{d} + 1$ implies their result for $F_{n d}$.

The ideas in our proof of \autoref{Theorem:MainTheoremOfPaper} are quite
different from those used in \cite{coleman1986torsion}. The classification of
torsion points on $F_{m}$ in \cite{coleman1986torsion} uses Coleman integration,
while we exploit the Galois action on torsion points.  If $P$ is a torsion point
of $\mathcal{C}_{n, d}$, then so are all of its Galois conjugates. If the Galois
action is large enough, there will be many relations among these torsion points,
which will force low-degree maps to $\mathbf{P}^1$. Now we obtain consequences
from these low-degree maps using two geometric techniques: the
Castelnuovo--Severi inequality and Riemann's theorem on the sum of the
Weierstrass weights on a Riemann surface.  Eventually we reduce to checking
finitely many points on finitely many $\mathcal{C}_{n, d}$, which we complete
with the aid of a computer.

During the analysis, we work out explicitly the torsion field
$\mathbf{Q}(\mathcal{J}_{p, q}[p], \mu_{p q})$ when $p$ and $q$ are distinct
primes (see \autoref{Theorem:pAndqTorsionFields}). The key ingredient is an
understanding of the $p$-adic and $q$-adic valuation of certain Jacobi sums;
this analysis is performed in \cite{ArulJacobi}. There is related work by
J{\k{e}}drzejak: in \cites{jkedrzejak2014torsion,jkedrzejak2016note} he studies
$J(\mathbf{Q})_{\tors}$ for the Jacobian $J$ of the curve $y^q = x^p + a$, where
$a \in \mathbf{Z}$. 

Additionally, we classify the torsion points on the generic superelliptic curve
$y^n = (x - a_1) \cdots (x - a_d)$ over $\mathbf{Q}(a_1, \dots, a_d)$ in the
case of coprime $n, d \ge 2$. The key idea is to specialize to the curves $y^n =
x^d + 1$ and $y^n = x^d + x$ and use \autoref{Theorem:MainTheoremOfPaper}.  
\begin{restatable*}{theorem}{GenericResult}
\label{Theorem:GenericResult}
Suppose that $n, d \ge 2$ are coprime and satisfy $n + d \ge 7$. Let
$\mathscr{C}_{n}$ be the curve over $k \colonequals \mathbf{Q}(a_1, \ldots,
a_d)$ defined by the equation 
\[
y^n = \prod_{x = 1}^{d} (x - a_i).
\]
Suppose that $\mathscr{C}_{n}$ is embedded into its jacobian $\mathscr{J}_{n}$
using the unique point $\infty$ at infinity. Points fixed by $\zeta_{n}$ are
torsion points of order dividing $n$.
\begin{enumerate}[label=\upshape(\arabic*),
ref=\autoref{Theorem:GenericResult}(\arabic*)]
\item \label{Theorem:GenericResultd3}
If $d \ge 3$, there are no other torsion points defined over $\overline{k}$.

\item \label{Theorem:GenericResultd2n7}
If $d = 2$ and $n \neq 5$, the only other torsion points defined over
$\overline{k}$ are
\[
\left\{ \left(
\frac{a_1 + a_2}{2}, - \zeta_n^i \sqrt[n]{\left( \frac{a_1 - a_2}{2} \right)^2}
\right) : 0 \le i \le n - 1 \right\}.
\]

\item \label{Theorem:GenericResultd2n5}
If $d = 2$ and $n = 5$, the only other torsion points defined over
$\overline{k}$ are 
\begin{align*}
&\left\{ \left( \frac{a_1 + a_2}{2}, - \zeta_5^i \sqrt[5]{\left( \frac{a_1 -
a_2}{2} \right)^2} \right) : 0 \le i \le 4 \right\} \bigcup \\
&\left\{ \left(
\frac{\pm(a_2 - a_1) \sqrt{5} + (a_1 + a_2)}{2}, \zeta_5^i \sqrt[5]{(a_2 -
a_1)^2}
\right)  : 0 \le i \le 4 \right\}.
\end{align*}
\end{enumerate}
\end{restatable*}

\autoref{Theorem:GenericResult} generalizes Theorem 7.1 of
\cite{poonen2014most}, which handles the $n = 2$ case.

\section{Tools from geometry}
\label{Section:ToolsFromGeometry}

\subsection{The Castelnuovo--Severi Inequality}
\label{Subsection:SuperellipticCSConsequences}

\begin{proposition}[Castelnuovo--Severi inequality]
\label{Proposition:Castelnuovo--Severi}
Let $k$ be a perfect field.  Let $F$, $F_1$, $F_2$ be function fields of curves
over $k$, of genera $g$, $g_1$, $g_2$, respectively.  Suppose that $F_i
\subseteq F$ for $i=1,2$ and the compositum of $F_1$ and $F_2$ in $F$ equals
$F$.  Let $d_i=[F:F_i]$ for $i=1,2$.  Then
\[
g \le d_1 g_1 + d_2 g_2 + (d_1-1)(d_2-1).
\]
\end{proposition}
\begin{proof}
See Theorem 3.11.3 of \cite{stichtenoth2009algebraic}.
\end{proof}

\begin{corollary}
\label{Corollary:FollowFromCS}
Let $\mathcal{C}$ be a superelliptic curve given by an equation of the form
$y^{n} = f(x)$ and let $d \colonequals \deg f$. Suppose that $\mathcal{C}$ has a
degree $d_{1}$ map to a genus zero curve and a degree $d_{2}$ map to a genus
zero curve. If $d_{1}$ and $d_{2}$ are coprime, then 
\[
(n - 1)(d - 1) / 2 \le (d_{1} - 1)(d_{2} - 1).
\]
\end{corollary}
\begin{proof}
Let $F$ be the function field of $\mathcal{C}$. Each map gives an embedding of
the function field of a genus zero curve into $F$; let their images be $F_{1}$
and $F_{2}$. Since $[F : F_{i}] = d_{i}$ and the $d_{i}$ are coprime, the
compositum $F_{1} F_{2}$ is $F$. Since $g = (n - 1)(d - 1) / 2$, we are done by
applying \autoref{Proposition:Castelnuovo--Severi} in this situation with $g_{1}
= g_{2} = 0$.
\end{proof}

\begin{lemma}
\label{Lemma:NotHyperelliptic}
Let $\mathcal{C}$ be a superelliptic curve given by an equation of the form
$y^{n} = f(x)$ and let $d \colonequals \deg f$. If $n, d \ge 3$, then
$\mathcal{C}$ cannot have a $2$-to-$1$ map to a genus zero curve.
\end{lemma}
\begin{proof}
For contradiction, suppose that $\varphi$ is a map from $\mathcal{C}$ to a genus
0 curve. We also have the degree $d$ map $y : \mathcal{C} \to \mathbf{P}^{1}$
and the degree $n$ map $x : \mathcal{C} \to \mathbf{P}^{1}$. Since $n$ and $d$
are coprime, they cannot both be even.

Suppose that $n$ is odd. Applying \autoref{Corollary:FollowFromCS} with
$\varphi$ and the $x$-map yields $(n - 1)(d - 1) / 2 \le (2 - 1)(n - 1)$, which
implies $d \le 3$, so since $d \ge 3$ by assumption, $d = 3$. Now $d$ is odd, so
similarly, $n = 3$, contradicting the assumption that $n$ and $d$ are coprime.

The case $d$ is odd is similar.
\end{proof}

\subsection{Weights and gaps on compact Riemann surfaces}
\label{Subsection:WeightsGaps}
Let $X$ be a compact Riemann surface of genus $g$.

\begin{definition}
\label{Definition:WeierstrassMonoid}
For each point $P$ on $X$, define $\WM(P)$ to be the set of pole orders at $P$
of meromorphic functions on $X$ which are holomorphic on $X \setminus \{ P \}$.
Then $\WM(P)$ is a monoid, and it is called the Weierstrass monoid of $P$.
\end{definition}

\begin{lemma}
\label{Lemma:WhereTheGapsAre}
$\mathbf{Z}_{\ge 0} \setminus \WM(P)$ is a subset of $[1, 2 g - 1]$ of size
exactly $g$.
\end{lemma}
\begin{proof}
This is a straightforward consequence of the Riemann--Roch theorem.
\end{proof}

\begin{definition}
\label{Definition:WeierstrassWeightPoint}
For each point $P$ on $X$, \autoref{Lemma:WhereTheGapsAre} implies that there
are $g$ integers $1 \le k_{1} < k_{2} < \cdots < k_{g} \le 2 g - 1$ such that
$\mathbf{Z}_{\ge 0} \setminus \WM(P) = \{ k_{1}, k_{2}, \dots, k_{g} \}$. Each
$k_{i}$ is called a \textit{gap} of $P$, and elements of $\WM(P)$ are called the
\textit{nongaps} of $P$. Define
\[
\wt_{P}(\mathcal{L}) \colonequals \sum_{i = 1}^{g} (k_{i} - (i - 1)).
\]
to be the Weierstrass weight of $P$. A point $P$ on $X$ is called a Weierstrass
point if $\wt(P) \ge 1$.
\end{definition}

\begin{theorem}
\label{Theorem:SumWeierstrassWeights}
Suppose that $g \ge 1$. Then
\[
\sum_{P \in X} \wt(P) = g^{3} - g.
\]
In particular, $X$ only has finitely many Weierstrass points.
\end{theorem}
\begin{proof}
See equation (5.11.1) on page 88 of \cite{farkas1992riemann}.
\end{proof}

\subsection{The homology of superelliptic curves}
\label{Subsection:HomologyBackground}

In this section, we compute $H_{1}(\mathcal{C}, \mathbf{Z})$ using topology. We
will apply results in Section 3 of \cite{molin2019computing}.

Fix $B \in \mathcal{C}(\mathbf{C}) \setminus \left( \mathcal{W} \cup \{ \infty
\} \right)$. For each $i \in [1, d]$, choose a loop $\beta_{i}$ in
$\mathbf{P}^{1} \setminus \pi \left( \mathcal{W} \cup \{ \infty \} \right)$ that
starts and ends at $\pi(B)$ which goes around $-\alpha_{i}$ once and does not go
around $\infty$ or $-\alpha_{k}$ for any $k \neq i$.

Then $\pi_{1}(\mathbf{P}^{1} \setminus \pi\left( \mathcal{W} \cup \{ \infty \}
\right), \pi(B))$ is the free group generated by $\beta_{1}, \cdots, \beta_{d}$.
Take the subscripts of $\beta$ modulo $d$, so that $\beta_{i + d} \colonequals
\beta_{i}$.  By Galois theory of covering spaces, $\pi_{1}(\mathcal{C} \setminus
\left( \mathcal{W} \cup \{ \infty \} \right), B)$ is the kernel of the map \[
\nu \colon \pi_{1}(\mathbf{P}^{1} \setminus \pi\left( \mathcal{W} \cup \{
\infty \} \right), \pi(B)) \to \mathbf{Z} / n \mathbf{Z} \] which sends each
$\beta_{i}$ to $1 \pmod{n}$. By the van Kampen theorem, $\pi_{1}(\mathcal{C},
B)$ is a quotient of $\pi_{1}(\mathcal{C} \setminus \left( \mathcal{W} \cup \{
\infty \} \right), B)$. In this way, we will view $\pi_{1}(\mathcal{C}, B)$ as a
subquotient of the free group generated by $\beta_{1}, \cdots, \beta_{d}$.
Recall that $H_{1}(\mathcal{C}, \mathbf{Z})$ is the abelianization of
$\pi_{1}(\mathcal{C}, B)$, so for each $\beta \in \pi_{1}(\mathcal{C}, B)$ we
will use $[\beta]$ to denote the class of $\beta$ in $H_{1}(\mathcal{C},
\mathbf{Z})$.

\begin{definition}
\label{Definition:XiLoops}
For each $i \in [1, d]$, $\beta_{i} \beta_{i + 1}^{-1}$ lies in $\ker \nu$, so
for each $j \in [0, n - 1]$, define $\psi_{i, j} \colonequals \zeta_{n}^{j}
[\beta_{i} \beta_{i + 1}^{-1}]$. Define $\Psi \colonequals \{ \psi_{i, j} \colon
i \in [1, d - 1] \text{ and } j \in [0, n - 2] \}$.
\end{definition}

\begin{lemma}
\label{Lemma:ZdActsCorrectlyOnLoops}
For each $j \in [0, n - 1]$, 
\[
\psi_{1, j} + \psi_{2, j} + \cdots + \psi_{d, j} = 0.
\]
\end{lemma}
\begin{proof}
Observe that $(\beta_{1} \beta_{2}^{-1}) (\beta_{2} \beta_{3}^{-1}) \cdots
(\beta_{d} \beta_{1}^{-1}) = 1$, so taking its image in $H_{1}(\mathcal{C},
\mathbf{Z})$ yields $\psi_{1, 0} + \psi_{2, 0} + \cdots + \psi_{d, 0} = 0$.
Apply $\zeta_{n}^{j}$ to both sides to finish.
\end{proof}

\begin{lemma}
\label{Lemma:ZnActsCorrectlyOnLoops}
For each $i \in [1, d]$, 
\[
\psi_{i, 0} + \psi_{i, 1} + \cdots + \psi_{i, n - 1} = 0.
\]
\end{lemma}
\begin{proof}
This is shown in the proof of Theorem 3.6 of \cite{molin2019computing}. Briefly,
the idea is that there exists a path $p_{i}$ from $\mathcal{W}_{i}$ to
$\mathcal{W}_{i + 1}$ in $\mathcal{C}$ and some $l \in \mathbf{Z} / n
\mathbf{Z}$ such that the cycle $\psi_{i, 0}$ is homotopic to $(\zeta_{n}^{l}
p_{i})(\zeta_{n}^{-(l + 1)} p_{i})^{-1}$, so the sum $\psi_{i, 0} + \psi_{i, 1}
+ \cdots + \psi_{i, n - 1}$ telescopes to give zero.
\end{proof}

\begin{proposition}
\label{Proposition:Cite36Molin}
The inclusion $\Psi \subseteq H_{1}(\mathcal{C}, \mathbf{Z})$ induces an
isomorphism
\[
\mathbf{Z}^{\Psi} \to H_{1}(\mathcal{C}, \mathbf{Z}).
\]
\end{proposition}
\begin{proof}
This is Theorem 3.6 of \cite{molin2019computing}.
\end{proof}

\begin{definition}
Let $S \colonequals \mathbf{Z}[T] / (1 + T + \dots + T^{n - 1})$.
\end{definition}

\begin{corollary}
\label{Corollary:FreeSModule}
$H_{1}(\mathcal{C}, \mathbf{Z})$ is a free $S$-module of rank $d -
1$ for which $T$ acts as $\zeta_{n}$. 
\end{corollary}
\begin{proof}
\autoref{Lemma:ZnActsCorrectlyOnLoops} implies that $1 + \zeta_{n} + \dots +
\zeta_{n}^{n - 1}$ acts trivially on all the $\psi_{i, j}$, and since these
generate $H_{1}(\mathcal{C}, \mathbf{Z})$ by \autoref{Proposition:Cite36Molin},
$H_{1}(\mathcal{C}, \mathbf{Z})$ is an $S$-module for which $T$ acts as
$\zeta_{n}$.  From \autoref{Proposition:Cite36Molin}, $\psi_{1, 0}$, \dots,
$\psi_{d - 1, 0}$ is an $S$-module basis for $H_{1}(\mathcal{C}, \mathbf{Z})$.
\end{proof}

Suppose now that $n = p$ is a prime. Then $S \simeq \mathbf{Z}[\zeta_{p}]$. 
\begin{definition}
Define the Tate module $T_{p} \mathcal{J} \simeq \varprojlim_{i}
\mathcal{J}[p^{i}]$. Since $ T_{p} \mathcal{J} \simeq H_{1}(\mathcal{C},
\mathbf{Z}) \otimes_{\mathbf{Z}} \mathbf{Z}_{p}$,
\autoref{Corollary:FreeSModule} gives that $ T_{p} \mathcal{J} $ is a free
$\mathbf{Z}_{p}[\zeta_{p}]$-module of rank $d - 1$. Define
$\End_{\mathbf{Z}_{p}[\zeta_{p}]} \left( T_{p} \mathcal{J} \right)$ to be the
ring of endomorphisms of $T_{p} \mathcal{J}$ that commute with $\zeta_{p}$. Then
\[
\End_{\mathbf{Z}_{p}[\zeta_{p}]} \left( T_{p} \mathcal{J} \right) \simeq M_{d -
1} \left( \mathbf{Z}_{p}[\zeta_{p}] \right).
\]
The relation $(1 - \zeta_{p})^{p - 1} \in p S^{\times}$ implies $\mathcal{J}[p]
= \mathcal{J}[(1 - \zeta_{p})^{p - 1}]$, so we also make the identifications
\[
T_{p} \mathcal{J} \simeq \varprojlim_{i} \mathcal{J}[(1 - \zeta_{p})^{i}]
\]
and
\[
\mathcal{J}[(1 - \zeta_{p})^{i}] \simeq T_{p} \mathcal{J} / (1 - \zeta_{p})^{i}.
\]
\end{definition}

\begin{lemma}
\label{Lemma:HowToKillpTorsionGeneralSuperelliptic}
Suppose $\eta \in \End_{\mathbf{Z}_{p}[\zeta_{p}]} \left( T_{p} \mathcal{J}
\right)$. Then $\eta$ kills $\mathcal{J}[(1 - \zeta_{p})^{i}]$ if and only if
$\eta \in (1 - \zeta_{p})^{i} \End_{\mathbf{Z}_{p}[\zeta_{p}]} \left( T_{p}
\mathcal{J} \right)$.
\end{lemma}
\begin{proof}
Note that $\eta$ kills $\mathcal{J}[(1 - \zeta_{p})^{i}]$ if and only if it lies
in the kernel of the reduction map
\[
\End_{\mathbf{Z}_{p}[\zeta_{p}]} \left( T_{p} \mathcal{J} \right)
\longrightarrow \End_{\mathbf{Z}_{p}[\zeta_{p}]} \left( T_{p} \mathcal{J} / (1 -
\zeta_{p})^{i} \right) = \End_{\mathbf{Z}_{p}[\zeta_{p}]} \left( \mathcal{J}[(1
- \zeta_{p})^{i}] \right).
\]
Since $\End_{\mathbf{Z}_{p}[\zeta_{p}]} \left( T_{p} \mathcal{J} \right) \simeq
M_{d - 1}(\mathbf{Z}_{p}[\zeta_{p}])$ and $\End_{\mathbf{Z}_{p}[\zeta_{p}]}
\left( T_{p} \mathcal{J} / (1 - \zeta_{p})^{i} \right) \simeq M_{d -
1}(\mathbf{Z}_{p}[\zeta_{p}] / (1 - \zeta_{p})^{i})$, the kernel of the
reduction map is $(1 - \zeta_{p})^{i} \End_{\mathbf{Z}_{p}[\zeta_{p}]} \left(
T_{p} \mathcal{J} \right)$, so we are done.
\end{proof}

\begin{definition}
\label{Definition:Thetap}
Define
\[
\theta_{p} \colon \mathbf{Z}_{p}\left[ \Gal\left( \mathbf{Q}(\mu_{p},
\mathcal{J}[p^{\infty}]) / \mathbf{Q}(\mu_{p}) \right) \right] \to
\End_{\mathbf{Z}_{p}[\zeta_{p}]} \left( T_{p} \mathcal{J} \right)
\]
to be the map which sends $\gamma \in \mathbf{Z}_{p}\left[ \Gal\left(
\mathbf{Q}(\mu_{p}, \mathcal{J}[p^{\infty}]) / \mathbf{Q}(\mu_{p}) \right)
\right]$ to its action on $T_{p} \mathcal{J}$.
\end{definition}

\begin{corollary}
\label{Corollary:KillTorsionInGeneral}
An element $\epsilon \in \mathbf{Z}_{p}\left[ \Gal\left( \mathbf{Q}(\mu_{p},
\mathcal{J}[p^{\infty}]) / \mathbf{Q}(\mu_{p}) \right) \right]$ kills
$\mathcal{J}[(1 - \zeta_{p})^{i}]$ if and only if
\[
\theta_{p}(\epsilon) \in (1 - \zeta_{p})^{i} \End_{\mathbf{Z}_{p}[\zeta_{p}]}
\left( T_{p} \mathcal{J} \right).
\]
\end{corollary}
\begin{proof}
This follows from the definition of $\theta_{p}$ and 
\autoref{Lemma:HowToKillpTorsionGeneralSuperelliptic}.
\end{proof}

\begin{lemma}
\label{Lemma:CreateKillerElements}
Let $i \ge 0$ be an integer and $\gamma \in \mathbf{Z}_{p}\left[
\Gal\left( \mathbf{Q}(\mu_{p}, \mathcal{J}[p^{\infty}]) / \mathbf{Q}(\mu_{p})
\right) \right]$. Suppose that $\gamma - 1$ kills $\mathcal{J}[(1 -
\zeta_{p})^{i}]$. 
\begin{enumerate}[label=\upshape(\arabic*),
ref=\autoref{Lemma:CreateKillerElements}(\arabic*)]
\item \label{Lemma:CreateKillerElementsPower}
For any integer $k \ge 0$, $(\gamma - 1)^{k}$ kills $\mathcal{J}[(1 -
\zeta_{p})^{i k}]$.

\item \label{Lemma:CreateKillerElementsGeoSum}
$\gamma^{p - 1} + \gamma^{p - 2} + \dots + 1$ kills $\mathcal{J}[p] =
\mathcal{J}[(1 - \zeta_{p})^{p - 1}]$.

\item \label{Lemma:CreateKillerElementsPpower}
$\gamma^{p} - 1$ kills $\mathcal{J}[(1 - \zeta_{p})^{p - 1 + i}]$.
\end{enumerate}
\end{lemma}

\begin{proof}
Define $\varepsilon \colonequals \gamma - 1$ and $\eta \colonequals
\theta_{p}(\varepsilon)$. Then $\varepsilon$ kills $\mathcal{J}[(1 -
\zeta_{p})^{i}]$, so \autoref{Corollary:KillTorsionInGeneral} implies that $\eta
\in (1 - \zeta_{p})^{i} \End_{\mathbf{Z}_{p}[\zeta_{p}]} \left( T_{p}
\mathcal{J} \right)$.
\begin{enumerate}[label=\upshape(\arabic*),
ref={the proof of \autoref{Lemma:CreateKillerElements}(\arabic*)}]
\item \label{LemmaProof:CreateKillerElementsPower}
Then $\eta^{k} \in (1 - \zeta_{p})^{i k} \End_{\mathbf{Z}_{p}[\zeta_{p}]} \left(
T_{p} \mathcal{J} \right)$, so we are done by
\autoref{Corollary:KillTorsionInGeneral}.

\item \label{LemmaProof:CreateKillerElementsGeoSum}
Using
\[
\gamma^{p - 1} + \gamma^{p - 2} + \dots + 1 \in (\gamma - 1)^{p - 1} + p
\mathbf{Z}[\gamma]
\]
yields
\begin{equation}
\label{Equation:GeometricSeriesGamma}
\theta_{p}(\gamma^{p - 1} + \gamma^{p - 2} + \dots + 1) \in p
\End_{\mathbf{Z}_{p}[\zeta_{p}]} \left( T_{p} \mathcal{J} \right) = (1 -
\zeta_{p})^{p - 1} \End_{\mathbf{Z}_{p}[\zeta_{p}]} \left( T_{p} \mathcal{J}
\right),
\end{equation}
and we are done by \autoref{Corollary:KillTorsionInGeneral}.

\item \label{LemmaProof:CreateKillerElementsPpower}
Multiplying both sides of \eqref{Equation:GeometricSeriesGamma} by
$\theta_{p}(\gamma - 1)$ yields 
\[
\theta_{p}(\gamma^{p} - 1) \in (1 - \zeta_{p})^{p - 1 + i}
\End_{\mathbf{Z}_{p}[\zeta_{p}]} \left( T_{p} \mathcal{J} \right),
\]
so we are done by \autoref{Corollary:KillTorsionInGeneral}. \qedhere
\end{enumerate}
\end{proof}

\begin{corollary}
\label{Corollary:ExponentOfTorsionFieldIsP}
For any $i \ge 1$, $\Gal(\mathbf{Q}(\mu_{p}, \mathcal{J}[(1 - \zeta_{p})^{i(p -
1) + 1}]) / \mathbf{Q}(\mu_{p}, \mathcal{J}[1 - \zeta_{p}]))$ is a group whose
exponent divides $p^{i}$.
\end{corollary}
\begin{proof}
Suppose that $\gamma \in \Gal(\mathbf{Q}(\mu_{p}, \mathcal{J}[p^{\infty}]) /
\mathbf{Q}(\mu_{p}, \mathcal{J}[1 - \zeta_{p}]))$. By assumption, $\gamma - 1$
kills $\mathcal{J}[1 - \zeta_{p}]$, so by induction with
\autoref{Lemma:CreateKillerElementsPpower}, $\gamma^{p^{i}} - 1$ kills
$\mathcal{J}[(1 - \zeta_{p})^{i(p - 1) + 1}]$, which means that $\gamma^{p^{i}}$
acts as the identity on $\mathcal{J}[(1 - \zeta_{p})^{i(p - 1) + 1}]$.
\end{proof}

\subsection{The homology of \texorpdfstring{$\mathcal{C}_{n, d}$}{C\_\{n,d\}}}
\label{Section:HomologySuperellipticCatalan}
\begin{definition}
Let $R$ be the ring
\[
R \colonequals \mathbf{Z}[T] / \left( 1 + T^{n} + T^{2 n} + \dots + T^{(d - 1)
n}, 1 + T^{d} + T^{2 d} + \dots + T^{(n - 1) d} \right).
\]
Define 
\[
\varphi_{n, d}(T) \colonequals \frac{(T^{n d} - 1)(T - 1)}{(T^{n} - 1)(T^{d} -
1)} \in \mathbf{Z}[T].
\]
\end{definition}

\begin{lemma}
\label{Lemma:IabEuclidean}
For integers $a, b \ge 1$, define the ideal 
\[
I_{a, b} \colonequals \left( 1 + T + \dots + T^{a - 1}, 1 + T + \dots + T^{b -
1} \right)
\]
of $\mathbf{Z}[T]$. Then $I_{a, b}$ is generated by $1 + T + \dots +
T^{\gcd(a, b) - 1}$.
\end{lemma}
\begin{proof}
If $a \ge b$, then 
\[
1 + T + \dots + T^{a - b - 1} = (1 + T + \dots + T^{a - 1}) - T^{a - b} (1 + T +
\dots + T^{b - 1})
\]
implies $I_{a, b} = I_{a - b, b}$, so by the Euclidean algorithm, $I_{a, b} =
I_{\gcd(a, b), 0}$.
\end{proof}

\begin{corollary} 
\label{Corollary:RIsom} \hfill
\begin{enumerate}[label=\upshape(\arabic*),
ref=\autoref{Corollary:RIsom}(\arabic*)]

\item \label{Corollary:RIsomphiND}
$R \simeq \mathbf{Z}[T] / (\varphi_{n, d}(T))$.

\item \label{Corollary:RIsomBasis}
$\{ T^{d a + n b} \colon a \in [0, n - 2] \text{ and } b \in [0, d - 2] \}$ is a
$\mathbf{Z}$-basis of $R$.

\end{enumerate}
\end{corollary}
\begin{proof}\hfill
\begin{enumerate}[label=\upshape(\arabic*),
ref={proof of \autoref{Corollary:RIsom}(\arabic*)}]

\item \label{CorollaryProof:RIsomphiND}
\autoref{Lemma:IabEuclidean} implies $I_{n, d}$ is the unit ideal, so
\begin{align*}
(\varphi_{n, d}(T)) &= \varphi_{n, d}(T) \left( I_{n, d} \right) \\
&= \frac{(T^{n d} - 1)(T - 1)}{(T^{n} - 1)(T^{d} - 1)} \left(
\frac{T^{n} - 1}{T - 1}, \frac{T^{d} - 1}{T - 1} \right) \\
&= \left( \frac{T^{n d} - 1}{T^{d} - 1}, \frac{T^{n d} - 1}{T^{n} - 1} \right)
\\
&= \left(  1 + T^{d} + T^{2 d} + \dots + T^{(n - 1) d}, 1 + T^{n} + T^{2 n} +
\dots + T^{(d - 1) n} \right),
\end{align*}
so applying this to the definition of $R$ yields $ R \simeq \mathbf{Z}[T] /
(\varphi_{n, d}(T))$.

\item \label{CorollaryProof:RIsomBasis}
For nonnegative integers $u, v$, define $B_{u, v} \colonequals \{ T^{d a + n b}
\colon a \in [0, u] \text{ and } j \in [0, v] \} $. Since $T^{n d} = 1$ in $R$
and the set $\{ d a + n b : a \in [0, n - 1] \text{ and } b \in [0, d - 1] \}$
contains every residue class modulo $n d$, $B_{n - 1, d - 1}$ must generate $R$
as a $\mathbf{Z}$-module. Using $1 + T^{n} + T^{2 n} + \dots + T^{(d - 1) n} =
0$ and $1 + T^{d} + T^{2 d} + \dots + T^{(n - 1) d} = 0$ shows that $B_{n - 2, d
- 2}$ generates $R$ as a $\mathbf{Z}$-module.  \autoref{Corollary:RIsomphiND}
implies that $R$ is a free $\mathbf{Z}$-module of rank $\deg \varphi_{n, d} = (n
- 1)(d - 1) = \# S_{n - 2, d - 2}$, so $S_{n - 2, d - 2}$ must be a basis.
\qedhere
\end{enumerate}
\end{proof}

\begin{proposition}
\label{Proposition:H1CndFreeRModule}
$H_{1}(\mathcal{C}_{n, d}, \mathbf{Z})$ is a free $R$-module of
rank 1 for which $T$ acts by $\zeta_{n d}$. 
\end{proposition}
\begin{proof}

We apply the results of \autoref{Subsection:HomologyBackground}. Our basepoint
will be $B \colonequals (0, 1)$. Choose $\alpha_{1} \in \mathbf{C}$ to be a root
of $(-x)^{d} + 1 = 0$, and define $\alpha_{i} \colonequals \zeta_{d}^{i - 1}
\alpha_{1}$.  Let $\beta_{1}$ be a loop in $\mathbf{P}^{1} \setminus \{
-\alpha_{1}, \dots, -\alpha_{d}, \infty \}$ starting and ending at $0$ which
encircles $-\alpha_{1}$ positively and does not encircle any of $\{ -\alpha_{2},
\dots, -\alpha_{d}, \infty \}$. Define $\beta_{i} \colonequals \zeta_{d}^{i - 1}
\beta_{1}$. As in \autoref{Definition:XiLoops}, for $i \in [1, d]$ and $j
\in [0, n - 1]$, define $\psi_{i, j} \in H_{1}(\mathcal{C}_{n, d}, \mathbf{Z})$
to be the cycle $\zeta_{n}^{j} [\beta_{i} \beta_{i + 1}^{-1}]$.  Therefore,
\[
\psi_{i, j} =  \zeta_{n}^{j} \zeta_{d}^{i - 1} \psi_{1, 0},
\]
so \autoref{Proposition:Cite36Molin} implies that $\{ \zeta_{n}^{j} \zeta_{d}^{i
- 1} \psi_{1, 0} \colon i \in [1, d - 1] \text{ and } j \in [0, n - 2] \}$ is a
$\mathbf{Z}$-basis for $H_{1}(\mathcal{C}_{n, d}, \mathbf{Z})$.
\autoref{Lemma:ZdActsCorrectlyOnLoops} and
\autoref{Lemma:ZnActsCorrectlyOnLoops} give the relations
\begin{align*}
(1 + \zeta_{n} + \dots + \zeta_{n}^{n - 1}) \gamma_{1, 0} &= 0, \\
(1 + \zeta_{d} + \dots + \zeta_{d}^{d - 1}) \gamma_{1, 0} &= 0,
\end{align*}
so there is an $R$-module map $R \to H_{1}(\mathcal{C}_{n, d}, \mathbf{Z})$
sending $1$ to $\gamma_{1, 0}$. The $\mathbf{Z}$-basis of $R$ given in
\autoref{Corollary:RIsomBasis} gets mapped to the $\mathbf{Z}$-basis $\{
\zeta_{n}^{a} \zeta_{d}^{b} \gamma_{1, 0} \colon a \in [0, d - 2] \text{
and } b \in [0, n - 2] \}$ of $H_{1}(\mathcal{C}_{n, d}, \mathbf{Z})$, so this
is an isomorphism.
\end{proof}

\begin{corollary}
\label{Corollary:SubringByXi}
The map $\mathbf{Z}[T] \to \End \mathcal{J}_{n, d}$ sending $T$ to $\zeta_{n d}$
has kernel $(\varphi_{n, d}(T))$.
\end{corollary}
\begin{proof}
By \autoref{Proposition:H1CndFreeRModule}, $\mathbf{Z}[T] \to \End
\mathcal{J}_{n, d} \to \End H_{1}(\mathcal{J}_{n, d}, \mathbf{Z}) = \End
H_{1}(\mathcal{C}_{n, d}, \mathbf{Z})$ has kernel equal to $(\varphi_{n,
d}(T))$. Since the map $\End \mathcal{J}_{n, d} \to \End H_{1}(\mathcal{J}_{n,
d}, \mathbf{Z})$ is injective, we are done.
\end{proof}

In view of \autoref{Corollary:SubringByXi}, we will view $R$ as the
subring of $\End \mathcal{J}_{n, d}$ generated by $Z$.

\section{The \texorpdfstring{$1 - \zeta_{n}$}{1 - zeta\_n} endomorphism and
\texorpdfstring{$(1 - \zeta_{n})$}{(1 - zeta\_n)}-descent}
\label{Section:1MinusZeta}

Suppose that $K$ contains a primitive $n$th root of unity $\zeta_{n}$.  Define
$\zeta_{n} \in \Aut(\mathcal{C})$, $\zeta_{n} \in \Aut(\mathcal{J})$, and $1 -
\zeta_{n} \in \End(J)$ as in \autoref{Definition:1MinusZetaEndomorphism}.  We
now state a few properties about the $1 - \zeta_{n}$ endomorphism. We adapt the
main results of Sections 6.1, 6.2, and 6.3 of \cite{poonen2006lectures}, which
states everything in the hyperelliptic case.  However, the extension to
superelliptic curves is straightforward and we omit aspects of the proofs that
generalize immediately. The case when $n$ is prime is considered in
\cites{poonen1997explicit,schaefer1998computing}.

We use a few more definitions.
\begin{align*}
\overline{K} &\colonequals\text{a separable closure of }K \\
\overline{\mathcal{C}} &\colonequals \mathcal{C} \times_{K} \overline{K} \\
G &\colonequals \Gal(\overline{K} / K) \\
\pi &\colonequals \text{the }x\text{-coordinate map }\mathcal{C} \to
\mathbf{P}^1 \\
\mathcal{W}_i &\colonequals (-\alpha_i, 0) \in \mathcal{C}(\overline{K}) \\
\mathcal{W} &\colonequals \{ \mathcal{W}_1, \ldots, \mathcal{W}_d \} \\
(\mathbf{Z} / n \mathbf{Z})^{\mathcal{W}} &\colonequals \text{the free
}\mathbf{Z} / n \mathbf{Z}\text{-module with basis }\mathcal{W}_1, \ldots,
\mathcal{W}_d.
\end{align*}

Observe that $\mathcal{W} \cup \{ \infty \}$ is the set of ramification points
of $\pi$ over $\overline{K}$ and that $\mathcal{W}$ is a $G$-module.

\begin{proposition} 
\label{Proposition:SuperellipticProp611Poonen}
There is a split exact sequence of $G$-modules
\begin{center}
\begin{tikzpicture}
\node(0m) at (0, 0) {$0$};
\node(1m) at (3, 0) {$\mathbf{Z} / n \mathbf{Z}$};
\node(2m) at (6, 0) {$(\mathbf{Z} / n \mathbf{Z})^{\mathcal{W}}$};
\node(3m) at (9, 0) {$\mathcal{J}[1 - \zeta_{n}]$};
\node(4m) at (12, 0) {$0$};
\draw[->] (0m) -- (1m);
\draw[->] (1m) -- (2m) node[midway, above] {$\Delta$};
\draw[->] (2m) -- (3m) node[midway, above] {$s$};
\draw[->] (3m) -- (4m);
\end{tikzpicture}
\end{center}
where
\begin{align*}
\Delta(1) &= (1,\; \dots,\; 1) \\
s(a_1, \dots, a_d) &= \sum_{i = 1}^d a_i [\mathcal{W}_i - \infty].
\end{align*}
\end{proposition}

\begin{proof} (c.f.~\cite{poonen2006lectures}, Proposition 6.1.1)

\begin{enumerate}[label=\textbf{Step \arabic*:}, ref={Step \arabic*},
leftmargin=*, itemindent=25pt]
\item \label{Step:S611PsIsWellDefined}
$s$\textit{ is well-defined.}

Each point in $\mathcal{W} \cup \{ \infty \}$ is fixed by $\zeta_{n}$, so
$[\mathcal{W}_i - \infty] \in \mathcal{J}[1 - \zeta_{n}]$. The calculation
\begin{equation}
\label{Equation:WeierstrassPointsNTorsion}
\divisor(x + \alpha_i) = n \mathcal{W}_i - n \infty
\end{equation}
shows that the divisor classes $[\mathcal{W}_i - \infty]$ are $n$-torsion. 

\item \label{Step:S611PGModHoms}
$\Delta$\textit{ and }$s$\textit{ are }$G$\textit{-module homomorphisms.}

This is clear.

\item \label{Step:S611PsDeltaZero}
$s \circ \Delta = 0$\textit{.}

This follows from $\divisor(y) = \displaystyle\sum_{i = 1}^{d} [\mathcal{W}_i -
\infty]$.

\item \label{Step:S611Pkers}
$\ker(s)$\textit{ is generated by }$(1, \dots, 1)$\textit{.}

\item \label{Step:S611Pcokers}
$s$\textit{ is surjective.}

We modify the proof of Proposition {3.2} in \cite{schaefer1998computing} to
prove \autoref{Step:S611Pkers} and \autoref{Step:S611Pcokers} simultaneously.
Use $\Div^{0}$ to denote the degree-zero divisors on $\mathcal{C}$ and use
$\Princ$ to denote the subgroup of principal divisors.  The following are exact
sequences of $\mathbf{Z}[\zeta_{n}]$-modules.

\begin{center}
\begin{tikzpicture}
\node(mm) at (-3, 1) {$0$};
\node(0m) at (0, 1) {$\overline{K}^\times$};
\node(1m) at (3, 1) {$\overline{K}(\mathcal{C})^\times$};
\node(2m) at (6, 1) {$\Princ$};
\node(3m) at (9, 1) {$0$;};
\node(mp) at (-3, 0) {$0$};
\node(0p) at (0, 0) {$\Princ$};
\node(1p) at (3, 0) {$\Div^0$};
\node(2p) at (6, 0) {$\mathcal{J}$};
\node(3p) at (9, 0) {$0$.};
\draw[->] (mm) -- (0m);
\draw[->] (0m) -- (1m);
\draw[->] (1m) -- (2m);
\draw[->] (2m) -- (3m);
\draw[->] (mp) -- (0p);
\draw[->] (0p) -- (1p);
\draw[->] (1p) -- (2p);
\draw[->] (2p) -- (3p);
\end{tikzpicture}
\end{center}
We now apply group cohomology with the group $Z_{n} = \langle \zeta_{n}
\rangle$. 

\begin{enumerate}[label=\upshape{(\roman*)}, ref={\theenumi(\roman*)}]
\item \label{Step:S611PcokersGalCohoHilb90}
Since $Z_{n} \simeq \Gal(\overline{K}(\mathcal{C}) / \overline{K}(x))$,
\begin{equation}
\label{Equation:H0KC}
H^0(\overline{K}(\mathcal{C})^\times) = \overline{K}(x)^\times
\end{equation}
and Hilbert's Theorem 90 yields
\begin{equation}
\label{Equation:H1KC}
H^1(\overline{K}(\mathcal{C})^\times) = 0.
\end{equation}

\item \label{Step:S611PcokersGalCohoTrivAction}
Since $\overline{K}^\times$ is a trivial $Z_{n}$-module, 
\begin{align}
H^0(\overline{K}^\times) &= 0 \label{Equation:H0K} \\
H^1(\overline{K}^\times) &= \mu_{n}(\overline{K}) \label{Equation:HoddK} \\
H^2(\overline{K}^\times) &= \overline{K}^\times / \overline{K}^{\times n} = 0.
\label{Equation:H2K}
\end{align}
Substituting \eqref{Equation:H1KC} and \eqref{Equation:H2K} into
\begin{center}
\begin{tikzpicture}
\node(0m) at (0, 1) {$H^1(\overline{K}(\mathcal{C})^\times)$};
\node(1m) at (3, 1) {$H^1(\Princ)$};
\node(2m) at (6, 1) {$H^2(\overline{K}^\times)$.};
\draw[->] (0m) -- (1m);
\draw[->] (1m) -- (2m);
\end{tikzpicture}
\end{center}
yields
\begin{equation}
\label{Equation:H1Princ}
H^1(\Princ) = 0.
\end{equation}

\item \label{Step:S611PcokersGalCohoTrivH0Princ}
Substituting \eqref{Equation:H0K}, \eqref{Equation:H0KC},
\eqref{Equation:HoddK}, \eqref{Equation:H1KC} into 
\begin{center}
\begin{tikzpicture}
\node(0m) at (0, 1) {$H^0(\overline{K}^\times)$};
\node(1m) at (2.75, 1) {$H^0(\overline{K}(C)^\times)$};
\node(2m) at (5.5, 1) {$H^0(\Princ)$};
\node(3m) at (8.25, 1) {$H^1(\overline{K}^\times)$};
\node(4m) at (11, 1) {$H^1(\overline{K}(\mathcal{C})^\times)$};
\draw[->] (0m) -- (1m);
\draw[->] (1m) -- (2m);
\draw[->] (2m) -- (3m);
\draw[->] (3m) -- (4m);
\end{tikzpicture}
\end{center}
yields
\begin{center}
\begin{tikzpicture}
\node(0m) at (0, 1) {$0$};
\node(1m) at (3, 1) {$\overline{K}(x)^\times$};
\node(2m) at (6, 1) {$H^0(\Princ)$};
\node(3m) at (9, 1) {$\mu_{n}(\overline{K})$};
\node(4m) at (12, 1) {$0$,};
\draw[->] (0m) -- (1m);
\draw[->] (1m) -- (2m);
\draw[->] (2m) -- (3m);
\draw[->] (3m) -- (4m);
\end{tikzpicture}
\end{center}
so since the image of $\divisor(y) \in H^{0}(\Princ)$ generates
$\mu_{n}(\overline{K})$, 
\begin{equation}
H^0(\Princ) \text{ is generated by } \{ \divisor(y) \} \cup \{ \divisor(u)
\colon u \in \overline{K}(x)^\times  \}.
\label{Equation:H0Princ}
\end{equation}

\item \label{Step:S611PcokersGalCohoTrivJ1mz}
We substitute \eqref{Equation:H1Princ} into the long exact sequence
\begin{center}
\begin{tikzpicture}
\node(mm) at (-3, 1) {$0$};
\node(0m) at (0, 1) {$H^0(\Princ)$};
\node(1m) at (3, 1) {$H^0(\Div^0)$};
\node(2m) at (6, 1) {$\mathcal{J}[1 - \zeta_{n}]$};
\node(3m) at (9, 1) {$H^1(\Princ)$};
\draw[->] (mm) -- (0m);
\draw[->] (0m) -- (1m);
\draw[->] (1m) -- (2m);
\draw[->] (2m) -- (3m);
\end{tikzpicture}
\end{center}
to obtain
\begin{center}
\begin{tikzpicture}
\node(mm) at (-3, 1) {$0$};
\node(0m) at (0, 1) {$H^0(\Princ)$};
\node(1m) at (3, 1) {$H^0(\Div^0)$};
\node(2m) at (6, 1) {$\mathcal{J}[1 - \zeta_{n}]$};
\node(3m) at (9, 1) {$0$.};
\draw[->] (mm) -- (0m);
\draw[->] (0m) -- (1m);
\draw[->] (1m) -- (2m);
\draw[->] (2m) -- (3m);
\end{tikzpicture}
\end{center}
The group $H^0(\Div^0)$ consists of the $\zeta_{n}$-fixed divisors, so it is
generated by $[\mathcal{W}_i - \infty]$ and $\Norm(P - \infty)$ for arbitrary $P
\in \mathcal{C}(\overline{K})$. Observe that 
\begin{align*}
\divisor (x - x(P)) &= \Norm(P - \infty), \\
\divisor (y) &= \sum [\mathcal{W}_i - \infty],
\end{align*}
so by \eqref{Equation:H0Princ}, $H^0(\Princ)$ is generated by $\sum
[\mathcal{W}_i - \infty]$ and $\Norm(P - \infty)$ for arbitrary $P \in
\mathcal{C}(\overline{K})$. Therefore, the $[\mathcal{W}_i - \infty]$ generate
$\mathcal{J}[1 - \zeta_{n}] \simeq H^0(\Div^0) / H^0(\Princ)$ and the only relation
is $\sum [\mathcal{W}_i - \infty] = 0$.
\end{enumerate}

\item \label{Step:S611Psplits}
\textit{The exact sequence in the statement of
\autoref{Proposition:SuperellipticProp611Poonen} splits.}

The splitting is given by
\begin{center}
\begin{tikzpicture}
\node(0m) at (0, 1) {$(\mathbf{Z} / n \mathbf{Z})^{\mathcal{W}}$};
\node(1m) at (4, 1) {$\mathbf{Z} / n \mathbf{Z}$};
\node(0p) at (0, 0) {$(a_1, \ldots, a_d)$};
\node(1p) at (4, 0) {$d^{-1} \displaystyle\sum a_i$. \qedhere};
\draw[->] (0m) -- (1m);
\draw[->] (0p) -- (1p);
\end{tikzpicture}
\end{center}
\end{enumerate}
\end{proof}

\begin{corollary}
\label{Corollary:UniqueRepresentationJ1mZ}
Each element of $\mathcal{J}[1 - \zeta_{n}]$ has a unique representation of the form
\[
\sum_{i = 1}^{d} a_{i} [\mathcal{W}_{i} - \infty]
\]
for $a_{i} \in \mathbf{Z} / n \mathbf{Z}$ satisfying $a_{1} + \dots + a_{d}
\equiv 0 \pmod{n}$.
\end{corollary}
\begin{proof}
If $ \sum_{i = 1}^{d} a_{i} [\mathcal{W}_{i} - \infty] = 0$, then
\autoref{Proposition:SuperellipticProp611Poonen} implies that $a_{1} \equiv
\cdots \equiv a_{d} \pmod{n}$, so since $a_{1} + \dots + a_{d} = 0$ and $(n, d)
= 1$, we see that $a_{1} \equiv \cdots \equiv a_{d} \equiv 0 \pmod{n}$; hence
the representation is unique.

For existence, \autoref{Proposition:SuperellipticProp611Poonen} implies that
each element of $\mathcal{J}[1 - \zeta_{n}]$ has a representation of the form
$\sum_{i = 1}^{d} a_{i}' [\mathcal{W}_{i} - \infty]$ for $a_{i}' \in \mathbf{Z}
/ n \mathbf{Z}$, so if we let $a_{i} = a_{i}' - d^{-1} \left( a_{1}' + \dots +
a_{d}' \right)$, then $a_{1} + \dots + a_{d} \equiv 0 \pmod{n}$ and
\begin{align*}
\sum_{i = 1}^{d} a_{i} [\mathcal{W}_{i} - \infty] &= \sum_{i = 1}^{d} a_{i}'
[\mathcal{W}_{i} - \infty] - d^{-1} \left( a_{1}' + \dots + a_{d}' \right)
\sum_{i = 1}^{d} [\mathcal{W}_{i} - \infty]\\
&= \sum_{i = 1}^{d} a_{i}' [\mathcal{W}_{i} - \infty]
\end{align*}
since $\sum_{i = 1}^{d} [\mathcal{W}_{i} - \infty] = 0$.
\end{proof}

Define
\[
L \colonequals K[T] / (f(T)).
\]
and 
\[
\overline{L} \colonequals L \otimes_{K} \overline{K} \simeq \overline{K}[T] /
(f(T)) \simeq \prod \overline{K}[T] / (T + \alpha_i) \simeq
\overline{K}^{\mathcal{W}}.
\]
We have the natural norm homomorphism $\Norm \colon \overline{L} \to
\overline{K}$ which sends a tuple $(a_1, \dots, a_d) \in
\overline{K}^{\mathcal{W}}$ to the product $a_1 \cdots a_d \in \overline{K}$.
For any ring $R$, let $\mu_n(R) \colonequals \{ r \in R \colon r^n = 1 \}$.

\begin{proposition}
\label{Proposition:SuperellipticProp612Poonen}
There is a split exact sequence of $G$-modules
\begin{center}
\begin{tikzpicture}
\node(0m) at (0, 0) {$0$};
\node(1m) at (3, 0) {$\mathcal{J}[1 - \zeta_{n}]$};
\node(2m) at (6, 0) {$\mu_n(\overline{L})$};
\node(3m) at (9, 0) {$\mu_n(\overline{K})$};
\node(4m) at (12, 0) {$0$.};
\draw[->] (0m) -- (1m);
\draw[->] (1m) -- (2m);
\draw[->] (2m) -- (3m) node[midway, above] {$\Norm$};
\draw[->] (3m) -- (4m);
\end{tikzpicture}
\end{center}
\end{proposition}
\begin{proof}
This is a straightforward generalization of Proposition 6.1.2 of
\cite{poonen2006lectures}.
\end{proof}

\begin{proposition}
\label{Proposition:SuperellipticProp621Poonen}
We have
\begin{equation}
\label{Equation:SuperellipticProp621Poonen}
H^{1}(K, \mathcal{J}[1 - \zeta_{n}]) \simeq \ker \left( \frac{L^\times}{L^{\times
n}} \xrightarrow{\Norm} \frac{K^\times}{K^{\times n}} \right).
\end{equation}
\end{proposition}
\begin{proof}
(c.f. Proposition 6.2.1 of \cite{poonen2006lectures}).  Since the short exact
sequence in \autoref{Proposition:SuperellipticProp612Poonen} is split, it
induces short exact sequences after applying $H^{1}(K, -)$, so
\begin{equation}
\label{Equation:H1SESSplit}
H^{1}(K, \mathcal{J}[1 - \zeta_{n}]) \simeq \ker \left( H^{1}(K,
\mu_{n}(\overline{L})) \xrightarrow{\Norm} H^{1}(K, \mu_{n}(\overline{K}))
\right).
\end{equation}
Applying an extension of Hilbert's Theorem 90 (exercise 2 on page 152 of
\cite{serre2013local}) gives the identifications
\begin{align} 
H^{1}(K, \mu_{n}(\overline{L})) &\simeq L^\times / L^{\times n}
\label{Equation:H1KmunL}\\ 
H^{1}(K, \mu_{n}(\overline{K})) &\simeq K^\times / K^{\times n}
\label{Equation:H1KmunK},
\end{align}
so we are done by substituting \eqref{Equation:H1KmunL} and
\eqref{Equation:H1KmunK} into \eqref{Equation:H1SESSplit}.
\end{proof}

Consider the short exact sequence
\begin{center}
\begin{tikzpicture}
\node(0m) at (0, 0) {$0$};
\node(1m) at (3, 0) {$\mathcal{J}[1 - \zeta_{n}]$};
\node(2m) at (6, 0) {$\mathcal{J}$};
\node(3m) at (9, 0) {$\mathcal{J}$};
\node(4m) at (12, 0) {$0$.};
\draw[->] (0m) -- (1m);
\draw[->] (1m) -- (2m);
\draw[->] (2m) -- (3m) node[midway, above] {$1 - \zeta_{n}$};
\draw[->] (3m) -- (4m);
\end{tikzpicture}
\end{center}
The first coboundary map in Galois cohomology induces the following
injective homomorphism, which we denote by $\delta$.
\[
\frac{\mathcal{J}(K)}{(1 - \zeta_{n}) \mathcal{J}(K)}
\stackrel{\delta}\hookrightarrow H^{1}(K, \mathcal{\mathcal{J}}[1 - \zeta_{n}]).
\]
Composing with the isomorphism of \eqref{Equation:SuperellipticProp621Poonen},
we obtain an injective homomorphism
\begin{equation}
\label{Equation:SecretlyXminusT}
\frac{\mathcal{J}(K)}{(1 - \zeta_{n}) \mathcal{J}(K)} \hookrightarrow  \ker \left(
\frac{L^\times}{L^{\times n}} \xrightarrow{\Norm} \frac{K^\times}{K^{\times n}}
\right).
\end{equation}
\begin{theorem}\hfill
\label{Theorem:SuperellipticThm631Poonen}
\begin{enumerate}[label={\upshape{(\arabic*)}},
ref={\autoref{Theorem:SuperellipticThm631Poonen}(\arabic*)}]

\item 
\label{Theorem:SuperellipticThm631PoonenNotWeierstrass}
Suppose that $P = (x_P, y_P) \in \mathcal{C}(K)$ and that $y_P \neq 0$. 
The image of 
\[
[P - \infty] \in \frac{\mathcal{J}(K)}{(1 - \zeta_{n}) \mathcal{J}(K)}
\]
under the map \eqref{Equation:SecretlyXminusT} equals
\[
[x_P - T] \in \ker \left( \frac{L^\times}{L^{\times n}} \xrightarrow{\Norm}
\frac{K^\times}{K^{\times n}} \right).
\]

\item
\label{Theorem:SuperellipticThm631PoonenWeierstrass}
Suppose that $\mathcal{W}_{1}, \cdots, \mathcal{W}_{d}$ are defined over $K$.
The image of
\[
[\mathcal{W}_{j} - \infty] \in \frac{\mathcal{J}(K)}{(1 - \zeta_{n}) \mathcal{J}(K)}
\]
under the map \eqref{Equation:SecretlyXminusT} is 
\[
(-\alpha_{j} - T) + \prod_{i \neq j} (-\alpha_{i} - T)^{n - 1} \pmod{L^{\times
n}}.
\]

\end{enumerate}
\end{theorem}

\begin{proof}
This is a straightforward generalization of Proposition 3.3 of
\cite{schaefer1998computing}; see the computation on page 461 of
\cite{schaefer1998computing}.
\end{proof}

As is standard, we will call \eqref{Equation:SecretlyXminusT} the ``$x - T$''
descent map. 

\begin{lemma}
\label{Lemma:DivisionFieldOnePoint}
Suppose that $\alpha_{1}, \dots, \alpha_{d} \in K$ and that $P = (x_P, y_P) \in
\mathcal{C}(K)$. Then 
\[
[P - \infty] \in (1 - \zeta_{n}) \mathcal{J}(K)
\]
if and only if
\[
x_{P} + \alpha_{i} \in K^{n} \text{ for all } i \in [1, d].
\]
\end{lemma}
\begin{proof}
Since $\alpha_{1}, \dots, \alpha_{d} \in K$, we have an isomorphism
\begin{equation}
\label{Equation:LBreaksUpIntoKs}
L \simeq \prod_{i = 1}^{d} \frac{K[T]}{(T + \alpha_{i})} \simeq \prod_{i =
1}^{d} K,
\end{equation}
such that the image of $g(T) \in L$ is $\left( g(-\alpha_{1}), \dots,
g(-\alpha_{d}) \right)$.

Since the map of \eqref{Equation:SecretlyXminusT} is an embedding, $[P - \infty]
\in (1 - \zeta_{n}) \mathcal{J}(K)$ if and only if 
\begin{equation}
\label{Equation:ConditionToBeDivisible}
\text{the image of }[P -
\infty]\text{ in }\ker \left( \frac{L^\times}{L^{\times n}} \xrightarrow{\Norm}
\frac{K^\times}{K^{\times n}} \right)\text{ is trivial.}
\end{equation}

\begin{enumerate}[label=\textbf{Case~\Alph*:}, ref={Case~\Alph*},
leftmargin=*, itemindent=25pt]
\item \label{Case:DivisionFieldOnePointPnotinW}
$P \not\in \mathcal{W}$

\autoref{Theorem:SuperellipticThm631PoonenNotWeierstrass} implies
\eqref{Equation:ConditionToBeDivisible} is equivalent to $[x_{P} - T] \in
L^{\times n}$, which from \eqref{Equation:LBreaksUpIntoKs} is equivalent to
$x_{P} + \alpha_{i} \in K^{\times n}$ for $i \in [1, d]$, and since $x_{P}
\not\in \{ -\alpha_{i} : i \in [1, d]\}$, this is equivalent to $x_{P} +
\alpha_{i} \in K^{n}$ for $i \in [1, d]$.

\item \label{Case:DivisionFieldOnePointPWj}
$P = \mathcal{W}_{j}$

\autoref{Theorem:SuperellipticThm631PoonenWeierstrass} implies
\eqref{Equation:ConditionToBeDivisible} is equivalent to 
\[
(-\alpha_{j} - T) + \prod_{i \neq j} (-\alpha_{i} - T)^{n - 1} \in L^{\times n},
\]
which from \eqref{Equation:LBreaksUpIntoKs} is equivalent to the two conditions
\begin{enumerate}[label=(\alph*)]
\item 
\label{ConditionAwayFromWP}
$\alpha_{i} - \alpha_{j} \in K^{\times n}$ for all $i \in [1, d] \setminus \{ j
\}$

\item 
\label{ConditionAtWP}
$\prod_{i \neq j} (\alpha_{j} - \alpha_{i})^{n - 1} \in K^{\times n}$

\end{enumerate}

Note that \autoref{ConditionAwayFromWP} implies that $K^{\times n} \ni \prod_{i
\neq j} (\alpha_{i} - \alpha_{j})^{n - 1} = \prod_{i \neq j} (\alpha_{j} -
\alpha_{i})^{n - 1}$ since $(n - 1)(d - 1) = 2 g$ is even, so
\autoref{ConditionAwayFromWP} implies \autoref{ConditionAtWP}. In particular,
the two conditions together are equivalent to $\alpha_{i} - \alpha_{j} \in
K^{n}$ for all $i \in [1, d]$. \qedhere
\end{enumerate}
\end{proof}

\begin{corollary}
\label{Corollary:DivisionFieldForAnyPoint}
Suppose that $P = (x_P, y_P) \in \mathcal{C}(K)$. Let $K'$ be the field
\[
K' = K\left( [D] \in \mathcal{J}(\overline{K}) \colon (1 - \zeta_{n})[D] = [P -
\infty]  \right).
\]
Then
\[
K' = K(\sqrt[n]{x_P + \alpha_i} \colon 1 \le i \le d).
\]
\end{corollary}
\begin{proof}

To avoid confusion, define
\begin{align*}
K_{1} &\colonequals K\left( [D] \in \mathcal{J}(\overline{K}) \colon (1 -
\zeta_{n})[D] = [P - \infty]  \right) \\
K_{2} &\colonequals K(\sqrt[n]{x_P + \alpha_i} \colon 1 \le i \le d).
\end{align*}

For any $[D_{1}] \in \mathcal{J}(\overline{K})$ satisfying $(1 -
\zeta_{n})[D_{1}] = [P - \infty]$,
\[
\{ [D] \in \mathcal{J}(\overline{K}) \colon (1 - \zeta_{n})[D] = [P - \infty] \} =
\{ [D_{1}] + T \colon T \in \mathcal{J}[1 - \zeta_{n}] \},
\]
so
\begin{equation}
\label{Equation:OneDForEntireDivision}
K_{1} = K\left( \mathcal{J}[1 - \zeta_{n}], [D_{1}] \right) = K\left( \alpha_{1},
\dots, \alpha_{d}, [D_{1}] \right).
\end{equation}
Observe that $\alpha_{i} = \left(\sqrt[n]{x_P + \alpha_i}\right)^{n} - x_{P}$
must lie in $K_{2}$, so we may as well assume that $K$ contains $\alpha_{1},
\dots, \alpha_{d}$.

Let $M \supseteq K$ be any extension. By \eqref{Equation:OneDForEntireDivision},
$M \supseteq K_{1}$ if and only if $[P - \infty] \in (1 - \zeta_{n})\mathcal{J}(M)$,
which by \autoref{Lemma:DivisionFieldOnePoint} holds if and only if $x_{P} +
\alpha_{i} \in M^{n}$, which is equivalent to $M \supseteq K_{2}$. Hence $K_{1}
= K_{2}$.
\end{proof}

\begin{corollary}
\label{Corollary:1mZSquaredDivisionField}
$K\left( \mathcal{J}[(1 - \zeta_{n})^{2}] \right) = K\left( \alpha_{1}, \dots,
\alpha_{d}, \sqrt[n]{\alpha_{i} - \alpha_{j}} \colon 1 \le i, j \le d \right).$
\end{corollary}
\begin{proof}
Since
\[
K \left( \mathcal{J}[(1 - \zeta_{n})] \right) = K \left( \alpha_{1}, \dots,
\alpha_{d} \right),
\]
we may as well assume that $K$ contains $\alpha_{1}, \dots, \alpha_{d}$.  Since
the $[\mathcal{W}_{i} - \infty]$ generate $\mathcal{J}[(1 - \zeta_{n})]$, 
\[
K\left( \mathcal{J}[(1 - \zeta_{n})^{2}] \right) = K \left( [D] \in
\mathcal{J}(\overline{K}) : (1 - \zeta_{n})[D] \in \{ [\mathcal{W}_{i} - \infty]
\colon i \in [1, d] \} \right),
\]
so we are done by applying \autoref{Corollary:DivisionFieldForAnyPoint} to
$P \in \mathcal{W}$.
\end{proof}

\section{The structure of \texorpdfstring{$T_{\ell} \mathcal{J}_{n,
d}$}{T\_ell J\_\{n,d\}} as a Galois representation}

Let $\zeta_{n d} \in \overline{\mathbf{Q}}$ be a primitive $n d$th root of
unity. Let
\begin{align*}
\zeta_{n} &\colonequals \zeta_{n d}^{d} \\
\zeta_{d} &\colonequals \zeta_{n d}^{n} \\
E &\colonequals \mathbf{Q}(\zeta_{n d}) \\
\mathcal{O}_{E} &\colonequals \mathbf{Z}[\zeta_{n d}].
\end{align*}
Suppose that $\xi_{n}$ and $\xi_{d}$ are primitive $n$th and $d$th roots of
unity respectively such that $\zeta_{n d} = \xi_{n} \xi_{d}$; then $\xi_{n}^{d}
= \zeta_{n}$ and $\xi_{d}^{n} = \zeta_{d}$. We will abuse notation and define
\begin{align*}
\zeta_{n d} &\colonequals \text{automorphism of }\mathcal{C}_{n, d}\text{ which
sends }(x, y)\text{ to }(\xi_{d} x, \xi_{n} y) \\
\zeta_{n} &\colonequals \zeta_{n d}^{d}\text{, which is the automorphism of
}\mathcal{C}_{n, d}\text{ which sends }(x, y)\text{ to }(x, \zeta_{n} y) \\
\zeta_{d} &\colonequals \zeta_{n d}^{n}\text{, which is the automorphism of
}\mathcal{C}_{n, d}\text{ which sends }(x, y)\text{ to }(\zeta_{d} x,  y) \\
Z &\colonequals \text{the subgroup of }\Aut(\mathcal{C}_{n, d})\text{
generated by } \zeta_{n d} \\
Z_n &\colonequals \text{the subgroup of }Z\text{ generated by } \zeta_{n} \\
Z_d &\colonequals \text{the subgroup of }Z\text{ generated by } \zeta_{d}.
\end{align*}

Let $\ell$ be a prime and $\lambda$ be a prime of $E$ lying above $\ell$. Let $r
\nmid \ell n d$ be another prime and $\mathfrak{r}$ be a prime of $E$ lying
above $r$. Define $\mathbf{F}_{\mathfrak{r}}$ to be the residue field at
$\mathfrak{r}$.  Define $\Frob_{\mathfrak{r}} \in \Gal \left( E(\mathcal{J}_{n,
d}[\ell^{\infty}]) / E \right)$ to be the Frobenius automorphism for
$\mathfrak{r}$; it is well-defined since the extension $E(\mathcal{J}_{n,
d}[\ell^{\infty}]) / E$ is unramified because $\mathcal{C}_{n, d}$ has good
reduction at $\ell$.

Define
\begin{align*}
T_{\ell} \mathcal{J}_{n, d} &\colonequals \text{the Tate module }
\varprojlim_{i} \mathcal{J}_{n, d}[\ell^i] \\
V_{\ell} \mathcal{J}_{n, d} &\colonequals \text{the rational Tate module }
(T_{\ell} \mathcal{J}_{n, d}) \otimes_{\mathbf{Z}_{\ell}} \mathbf{Q}_{\ell} \\
R_{\ell} &\colonequals R \otimes_{\mathbf{Z}} \mathbf{Z}_{\ell}.
\end{align*}

\begin{definition}
Define $\End_{R_{\ell}} \left( T_{\ell} \mathcal{J}_{n, d} \right)$ to be the
ring of endomorphisms of $T_{\ell} \mathcal{J}_{n, d}$ that commute with
$\zeta_{n d}$. 
\end{definition}

\begin{lemma}
\label{Lemma:TellJndFreeRellModule}
$T_{\ell} \mathcal{J}_{n, d}$ is free $R_{\ell}$-module of rank 1.  Hence, 
\begin{align}
\Aut_{R_{\ell}} \left( T_{\ell} \mathcal{J}_{n, d} \right) \simeq
R_{\ell}^{\times},  \label{Equation:AutTellJnd} \\
\End_{R_{\ell}} \left( T_{\ell} \mathcal{J}_{n, d} \right) \simeq R_{\ell}.
\label{Equation:EndTellJnd}
\end{align}
\end{lemma}
\begin{proof}
\autoref{Proposition:H1CndFreeRModule} gives that $H_{1}(\mathcal{C}_{n, d},
\mathbf{Z})$ is a free $R$-module of rank 1. Therefore, $T_{\ell}
\mathcal{J}_{n, d} \simeq H_{1}(\mathcal{C}_{n, d}, \mathbf{Z})
\otimes_{\mathbf{Z}} \mathbf{Z}_{\ell}$ is a free $R_{\ell}$-module of rank 1,
and this implies \eqref{Equation:AutTellJnd} and \eqref{Equation:EndTellJnd}. 
\end{proof}

\begin{definition}
\label{Definition:Thetaell}
Using \eqref{Equation:AutTellJnd}, define
\[
\theta_{\ell} \colon \Gal\left( E \left( \mathcal{J}_{n, d}[\ell^{\infty}]
\right) / E \right) \hookrightarrow \Aut_{R_{\ell}} \left( T_{\ell}
\mathcal{J}_{n, d} \right) \simeq R_{\ell}^{\times}
\]
to be the injective group homomorphism which sends each element of $\Gal\left( E
\left( \mathcal{J}_{n, d}[\ell^{\infty}] \right) / E \right)$ to its action on
$T_{\ell} \mathcal{J}_{n, d}$. Extend $\theta_{\ell}$ linearly to a ring
homomorphism
\[
\theta_{\ell} \colon \mathbf{Z}_{\ell}\left[ \Gal\left( E \left(
\mathcal{J}_{n, d}[\ell^{\infty}] \right) / E \right) \right] \to
\End_{R_{\ell}} \left( T_{\ell} \mathcal{J}_{n, d} \right) \simeq R_{\ell}.
\]
\end{definition}
\begin{remark}
When $n = p$ is prime, recall that $\theta_{p}$ was previously defined in
\autoref{Definition:Thetap} as a map from $\mathbf{Z}_{p}\left[
\Gal\left( \mathbf{Q}(\mu_{p}, \mathcal{J}_{p, d}[p^{\infty}]) /
\mathbf{Q}(\mu_{p}) \right) \right]$, but we are now defining $\theta_{p}$ to be
its restriction to the subring $\mathbf{Z}_{p}\left[ \Gal\left( E \left(
\mathcal{J}_{p, d}[p^{\infty}] \right) / E \right) \right]$; this is a subring
because the field of definition of $\mathcal{J}_{p, d}[p^{\infty}] \supseteq
\mathcal{J}_{p, d}[1 - \zeta_{p}]$ contains $\mathbf{Q}(\mu_{d})$, so
$\mathbf{Q}(\mu_{p}, \mathcal{J}[p^{\infty}])$ does contain $E$. 
\end{remark}

\begin{lemma}
\label{Lemma:KillTorsionInCpq}
Suppose that $n = p$ is prime.  Then $\gamma \in \mathbf{Z}_{p}\left[ \Gal\left(
E(\mathcal{J}_{p, d}[p^{\infty}]) / E \right) \right]$ kills $\mathcal{J}_{p,
d}[(1 - \zeta_{p})^{i}]$ if and only if
\[
\theta_{p}(\gamma) \in (1 - \zeta_{p})^{i} R_{p}.
\]
\end{lemma}
\begin{proof}
Since $T_{p} \mathcal{J}_{p, d}$ is a free $R_{p}$-module of rank 1 by
\autoref{Lemma:TellJndFreeRellModule}, this lemma is proved in the same way as
\autoref{Corollary:KillTorsionInGeneral}.
\end{proof}

\begin{lemma}
\label{Lemma:MetabelianTorsionField}
For each positive integer $m$, $\Gal\left( E(\mathcal{J}_{n, d}[m^{\infty}]) / E
\right)$ is abelian.
\end{lemma}
\begin{proof} 
Using the fact that the $\theta_{\ell}$ are injective, we see that
\[
\Gal\left( E(\mathcal{J}_{n, d}[m^\infty]) / E \right) \hookrightarrow
\prod_{\ell | m} \Gal\left( E(\mathcal{J}_{n, d}[\ell^\infty]) / E \right)
\hookrightarrow \prod_{\ell | m} R_{\ell}^{\times},
\]
so we are done since $R_{\ell}^{\times}$ is abelian.
\end{proof}

\begin{definition}
\label{Definition:JacobiSum}
Suppose that $\mathbf{F}_{Q}$ is a finite field and that $\mathcal{S}$ is a
ring. Suppose that $\chi_{1}, \chi_{2} \colon \mathbf{F}_{Q}^{\times} \to
\mathcal{S}^{\times}$ are nontrivial characters. For each integer $k \ge 1$,
define the Jacobi sum
\[
J_{k}(\chi_{1}, \chi_{2}) \colonequals \sum_{\alpha \in \mathbf{F}_{Q^{k}}
\setminus \{ 0, 1 \}} \chi_{1}\left(\alpha^{(Q^k - 1) / (Q - 1)}\right)
\chi_{2}\left( (1 - \alpha)^{(Q^k - 1) / (Q - 1)}\right).
\]
\end{definition}

\begin{definition}
\label{Definition:KappaGammandfrakr}
Define natural isomorphisms
\begin{align*}
\kappa_{n, \mathfrak{r}} &\colon \mu_{n}(\mathbf{F}_{\mathfrak{r}}) \to Z_{n} \\
\kappa_{d, \mathfrak{r}} &\colon \mu_{d}(\mathbf{F}_{\mathfrak{r}}) \to Z_{d}.
\end{align*}
Compose the ``exponentiation by $(\#\mathbf{F}_{\mathfrak{r}} -
1) / n$'' map and $\kappa_{n, \mathfrak{r}}$ to define
\[ 
\gamma_{n, \mathfrak{r}} \colon \mathbf{F}_{\mathfrak{r}}^\times \to
\mu_{n}(\mathbf{F}_{\mathfrak{r}}) \to Z_{n}.  
\] 
In an analogous fashion, define the composite morphism 
\[ 
\gamma_{d, \mathfrak{r}} \colon \mathbf{F}_{\mathfrak{r}}^\times \to
\mu_{d}(\mathbf{F}_{\mathfrak{r}}) \to Z_{d}.
\]
\end{definition}

\begin{definition}
For characters $\rho_{n} \colon Z_{n} \to E^\times$ and $\rho_{d} \colon Z_d \to
E^\times$, use $\left(V_{\ell} \mathcal{J}_{n, d}\right)^{(\rho_{n}, \rho_{d})}$
to denote the $(\rho_{n}, \rho_{d})$-isotypic component of $(V_{\ell}
\mathcal{J}_{n, d}) \otimes_{\mathbf{Q}_{\ell}} E_{\lambda}$, i.e., 
\[
\left(V_{\ell} \mathcal{J}_{n, d}\right)^{(\rho_{n}, \rho_{d})} \colonequals
\left\{ v \in (V_{\ell} \mathcal{J}_{n, d}) \otimes_{\mathbf{Q}_{\ell}}
E_{\lambda} : \begin{tabu}{l} z_{n}(v) = \rho_{n}(z_{n}) v \text{ for all }
z_{n} \in Z_{n} \\ z_{d}(v) = \rho_{d}(z_{d}) v \text{ for all } z_{d} \in Z_{d}
\end{tabu}\right\}.
\]
\end{definition}

\begin{proposition}
\label{Proposition:RationalTateModuleEigenspaces} \hfill
\begin{enumerate}[label=\upshape(\arabic*),
ref=\autoref{Proposition:RationalTateModuleEigenspaces}(\arabic*)]

\item \label{Proposition:RationalTateModuleEigenspacesDirectSum}
There is a direct sum decomposition
\[
(V_{\ell} \mathcal{J}_{n, d}) \otimes_{\mathbf{Q}_{\ell}} E_{\lambda} \simeq
\bigoplus_{\substack{\rho_n \colon Z \to E^{\times} \\ \rho_d \colon Z \to
E^{\times} }} \left( V_{\ell} \mathcal{J}_{n, d} \right)^{(\rho_n, \rho_d)}.
\]

\item \label{Proposition:RationalTateModuleEigenspacesNonzero}
The $\mathbf{Q}_{\ell}$-dimension of $\left( V_{\ell} \mathcal{J}_{n, d}
\right)^{(\rho_n, \rho_d)}$ is 
\[
\dim_{\mathbf{Q}_{\ell}} \left( V_{\ell} \mathcal{J}_{n, d}
\right)^{(\rho_n, \rho_d)} = \begin{cases}
1 &\text{if } \rho_n \neq 1 \text{ and } \rho_d \neq 1, \\
0 &\text{otherwise}.
\end{cases}
\]

\end{enumerate}
\end{proposition}
\begin{proof}
\autoref{Proposition:RationalTateModuleEigenspacesDirectSum} is just
representation theory (c.f. page 172 of \cite{katz1981crystalline}). 

By \autoref{Proposition:H1CndFreeRModule} and \autoref{Corollary:RIsomphiND}, we
know that the characteristic polynomial of $\zeta_{n d}$ on
$H_{1}(\mathcal{C}_{n, d}, \mathbf{Z})$ is $\varphi_{n, d}(T)$, so the
eigenvalues of $\zeta_{n d}$ on $T_{\ell} \mathcal{J}_{n, d} =
H_{1}(\mathcal{C}_{n, d}, \mathbf{Z}) \otimes_{\mathbf{Z}} \mathbf{Z}_{\ell}$
are $\mu_{n d} (\overline{\mathbf{Q}_{\ell}}) \setminus \left( \mu_{n}
(\overline{\mathbf{Q}_{\ell}}) \cup \mu_{d} (\overline{\mathbf{Q}_{\ell}})
\right)$; \autoref{Proposition:RationalTateModuleEigenspacesNonzero} follows.
\end{proof}

\begin{proposition} 
\label{Proposition:ApplyKatz}
Suppose that $\rho_{n} \colon Z_{n} \to E^\times$ and $\rho_{d} \colon Z_d \to
E^\times$ are nontrivial multiplicative characters.
\begin{enumerate}[label=\upshape(\arabic*),
ref=\autoref{Proposition:ApplyKatz}(\arabic*)]

\item \label{Proposition:ApplyKatzFrobeniusEigenvalue}
Let $\chi_{n, \mathfrak{r}} = \rho_n \circ \gamma_{n, \mathfrak{r}}$ and
$\chi_{d, \mathfrak{r}} = \rho_d \circ \gamma_{d, \mathfrak{r}}$. Then the
eigenvalue of $\Frob_{\mathfrak{r}}$ acting upon $\left( V_{\ell}
\mathcal{J}_{n, d} \right)^{(\rho_n, \rho_d)}$ is
\[
-\chi_{d, \mathfrak{r}}(-1) J_{1}(\chi_{n, \mathfrak{r}}, \chi_{d,
\mathfrak{r}}).
\]

\item \label{Proposition:ApplyKatzFrobeniusInZ}
Let $\gamma'_{n, \mathfrak{r}} \colon \mathbf{F}_{\mathfrak{r}}^{\times} \to
R^{\times}$ be the composite of $\gamma_{n, \mathfrak{r}}$ with the inclusion
$Z_{n} \subseteq R^{\times}$. Define $\gamma'_{d, \mathfrak{r}}$ similarly. Then
\begin{equation}
\label{Equation:ApplyKatzFrobeniusInZEquality}
\theta_{\ell}\left( \Frob_{\mathfrak{r}} \right) = - \gamma'_{d,
\mathfrak{r}}(-1) J_{1}(\gamma'_{n, \mathfrak{r}}, \gamma'_{d, \mathfrak{r}})
\in R.
\end{equation}

\end{enumerate}
\end{proposition}
\begin{proof}
We apply the results in Section 1 of Katz \cite{katz1981crystalline} to the
group $Z$ acting on the curve $\mathcal{C}_{n, d, \mathfrak{r}}$, which is the
reduction of $\mathcal{C}_{n, d}$ at the prime ideal $\mathfrak{r}$.

Define $\rho \colon Z \to E^\times$ to be the character that restricts to
$\rho_{n}$ on $Z_{n}$ and to $\rho_{d}$ on $Z_{d}$. As on page 172 of
\cite{katz1981crystalline}, for every integer $k \ge 1$, define $\Fix \left(
\Frob_{\mathfrak{r}}^k z^{-1} \right)$ to be the subset of $\mathcal{C}_{n, d,
\mathfrak{r}}(\overline{\mathbf{F}_{\mathfrak{r}}})$ fixed by
$\Frob_{\mathfrak{r}}^k z^{-1}$ and
\begin{align}
S(\rho, k) &\colonequals \frac{1}{\# Z} \sum_{z \in Z} \rho(z) \# \Fix
\left( \Frob_{\mathfrak{r}}^k z^{-1} \right) \nonumber \\
&= \frac{1}{\# Z} \sum_{P \in \mathcal{C}_{n, d,
\mathfrak{r}}(\overline{\mathbf{F}_{\mathfrak{r}}})} \sum_{\substack{z \in Z \\
\Frob_{\mathfrak{r}}^k z^{-1} P = P} } \rho(z)
\label{Equation:DefinitionOfSrhok}.
\end{align}

\

\noindent\textbf{Claim.} Let $Q = \# \mathbf{F}_{\mathfrak{r}}$. For every $k
\ge 1$, we have 
\[
S(\rho, k) = (\chi_{d, \mathfrak{r}}(-1))^{(Q^{k} - 1) / (Q - 1)} J_{k}(\chi_{n,
\mathfrak{r}}, \chi_{d, \mathfrak{r}}).
\]
\textbf{Proof of claim.} Suppose that $P \in \mathcal{C}_{n, d,
\mathfrak{r}}(\overline{\mathbf{F}_{\mathfrak{r}}})$ is fixed by $Z_{n}$. Then
$\{ z \in Z \colon \Frob_{\mathfrak{r}}^k z^{-1} P = P \} $ will be a union of
cosets for the subgroup $Z_{n}$ of $Z$. Since $\rho_{n} \neq 1$, the sum of
$\rho(z)$ for any coset of $Z_{n}$ is zero, and hence the inner sum of
\eqref{Equation:DefinitionOfSrhok} is zero. Hence we may ignore $P$ that are
fixed by $Z_{n}$. Similarly, we may also ignore $P$ that are fixed by $Z_{d}$.
Therefore, we may restrict to the subset
\begin{align*}
\mathcal{C}_{n, d, \mathfrak{r}}(\overline{\mathbf{F}_{Q}})^{*} &\colonequals \{
P \in \mathcal{C}_{n, d, \mathfrak{r}}(\overline{\mathbf{F}_{Q}}) \colon P
\text{ is not fixed by }
Z_{n} \text{ nor by } Z_{d} \} \\
&= \{ (x, y) \in \overline{\mathbf{F}_{\mathfrak{r}}}^{\times} \times
\overline{\mathbf{F}_{\mathfrak{r}}}^{\times} \colon y^{n} = x^{d} + 1 \}. 
\end{align*}
Suppose that $P = (x, y) \in \mathcal{C}_{n, d,
\mathfrak{r}}(\overline{\mathbf{F}_{Q}})^{*}$ and that $z \in Z$ satisfy
$\Frob_{\mathfrak{r}}^{k} z^{-1} P = P$. Since $\mathbf{F}_{\mathfrak{r}}$
contains an $nd$th root of unity, $\Frob_{\mathfrak{r}}^{k}$ and $z^{-1}$
commute, so this is equivalent to $\Frob_{\mathfrak{r}}^{k} (x, y) = z(x, y)$.
Since $z$ scales the $x$-coordinate by a $d$th root of unity and scales the
$y$-coordinate by a $n$th root of unity, we see that $x^{d}, y^{n}$ are both
fixed by $\Frob_{\mathfrak{r}}^{k}$. Define $\alpha \colonequals y^{n} = x^{d} +
1$, so that $\alpha$ is fixed by $\Frob_{\mathfrak{r}}^{k}$. Also, $\alpha \not
\in \{ 0, 1 \}$ since $x, y \neq 0$. Rewrite \eqref{Equation:DefinitionOfSrhok}
as
\[
S(\rho, k) = \frac{1}{\# Z} \sum_{\substack{\alpha \in
\overline{\mathbf{F}_{\mathfrak{r}}} \setminus \{ 0, 1 \} \\
\Frob_{\mathfrak{r}}^{k} \alpha = \alpha} } \; \; \sum_{\substack{(x,
y, z) \in \overline{\mathbf{F}_{\mathfrak{r}}}^{\times} \times
\overline{\mathbf{F}_{\mathfrak{r}}}^{\times} \times Z \\ \Frob_{\mathfrak{r}}^k
z^{-1} (x, y) = (x, y) \\ \alpha = y^{n} = x^{d} + 1} } \rho(z).
\]
Define $z_{n} \in Z_{n}$ and $z_{d} \in Z_{d}$ such that $z = z_{n} z_{d}$. By
definition of $\kappa_{n, \mathfrak{r}}$ and $\kappa_{d, \mathfrak{r}}$, we know
that $z(x, y) = \left( \kappa_{d, \mathfrak{r}}^{-1}(z_{d}) x, \kappa_{n,
\mathfrak{r}}^{-1}(z_{n}) y  \right)$, so the condition that
$\Frob_{\mathfrak{r}}^{k} (x, y) = z(x, y)$ is equivalent to the two conditions
$x^{Q^{k}} = \kappa_{d, \mathfrak{r}}^{-1}(z_{d}) x$ and $y^{Q^{k}} = \kappa_{n,
\mathfrak{r}}^{-1}(z_{d}) y$, which is equivalent to the two conditions
$z_{d} = \kappa_{d, \mathfrak{r}}(x^{Q^{k} - 1})$ and 
$z_{n} = \kappa_{n, \mathfrak{r}}(y^{Q^{k} - 1})$, which by definition of
$\alpha$, is equivalent to $z_{d} = \kappa_{d, \mathfrak{r}}((\alpha -
1)^{(Q^{k} - 1) / d})$ and $z_{n} = \kappa_{n, \mathfrak{r}}(\alpha^{(Q^{k} -
1) / n})$, which uniquely determines $z$. Since $\rho(z) = \rho(z_{n} z_{d}) =
\rho_{n}(z_{n}) \rho_{d}(z_{d})$, we may rewrite the sum as
\[
S(\rho, k) = \frac{1}{\# Z} \sum_{\substack{\alpha \in
\overline{\mathbf{F}_{\mathfrak{r}}} \setminus \{ 0, 1 \} \\
\Frob_{\mathfrak{r}}^{k} \alpha = \alpha} } \; \; \sum_{\substack{(x, y) \in
\overline{\mathbf{F}_{\mathfrak{r}}}^{\times} \times
\overline{\mathbf{F}_{\mathfrak{r}}}^{\times} \\ \alpha = y^{n} = x^{d} + 1} }
\rho_{n} \left( \kappa_{n, \mathfrak{r}}(\alpha^{(Q^{k} - 1) / n}) \right)
\rho_{d}\left( \kappa_{d, \mathfrak{r}}((\alpha - 1)^{(Q^{k} - 1) / d}) \right).
\]
For each $\alpha$, there are $n d = \# Z$ pairs $(x, y) \in
\overline{\mathbf{F}_{\mathfrak{r}}}^{\times} \times
\overline{\mathbf{F}_{\mathfrak{r}}}^{\times}$ satisfying $y^{n} = x^{d} + 1$,
so we simplify to
\[
S(\rho, k) = \sum_{\substack{\alpha \in \overline{\mathbf{F}_{\mathfrak{r}}}
\setminus \{ 0, 1 \} \\ \Frob_{\mathfrak{r}}^{k} \alpha = \alpha} } \rho_{n}
\left( \kappa_{n, \mathfrak{r}}(\alpha^{(Q^{k} - 1) / n}) \right) \rho_{d}\left(
\kappa_{d, \mathfrak{r}}((\alpha - 1)^{(Q^{k} - 1) / d}) \right).
\]
By definition of $\chi_{n, \mathfrak{r}}$, we know that $\chi_{n,
\mathfrak{r}}(\alpha) = \rho_{n}\left( \gamma_{n, \mathfrak{r}}(\alpha) \right)
= \rho_{n}\left( \kappa_{n, \mathfrak{r}} \left( \alpha^{(Q - 1) / n)} \right)
\right)$. A similar statement holds for $\chi_{d, \mathfrak{r}}$, so this sum
equals
\begin{align*}
S(\rho, k) &= \sum_{\substack{\alpha \in \overline{\mathbf{F}_{\mathfrak{r}}}
\setminus \{ 0, 1 \} \\ \Frob_{\mathfrak{r}}^{k} \alpha = \alpha} }
\chi_{n, \mathfrak{r}} \left( \alpha^{(Q^{k} - 1) / (Q - 1)} \right)  \chi_{d,
\mathfrak{r}} \left( (\alpha - 1)^{(Q^{k} - 1) / (Q - 1)} \right) \\
&= \left(\chi_{d, \mathfrak{r}}(-1)\right)^{(Q^{k} - 1) / (Q - 1)}
\sum_{\substack{\alpha \in \overline{\mathbf{F}_{\mathfrak{r}}} \setminus \{ 0,
1 \} \\ \Frob_{\mathfrak{r}}^{k} \alpha = \alpha} } \chi_{n, \mathfrak{r}}
\left( \alpha^{(Q^{k} - 1) / (Q - 1)} \right)  \chi_{d, \mathfrak{r}} \left( (1
- \alpha)^{(Q^{k} - 1) / (Q - 1)} \right) \\
&= \left(\chi_{d, \mathfrak{r}}(-1)\right)^{(Q^{k} - 1) / (Q - 1)}
J_{k}(\chi_{n, \mathfrak{r}}, \chi_{d, \mathfrak{r}})
\end{align*}
by definition of the Jacobi sum (\autoref{Definition:JacobiSum}). So
\textbf{the proof of the claim is complete.}

Theorem 2.1.3(b) of \cite{berndt1998gauss} implies that $|J_{k}(\chi_{n},
\chi_{d})| = Q^{k / 2}$, so by the claim, $|S(\rho, k)| = Q^{k / 2}$. Using
Lemma 1.1 of \cite{katz1981crystalline} (3) $\Longleftrightarrow$ (6), the
eigenvalue of $\Frob_{\mathfrak{r}}$ on $\left( V_{\ell} \mathcal{J}_{n, d,
\mathfrak{r}} \right)^{(\rho_n, \rho_d)}$ is $-S(\rho, 1) = - \chi_{d,
\mathfrak{r}}(-1) J_{1}(\chi_{n, \mathfrak{r}}, \chi_{d, \mathfrak{r}})$, which
gives \autoref{Proposition:ApplyKatzFrobeniusEigenvalue}.

We may check \autoref{Proposition:RationalTateModuleEigenspacesNonzero} after
tensoring up to $E_{\lambda}$ to work with $(V_{\ell} \mathcal{J}_{n, d})
\otimes_{\mathbf{Q}_{\ell}} E_{\lambda}$, and
\autoref{Proposition:RationalTateModuleEigenspaces} implies that it is
sufficient to check this on each $(V_{\ell} \mathcal{J}_{n, d})^{(\rho_n,
\rho_d)}$ whenever $\rho_n, \rho_d \neq 1$. For any $z_{n} \in Z_{n}$, the
eigenvalue of $z_{n}$ on $(V_{\ell} \mathcal{J}_{n, d})^{(\rho_n, \rho_d)}$ is
$\rho_{n}(z_{n})$, so the eigenvalue of $\gamma'_{n, \mathfrak{r}}(\alpha)$
acting on $(V_{\ell} \mathcal{J}_{n, d})^{(\rho_n, \rho_d)}$ is $\rho_{n}\left(
\gamma_{n, \mathfrak{r}}(\alpha) \right) = \chi_{n, \mathfrak{r}}(\alpha)$ (and
similarly for $\gamma'_{d, \mathfrak{r}}(1 - \alpha)$), meaning that the right
hand side of \eqref{Equation:ApplyKatzFrobeniusInZEquality} acts on $(V_{\ell}
\mathcal{J}_{n, d})^{(\rho_n, \rho_d)}$ by the scalar $-\chi_{n,
\mathfrak{r}}(-1) J_{1}(\chi_{n, \mathfrak{r}}, \chi_{d, \mathfrak{r}})$, so we
are done by \autoref{Proposition:ApplyKatzFrobeniusEigenvalue}.
\end{proof}

\begin{definition}
By Galois theory, $\Gal(E(\mathcal{J}_{n, d}[m^\infty]) / \mathbf{Q})$ acts by
conjugation on the subgroup $\Gal\left( E(\mathcal{J}_{n, d}[m^\infty]) / E
\right)$. By \autoref{Lemma:MetabelianTorsionField}, the subgroup $\Gal\left(
E(\mathcal{J}_{n, d}[m^\infty]) / E \right)$ is abelian, so this action factors
through an action of $\Gal(E / \mathbf{Q})$ on $\Gal\left( E(\mathcal{J}_{n,
d}[m^\infty]) / E \right)$. For $h \in \Gal\left( E(\mathcal{J}_{n,
d}[m^\infty]) / E \right)$ and $\gamma \in \Gal(E / \mathbf{Q})$, write
${}^{\gamma} \! h$ to denote the action of $\gamma$ on $h$.  Let $\sigma \in
\Gal(E / \mathbf{Q})$ be complex conjugation.
\end{definition}

\begin{proposition}
\label{Proposition:TauConstruction}
$\Frob_{\mathfrak{r}} \cdot {}^{\sigma}\!\Frob_{\mathfrak{r}}$ acts on $T_{\ell}
\mathcal{J}_{n, d}$ as multiplication by $\# \mathbf{F}_{\mathfrak{r}}$.
\end{proposition}
\begin{proof}
By \autoref{Proposition:RationalTateModuleEigenspaces}, we may as well verify
this on each eigenspace $\left( V_{\ell} \mathcal{J}_{n, d} \right)^{(\rho_{n},
\rho_{d})}$.  Apply $\sigma$ to everything in \autoref{Proposition:ApplyKatz}
to see that $\Frob_{\sigma(\mathfrak{r})}$ acts on $\left( V_{\ell}
\mathcal{J}_{n, d} \right)^{(\rho_{n}, \rho_{d})}$ as multiplication by
$\chi_{d, \sigma(\mathfrak{r})}(-1) J_{1}\left( \chi_{n, \sigma(\mathfrak{r})},
\chi_{d, \sigma(\mathfrak{r})} \right) = \sigma\left( \chi_{d, \mathfrak{r}}(-1)
J_{1}(\chi_{n, \mathfrak{r}}, \chi_{d, \mathfrak{r}}) \right)$, implying that
$\Frob_{\mathfrak{r}} \cdot {}^{\sigma}\!\Frob_{\mathfrak{r}} =
\Frob_{\mathfrak{r}} \Frob_{\sigma(\mathfrak{r})}$ acts on $\left( V_{\ell}
\mathcal{J}_{n, d} \right)^{(\rho_{n}, \rho_{d})}$ as multiplication by
$J_{1}(\chi_{n, \mathfrak{r}}, \chi_{d, \mathfrak{r}}) \cdot \sigma\left(
J_{1}(\chi_{n, \mathfrak{r}}, \chi_{d, \mathfrak{r}}) \right)$, which equals $\#
\mathbf{F}_{\mathfrak{r}}$ by Theorem 2.1.3(b) of \cite{berndt1998gauss}.
\end{proof}

\begin{corollary}
\label{Corollary:TauConstruction1.5}
Let $m$ be a nonnegative integer and let $h \in \Gal(E(\mathcal{J}_{n,
d}[m^{\infty}]) / E)$. By the Weil pairing, $E(\mathcal{J}_{n, d}[m^{\infty}])$
contains $E(\mu_{m^{\infty}})$, so suppose that $h$ acts on $\mu_{m^{\infty}}$
as multiplication by $c$.  Then $h \cdot {}^{\sigma}\! h$ acts on
$\mathcal{J}[m^{\infty}]$ as multiplication by $c$. 
\end{corollary}
\begin{proof}
Frobenius elements are dense by the Chebotarev density theorem, so we reduce to
checking on $h = \Frob_{\mathfrak{r}}$. By definition of Frobenius, $h$ acts on
$\mu_{m^{\infty}}$ as multiplication by $\# \mathbf{F}_{\mathfrak{r}}$, so we
are done by applying \autoref{Proposition:TauConstruction} to every
prime $\ell$ dividing $m$.
\end{proof}

\section{Galois action on the torsion of \texorpdfstring{$\mathcal{J}_{p,
q}$}{J\_\{p,q\}}}

In this section, we assume that $p$ and $q$ are distinct odd primes.

\subsection{Computation of the \texorpdfstring{$p$}{p}-torsion field and
\texorpdfstring{$q$}{q}-torsion field} 
\label{Subsection:TorsionFieldJacobiSum}

In this section, we use results of \cite{ArulJacobi} to compute some torsion
fields.  Since $p$ and $q$ are distinct odd primes, we may identify $R$ with
$\mathcal{O}_{E}$.

\begin{definition}
For nonnegative $i, j$, define the torsion field
\[
L_{i, j} \colonequals E( \mathcal{J}_{p, q}[(1 - \zeta_{p})^{i} (1 -
\zeta_{q})^{j}] ).
\]
\end{definition}

\begin{lemma}
\label{Lemma:LijFacts}
\hfill
\begin{enumerate}[label=\upshape(\arabic*),
ref=\autoref{Lemma:LijFacts}(\arabic*)]

\item \label{Lemma:LijFactsAbelian}
Each $L_{i, j}$ is an abelian extension of $E$.

\item \label{Lemma:LijFactsL11}
$L_{0, 0} = L_{0, 1} = L_{1, 0} = L_{1, 1} = E$. 

\item \label{Lemma:LijDisjointComposite}
$L_{i, j} = L_{i, 0} L_{0, j}$ and $E = L_{i, 0} \cap L_{0, j}$.

\item \label{Lemma:LijFactsLi0i1}
$L_{i, 1} = L_{i, 0} = E(\mathcal{J}_{p, q}[(1 - \zeta_{p})^{i}])$. 

\item \label{Lemma:LijFactsL0j01}
$L_{1, j} = L_{0, j} = E(\mathcal{J}_{p, q}[(1 - \zeta_{q})^{j}])$. 

\item \label{Lemma:LijFactsLpm11}
$L_{p - 1, 1} = L_{p - 1, 0} = E(\mathcal{J}_{p, q}[p])$.

\item \label{Lemma:LijFactsLqm11}
$L_{1, q - 1} = L_{0, q - 1} = E(\mathcal{J}_{p, q}[q])$.

\item \label{Lemma:LijFactsL21}
$L_{2, 1} = E\left(\sqrt[p]{1 - \zeta_{q}^{i}} : 1 \le i \le q - 1\right)$ and
$[L_{2, 1} : E] > 1$.

\item \label{Lemma:LijFactsL12}
$L_{1, 2} = E\left(\sqrt[q]{1 - \zeta_{p}^{i}} : 1 \le i \le p - 1\right)$ and
$[L_{1, 2} : E] > 1$.

\item \label{Lemma:LijFactsLp1}
$L_{p, 1} / E$ is a $p$-Kummer extension, i.e., it is generated by $p$th roots
of elements of $E$.

\item \label{Lemma:LijFactsL1q}
$L_{1, q} / E$ is a $q$-Kummer extension, i.e., it is generated by $q$th roots
of elements of $E$.

\end{enumerate}
\end{lemma}
\begin{proof}
\hfill
\begin{enumerate}[label=\upshape(\arabic*),
ref={the proof of \autoref{Lemma:LijFacts}(\arabic*)}]

\item \label{LemmaProof:LijFactsAbelian}
Since $\mathcal{J}_{p, q}[1 - \zeta_{p}] \subseteq \mathcal{J}_{p, q}[p]$ and
$\mathcal{J}_{p, q}[1 - \zeta_{q}] \subseteq \mathcal{J}_{p, q}[q]$, this is a
special case of \autoref{Lemma:MetabelianTorsionField}.

\item \label{LemmaProof:LijFactsL11}
By \autoref{Proposition:SuperellipticProp611Poonen}, $\mathcal{J}_{p, q}[1 -
\zeta_{p}]$ is generated by $[(-\zeta_{q}^{i}, 0) - \infty]$ and
$\mathcal{J}_{p, q}[1 - \zeta_{q}]$ is generated by $[(0, \zeta_{p}^{j}) -
\infty]$, so $L_{1, 1} = E$.

\item \label{LemmaProof:LijDisjointComposite}
By definition, $L_{i, j} = L_{i, 0} L_{0, j}$. By
\autoref{Corollary:ExponentOfTorsionFieldIsP}, $[L_{i, 0} : L_{0, 0}]$ is a
power of $p$ and $[L_{0, j} : L_{0, 0}]$ is a power of $q$, so $L_{i, 0} \cap
L_{0, j} = L_{0, 0} = E$.

\item \label{LemmaProof:LijFactsLi0i1}
\autoref{Lemma:LijDisjointComposite} implies that $L_{i, 1} = L_{i, 0} L_{0,
1}$ and \autoref{Lemma:LijFactsL11} gives that $L_{0, 1} = E$, so $L_{i, 1} =
L_{i, 0}$. By definition, $L_{i, 0} = E(\mathcal{J}_{p, q}[(1 -
\zeta_{p})^{i}])$.

\item \label{LemmaProof:LijFactsL0j01}
Similar to the proof of \autoref{Lemma:LijFactsLi0i1}.

\item \label{LemmaProof:LijFactsLpm11}
Since $(1 - \zeta_{p})^{p - 1} \in p R^{\times}$, we see that $\mathcal{J}_{p,
q}[(1 - \zeta_{p})^{p - 1}] = \mathcal{J}_{p, q}[p]$, so $L_{p - 1, 0} =
E(\mathcal{J}_{p, q}[p])$, and we are done by \autoref{Lemma:LijFactsLi0i1}.

\item \label{LemmaProof:LijFactsLqm11}
Similar to the proof of \autoref{Lemma:LijFactsLpm11}.

\item \label{LemmaProof:LijFactsL21}
Apply \autoref{Corollary:1mZSquaredDivisionField} by using the $x$-coordinate
map to view $\mathcal{C}_{p, q}$ as a degree $p$ superelliptic cover of
$\mathbf{P}^{1}$. This shows that $L_{2, 1}$ is generated over $E$ by adjoining
$p$th roots of $\zeta_{q}^{a} - \zeta_{q}^{b}$. Since $\zeta_{q}^{a} -
\zeta_{q}^{b} = \zeta_{q}^{a} (1 - \zeta_{q}^{b - a})$ and $\zeta_{q}$ already
has a $p$th root in $E$, we see that $L_{2, 1}$ is generated over $E$ by
adjoining $p$th roots of $1 - \zeta_{q}^{i}$.

Consider the ramification of $L_{2, 1}$ and $E$ above the prime $q$. The field
$L_{2, 1}$ contains $(1 - \zeta_{q})^{1 / p}$, so 
\[
e_{q}(L_{2, 1} / \mathbf{Q}) \ge p (q - 1) > q - 1 = e_{q}(E / \mathbf{Q}),
\]
so $L_{2, 1}$ has to strictly contain $E$. 

\item \label{LemmaProof:LijFactsL12}
Similar to the proof of \autoref{Lemma:LijFactsL21}.

\item \label{LemmaProof:LijFactsLp1}
\autoref{Lemma:LijFactsAbelian} implies that $L_{p, 1} / E$ is abelian. Since
$E$ already contains the $p$th roots of unity and
\autoref{Corollary:ExponentOfTorsionFieldIsP} implies that the exponent of
$\Gal(L_{p, 1} / E)$ divides $p$, we are done by Kummer theory.

\item \label{LemmaProof:LijFactsL1q}
Similar to the proof of \autoref{Lemma:LijFactsLp1}. \qedhere
\end{enumerate}
\end{proof}

\begin{definition}
Suppose that $\mathfrak{r}$ is a prime of $E$ lying over a prime $r$ of
$\mathbf{Q}$ such that $r \not\in \{ p, q \}$. Abuse notation and write
$\zeta_{p}, \zeta_{q} \in \mathbf{F}_{\mathfrak{r}}$ to denote the images of
$\zeta_{p}, \zeta_{q} \in \mathcal{O}_{E}$ under the reduction map
$\mathcal{O}_{E} \to \mathbf{F}_{\mathfrak{r}}$. 

For integers $i \in [0, p - 2]$, $j \in [1, q - 1]$, and $s \in [0, p - 1]$,
define 
\begin{align*}
u_{s, j} &\colonequals 1 - \zeta_{q}^{j} \zeta_{p}^{s} \in \mathcal{O}_{E} \\
\eta_{i, j} &\colonequals \prod_{s = 0}^{p - 1} u_{s, j}^{\binom{s}{i}} \in
\mathcal{O}_{E} \\
\eta_{i, j}' &\colonequals \prod_{s = 0}^{p - 1} u_{s, j}^{s^{i}} \in
\mathcal{O}_{E}.
\end{align*}
(We adopt the convention $0^{0} = 1$ here.) We will also use $u_{i, j}, \eta_{i,
j}, \eta_{i, j}'$ to denote the images of the same expressions in
$\mathbf{F}_{\mathfrak{r}}$.
\end{definition}

\begin{theorem}
\label{Theorem:MainResultJacobi}
Let $\chi_{p}$ and $\chi_{q}$ denote characters
$\mathbf{F}_{\mathfrak{r}}^\times \to E^{\times}$ of exact order $p$ and $q$,
respectively. Let $k \in [1, p - 1]$. Let $J$ be the Jacobi sum $J_{1}(\chi_p,
\chi_q)$. The following are equivalent:
\begin{enumerate}[label={\upshape{(\arabic*)}},
ref={\autoref{Theorem:MainResultJacobi}(\arabic*)}]

\item 
\label{Theorem:MainResultJacobiCongruence}
$J + 1 \in (1 - \zeta_{p})^{k} \mathcal{O}_{E}$;

\item 
\label{Theorem:MainResultJacobiCongruenceUnits}
$\eta_{i, j} \in \mathbf{F}_{\mathfrak{r}}^{\times p}$ for all $i \in [0, k -
2]$ and $j \in [1, q - 1]$;

\item
\label{Theorem:MainResultJacobiCongruenceUnitsHalf}
$\eta_{i, j} \in \mathbf{F}_{\mathfrak{r}}^{\times p}$ for all $i \in [0, k -
2]$ and $j \in [1, (q - 1) / 2]$.

\end{enumerate}
\end{theorem}
\begin{proof}
This is a special case (since $q$ is odd) of Theorem 8.0.2 of
\cite{ArulJacobi}.
\end{proof}

\begin{corollary}
\label{Corollary:FrobFixJacobi2}
Let $k \in [1, p - 1]$ be an integer. The following are equivalent:

\begin{enumerate}[label={\upshape{(\arabic*)}},
ref={\autoref{Corollary:FrobFixJacobi2}(\arabic*)}]

\item 
\label{Corollary:FrobFixJacobi2Congruence}
$\mathfrak{r}$ splits in the field $E(\mathcal{J}_{p, q}[(1 - \zeta_{p})^{k}])$

\item 
\label{Corollary:FrobFixJacobi2CongruenceUnits}
$\eta_{i, j} \in \mathbf{F}_{\mathfrak{r}}^{\times p}$ for all $i \in [0, k -
2]$ and $j \in [1, q - 1]$;

\item
\label{Corollary:FrobFixJacobi2CongruenceUnitsHalf}
$\eta_{i, j} \in \mathbf{F}_{\mathfrak{r}}^{\times p}$ for all $i \in [0, k -
2]$ and $j \in [1, (q - 1) / 2]$.

\end{enumerate}

\end{corollary}
\begin{proof}
Define $\gamma'_{p, \mathfrak{r}}, \gamma'_{q, \mathfrak{r}} \colon
\mathbf{F}_{\mathfrak{r}}^{\times} \to R^{\times} \simeq
\mathcal{O}_{E}^{\times}$ as in \autoref{Proposition:ApplyKatzFrobeniusInZ};
they are characters of exact orders $p$ and $q$. Then $\mathfrak{r}$ splits in
$E(\mathcal{J}_{p, q}[(1 - \zeta_{p})^{k}])$ if and only if
$\Frob_{\mathfrak{r}} - 1\text{ kills }\mathcal{J}_{p, q}[(1 - \zeta_{p})^{k}]$,
which by \autoref{Lemma:KillTorsionInCpq} is equivalent to
$\theta_{p}(\Frob_{\mathfrak{r}} - 1) \in (1 - \zeta_{p})^{k} R_{p}$, which by
\autoref{Proposition:ApplyKatzFrobeniusInZ} is equivalent to $- \gamma'_{q,
\mathfrak{r}}(-1) J(\gamma'_{p, \mathfrak{r}}, \gamma'_{q, \mathfrak{r}}) - 1
\in (1 - \zeta_{p})^{k} R = (1 - \zeta_{p})^{k} \mathcal{O}_{E}$ (recall from
\autoref{Proposition:ApplyKatzFrobeniusInZ} that
$\theta_{p}(\Frob_{\mathfrak{r}})$ lies in $R$). Since $q$ is odd, we know
$\gamma'_{q, \mathfrak{r}}(-1) = 1$, so we are done by
\autoref{Theorem:MainResultJacobi}.
\end{proof}

\begin{theorem}
\label{Theorem:pAndqTorsionFields}
Let $k \in [1, p - 1]$ be an integer. Then
\begin{align*}
L_{k, 1} &= E\left(\sqrt[p]{\eta_{i, j}} \colon i \in [0, k - 2]
\text{ and } j \in [1, q - 1]\right) \\
&= E\left(\sqrt[p]{\eta_{i, j}} \colon i \in [0, k - 2]
\text{ and } j \in [1, (q - 1) / 2]\right). 
\end{align*}
\end{theorem}
\begin{proof}
Define
\begin{align*}
L_{k, 1}' &\colonequals E\left(\sqrt[p]{\eta_{i, j}} \colon i \in [0, k - 2]
\text{ and } j \in [1, q - 1]\right), \\
L_{k, 1}'' &\colonequals E\left(\sqrt[p]{\eta_{i, j}} \colon i \in [0, k - 2]
\text{ and } j \in [1, (q - 1) / 2]\right).
\end{align*}
For any extension $M$ of $E$ and subset $S$ of primes of $\mathbf{Q}$, define
\[
\Spl_{S}(M / E) \colonequals \left\{ \mathfrak{r} \text{ is a prime of } E
\colon \begin{tabu}{c} \mathfrak{r}\text{ splits in } M \text{ and } \\
\mathfrak{r} \text{ does not lie above a prime in } S \end{tabu} \right\}.
\]
\autoref{Corollary:FrobFixJacobi2} implies that $\Spl_{ \{ p, q \} }(L_{k, 1} /
E) = \Spl_{ \{ p, q \} }(L_{k, 1}' / E) = \Spl_{ \{ p, q \} }(L_{k, 1}'' / E)$,
so since $L_{k, 1}$, $L_{k, 1}'$, and $L_{k, 1}''$ are Galois extensions of $E$,
the Chebotarev density theorem implies that $L_{k, 1} = L_{k, 1}' = L_{k, 1}''$.
\end{proof}

\begin{lemma}
\label{Lemma:EtaijEta'ij}
Suppose that $i \in [0, p - 3]$ and $j \in [1, q - 1]$.
\begin{enumerate}
\item 
The image of $\eta_{i, j}$ in $E^{\times} / E^{\times p}$ lies in the subgroup
generated by $\eta_{i, j}', \cdots, \eta_{0, j}'$.
\item 
The image of $\eta_{i, j}'$ in $E^{\times} / E^{\times p}$ lies in the subgroup
generated by $\eta_{i, j}, \cdots, \eta_{0, j}$.
\end{enumerate}
\end{lemma}
\begin{proof}
Observe that there exist integers $b_{i, k} \in \mathbf{Z}$ such that for each
$i$,
\[
T^{i} = b_{i, i} \binom{T}{i} + b_{i, i - 1} \binom{T}{i - 1} + \cdots +
b_{i, 0} \binom{T}{0} \quad\quad \text{in }\mathbf{Z}[T]
\]
and $b'_{i, j} \in \mathbf{Z}_{(p)}$ such that for each $i$,
\[
\binom{T}{i} = b'_{i, i} T^{i} + b'_{i, i - 1} T^{i - 1} + \cdots + b'_{i, 0}
T^{0} \quad\quad \text{in }\mathbf{Z}_{(p)}[T].
\]
We are now done by using the definition of $\eta_{i, j}$ and $\eta_{i, j}'$.
\end{proof}

\begin{corollary}
\label{Corollary:pAndqTorsionFieldsEtaPrime}
Let $k \in [1, p - 1]$ be an integer. Then
\begin{align*}
L_{k, 1} &= E\left(\sqrt[p]{\eta_{i, j}'} \colon i \in [0, k - 2]
\text{ and } j \in [1, q - 1]\right) \\
&= E\left(\sqrt[p]{\eta_{i, j}'} \colon i \in [0, k - 2]
\text{ and } j \in [1, (q - 1) / 2]\right).
\end{align*}
\end{corollary}
\begin{proof}
This follows immediately from \autoref{Theorem:pAndqTorsionFields} and
\autoref{Lemma:EtaijEta'ij}.
\end{proof}

\begin{definition}
Let $\omega : \Gal(E / \mathbf{Q}(\mu_{q})) \simeq \Gal(\mathbf{Q}(\mu_{p}) /
\mathbf{Q}) \to \mathbf{Z}_{p}$ be the composite of the natural isomorphism
$\Gal(E / \mathbf{Q}(\mu_{q})) \simeq \Gal(\mathbf{Q}(\mu_{p}) / \mathbf{Q})$
with the Teichm\"uller character. If $A$ is an abelian group which has an action
of $\Gal(E / \mathbf{Q}(\mu_{q}))$ and $i \in \mathbf{Z}$, use
$\varepsilon_{\omega^{i}} A$ to denote the subgroup of $A$ for which $\Gal(E /
\mathbf{Q}(\mu_{q}))$ acts as $\omega^{i}$.
\end{definition}

\begin{definition} 
\label{Definition:DeltaMi}
For each $i \in [0, p - 2]$, define
\begin{align*}
\Delta_{i} &\colonequals \text{the subgroup of } E^{\times} / E^{\times p}
\text{ generated by } \eta_{i, 1}', \ldots, \eta_{i, (q - 1) / 2}' \\
M_{i} &\colonequals E(\sqrt[p]{\delta} : \delta \in \Delta_{i}).
\end{align*}
\end{definition}

\begin{lemma}
\label{Lemma:DeltaiTeichmullerEigenspace}
$\Delta_{i} \subseteq \varepsilon_{\omega^{-i}} (E^{\times} / E^{\times p})$ for
each $i \in [0, p - 2]$.
\end{lemma}
\begin{proof}
This follows from a straightforward computation of the $\Gal(E /
\mathbf{Q}(\mu_{q}))$-action on each $\eta_{i, j}'$. 
\end{proof}

\begin{lemma} 
\label{Lemma:DisjointCompositeOfMi}
Let $k \in [2, p - 1]$ be an integer. 

\begin{enumerate}[label=\upshape{(\arabic*)},
ref={\autoref{Lemma:DisjointCompositeOfMi}(\arabic*)}]

\item 
\label{Lemma:DisjointCompositeOfMiComposite}
$L_{k, 1}$ is the compositum of $M_{0}$, \ldots, $M_{k - 2}$ over $E$.

\item
\label{Lemma:DisjointCompositeOfMiDisjoint}
The fields $M_{0}$, \dots, $M_{p - 2}$ are disjoint over $E$.

\item
\label{Lemma:DisjointCompositeOfMiDegree}
$[L_{k, 1} : L_{k - 1, 1}] = [M_{k - 2} : E]$.
\end{enumerate}
\end{lemma}

\begin{proof} \hfill
\begin{enumerate}[label=\upshape{(\arabic*)},
ref={the proof of \autoref{Lemma:DisjointCompositeOfMi}(\arabic*)}]

\item 
\label{LemmaProof:DisjointCompositeOfMiComposite}
This follows immediately from \autoref{Corollary:pAndqTorsionFieldsEtaPrime} and
\autoref{Definition:DeltaMi}.

\item
\label{LemmaProof:DisjointCompositeOfMiDisjoint}
By Kummer theory, we must check that if $\delta_{i} \in \Delta_{i}$ satisfy
$\prod_{i = 0}^{p - 2} \delta_{i} = 1$, then $\delta_{i} = 1$ for all $i$.
\autoref{Lemma:DeltaiTeichmullerEigenspace} implies that each $\delta_{i}$ lies
in a different isotypic component for the $\omega$-action, so they must all be
$1$. \qedhere

\item
\label{LemmaProof:DisjointCompositeOfMiDegree}
This follows immediately from \autoref{Lemma:DisjointCompositeOfMiComposite} and 
\autoref{Lemma:DisjointCompositeOfMiDisjoint}.
\end{enumerate}
\end{proof}

\begin{corollary}
\label{Corollary:Lk1DegreeUpperBound}
For each integer $k \in [2, p - 1]$, there exists an integer $e(k) \in [0, (q -
1)/ 2]$ such that 
\[
[L_{k, 1} : L_{k - 1, 1}] = p^{e(k)}.
\]
\end{corollary}
\begin{proof}
This follows immediately from \autoref{Lemma:DisjointCompositeOfMiDegree}.
\end{proof}

\begin{remark}
For our strategy of using large Galois action to classify torsion points, we
need a lower bound for $[L_{k, 1} : L_{k - 1, 1}]$.
\autoref{Corollary:Lk1DegreeUpperBound} gives an upper bound, but we do not know
when it is possible to attain this value. In light of
\autoref{Lemma:DisjointCompositeOfMiDegree}, we focus our efforts on studying
$M_{i}$ for the rest of the section.
\end{remark}

\begin{corollary}
\label{Corollary:L21L11NontrivialAlways}
$[L_{2, 1} : L_{1, 1}] \ge p$.
\end{corollary}
\begin{proof}
This follows immediately from \autoref{Corollary:Lk1DegreeUpperBound} and
\autoref{Lemma:LijFactsL21}.
\end{proof}

\begin{definition}
Recall the notion of a cyclotomic unit as in Section 8.1 of
\cite{washington1997introduction}. Define 
\begin{align*}
& U &&\quad\text{to be the unit group of }\mathbf{Q}(\mu_{p}) \\
& C &&\quad\text{to be the group of cyclotomic units of }\mathbf{Q}(\mu_{p}) \\
& Q(\mu_{p})^{+} &&\quad\text{to be the totally real subfield of
}\mathbf{Q}(\mu_{p}) \\
& U^{+} &&\quad\text{to be the unit group of }\mathbf{Q}(\mu_{p})^{+} \\
& C^{+} &&\quad\text{to be the group of cyclotomic units of
}\mathbf{Q}(\mu_{p})^{+} \\
& A &&\quad\text{to be the class group of } \mathbf{Q}(\mu_{p}).
\end{align*}
For any group $B$, let $B_{p}$ be the $p$-Sylow subgroup of $B$.

For $i \in [0, p - 3]$ and $b \in (\mathbf{Z} / p \mathbf{Z})^{\times}$, define
\[
U_{i}(b) \colonequals \prod_{s = 0}^{p - 1} \left( \zeta_{p}^{(p + 1)(1 - b) s /
2} \frac{1 - \zeta_{p}^{b s}}{1 - \zeta_{p}^{s}} \right)^{s^{i}} \in
\mathbf{Q}(\mu_{p})^{+}
\]
($U_{i}(b)$ lies in $\mathbf{Q}(\mu_{p})^{+}$ since each term $\zeta_{p}^{(p +
1)(1 - b) s / 2} (1 - \zeta_{p}^{b s})/(1 - \zeta_{p}^{s})$ is fixed by complex
conjugation.)

Let $\nu \in (\mathbf{Z} / p \mathbf{Z})^{\times}$ be a generator and define
$U_{i} \colonequals U_{i}(\nu)$.
\end{definition}

\begin{lemma}
\label{Lemma:AnnoyingUi}
Suppose that $i \in [0, p - 3]$, $b \in (\mathbf{Z} / p \mathbf{Z})^{\times}$,
and $b^{i} \not\equiv 1 \pmod{p}$. 

\begin{enumerate}[label=\upshape{(\arabic*)},
ref={\autoref{Lemma:AnnoyingUi}(\arabic*)}]

\item 
\label{Lemma:AnnoyingUiSimpler}
Then
\[
U_{i}(b) = \prod_{s = 0}^{p - 1} \left( \frac{1 - \zeta_{p}^{b s}}{1 -
\zeta_{p}^{s}} \right)^{s^{i}}.
\]

\item
\label{Lemma:AnnoyingUiRelations}
The images of $U_{i}$ and $U_{i}(b)$ in $\mathbf{Q}(\mu_{p})^{+ \times} /
\mathbf{Q}(\mu_{p})^{+ \times p}$ generate the same subgroup.
\end{enumerate}

\end{lemma}
\begin{proof} \hfill
\begin{enumerate}[label=\upshape{(\arabic*)},
ref={the proof of \autoref{Lemma:AnnoyingUi}(\arabic*)}]

\item 
\label{LemmaProof:AnnoyingUiSimpler}
Since $i \in [0, p - 3]$, 
\[
\sum_{s = 0}^{p - 1} s^{i + 1} \equiv 0 \pmod{p},
\]
so
\[
\prod_{s = 0}^{p - 1} \left( \zeta_{p}^{(p + 1)(1 - b) s/ 2} \right)^{s^{i}} = 1
\]
and we are done by definition of $U_{i}(b)$.

\item
\label{LemmaProof:AnnoyingUiRelations}

For notational convenience, use the shorthand $\zeta_{p}^{m / 2}$ to mean
$\zeta_{p}^{m (p + 1) / 2}$, i.e., it is a $p$th root of unity whose square is
$\zeta_{p}^{m}$. For any $c \in \mathbf{Z} / p \mathbf{Z}$,
\[
U_{i}(c) = \prod_{s = 0}^{p - 1} \left( \frac{\zeta_{p}^{cs/2} -
\zeta_{p}^{-cs/2}}{\zeta_{p}^{s/2} - \zeta_{p}^{-s/2}} \right)^{s^{i}},
\]
so
\begin{align*}
(U_{i}(c))^{c^{i}} &= \prod_{s = 0}^{p - 1} \left( \frac{\zeta_{p}^{cs/2} -
\zeta_{p}^{-cs/2}}{\zeta_{p}^{s/2} - \zeta_{p}^{-s/2}} \right)^{(cs)^{i}} \\
&= \left(\prod_{s = 0}^{p - 1} (\zeta_{p}^{cs/2} -
\zeta_{p}^{-cs/2})^{(cs)^{i}}\right)\left( \prod_{s = 0}^{p - 1}
(\zeta_{p}^{s/2} - \zeta_{p}^{-s/2})^{(cs)^{i}} \right)^{-1} \\
&\equiv \left(\prod_{t = 0}^{p - 1} (\zeta_{p}^{t/2} -
\zeta_{p}^{-t/2})^{t^{i}}\right)\left( \prod_{s = 0}^{p - 1} (\zeta_{p}^{s/2} -
\zeta_{p}^{-s/2})^{(cs)^{i}} \right)^{-1} \pmod{\mathbf{Q}(\mu_{p})^{+ \times
p}}
\end{align*}
by setting $t \equiv c s \pmod{p}$ in the first product and observing that this
change-of-variable preserves the product modulo $\mathbf{Q}(\mu_{p})^{+ \times
p}$. Combining the products yields
\[
(U_{i}(c))^{c^{i}} \equiv \left(\prod_{t = 0}^{p - 1} (\zeta_{p}^{t/2} -
\zeta_{p}^{-t/2})^{t^{i}}\right)^{1 - c^{i}} \pmod{\mathbf{Q}(\mu_{p})^{+
\times p}},
\]
so if we define
\[
U \colonequals \prod_{t = 0}^{p - 1} (\zeta_{p}^{t/2} - \zeta_{p}^{-t/2})^{t^{i}}
\]
then
\[
(U_{i}(c))^{c^{i}} \equiv U^{1 - c^{i}} \pmod{\mathbf{Q}(\mu_{p})^{+ \times p}},
\]
so
\[
(U_{i}(b))^{b^{i}(1 - \nu^{i})} \equiv U^{(1 - b^{i})(1 - \nu^{i})} \equiv
U_{i}^{\nu^{i}(1 - b^{i})}\pmod{\mathbf{Q}(\mu_{p})^{+ \times p}}.
\]
Since $b^{i}, 1 - b^{i}, \nu^{i}, 1 - \nu^{i}$ are invertible modulo $p$, we are
done. \qedhere
\end{enumerate}
\end{proof}

\begin{lemma} 
Suppose that $i \in [1, p - 3]$ and $j \in [1, q - 1]$.
\label{Lemma:Normueta'ij}
\begin{enumerate}[label=\upshape{(\arabic*)},
ref={\autoref{Lemma:Normueta'ij}(\arabic*)}]

\item 
\label{Lemma:Normueta'iju}
For each $s \in [1, p - 1]$, 
\[
\Norm_{E / \mathbf{Q}( \mu_{p} )} u_{s, j} = \frac{1 - \zeta_{p}^{q s}}{1 -
\zeta_{p}^{s}}.
\]

\item
\label{Lemma:Normueta'ijeta'}
We have 
\[
\Norm_{E / \mathbf{Q}( \mu_{p} )} \eta'_{i, j} = U_{i}(q). 
\]

\item
\label{Lemma:Normueta'ijUiMip}
Suppose that $q^{i} \not\equiv 1 \pmod{p}$. Then $U_{i} \in M_{i}^{p}$.
\end{enumerate}
\end{lemma}

\begin{proof} \hfill
\begin{enumerate}[label=\upshape{(\arabic*)},
ref={the proof of \autoref{Lemma:Normueta'ij}(\arabic*)}]

\item 
\label{LemmaProof:Normueta'iju}
This is a straightforward computation.

\item
\label{LemmaProof:Normueta'ijeta'}
Combine the definition of $\eta_{i, j}'$, \autoref{Lemma:Normueta'iju}, and 
\autoref{Lemma:AnnoyingUiSimpler}.

\item
\label{LemmaProof:Normueta'ijUiMi}
$M_{i}$ is Galois over $\mathbf{Q}(\mu_{p})$ and $\eta_{i, 1}' \in
M_{i}^{p}$, so $\Norm_{E / \mathbf{Q}( \mu_{p} )} \eta_{i, 1}' \in M_{i}^{p}$,
so we are done by \autoref{Lemma:Normueta'ijeta'} and
\autoref{Lemma:AnnoyingUiRelations}.  \qedhere
\end{enumerate}
\end{proof}

\begin{corollary}
\label{Corollary:CyclotomicRibetGenDisjoint}
For any $i \in [1, p - 3]$ such that $q^{i} \not\equiv 1 \pmod{p}$,
\[
[M_{i} : E] \ge [E(\sqrt[p]{U_{i}}) : E] =
[\mathbf{Q}(\mu_{p})^{+}(\sqrt[p]{U_{i}}) : \mathbf{Q}(\mu_{p})^{+}].
\]
\end{corollary}
\begin{proof}
\autoref{Lemma:Normueta'ijUiMip} implies that $M_{i} \supseteq
E(\sqrt[p]{U_{i}})$, so $[M_{i} : E] \ge [E(\sqrt[p]{U_{i}}) : E]$.

Let $F \colonequals \mathbf{Q}(\mu_{p})^{+}(\sqrt[p]{U_{i}})$. Since $E /
\mathbf{Q}(\mu_{p})$ is totally ramified at $q$ and $F(\mu_{p}) /
\mathbf{Q}(\mu_{p})$ is unramified at $q$, $E \cap F \subseteq
\mathbf{Q}(\mu_{p})$. Therefore, 
\begin{equation}
\label{Equation:SandwichECapF}
\mathbf{Q}(\mu_{p})^{+} \subseteq E \cap F \subseteq \mathbf{Q}(\mu_{p}).
\end{equation}
Since $[F : \mathbf{Q}(\mu_{p})^{+}] \in \{ 1, p \}$ and $[F :
\mathbf{Q}(\mu_{p})^{+}]$ is divisible by $[E \cap F : \mathbf{Q}(\mu_{p})^{+}
]$, we must have $E \cap F \neq \mathbf{Q}(\mu_{p})$, so
\eqref{Equation:SandwichECapF} implies
\begin{equation}
\label{Equation:IntersectEAndF}
E \cap F = \mathbf{Q}(\mu_{p})^{+}.
\end{equation}
Note that $[E(\sqrt[p]{U_{i}}) : E] = [EF : E] = [F : E \cap F]$ since $E$ is
Galois over $E \cap F$, so we are done by \eqref{Equation:IntersectEAndF}.
\end{proof}

\begin{theorem}
\label{Theorem:IsotypicGenerator}
For even $i \in [2, p - 3]$, $\# \varepsilon_{\omega^{i}} (U^{+} / C^{+})_{p} =
1$ if and only if $U_{p - 1 - i} \not\in \mathbf{Q}(\mu_{p})^{+ \times p}$.
\end{theorem}
\begin{proof}
This follows from Section 8.3 of \cite{washington1997introduction} (Washington
uses $E$ and $E^{+}$ to denote the unit groups of $\mathbf{Q}(\mu_{p})$ and
$\mathbf{Q}(\mu_{p})^{+}$, respectively); see the discussion on pages 155--156:
$U_{p - 1 - i}$ is a $p$th power if and only if Washington's $E_{i}^{(N)}$ is a
$p$th power, if and only if the $\omega^{i}$-isotypic component of $(U^{+} /
C^{+})_{p}$ is nontrivial. 
\end{proof}

\begin{theorem}
\label{Theorem:HerbrandRibet}
For even $i \in [2, p - 3]$, 
\[
\# \varepsilon_{\omega^{i}} (U^{+} / C^{+})_{p} = \# \varepsilon_{\omega^{i}}
A_{p}.
\]
\end{theorem}
\begin{proof}
See Theorem 15.7 on page 342 of \cite{washington1997introduction}.
\end{proof}

\begin{theorem}
\label{Theorem:Kurihara}
For any odd prime $p$, $\# \varepsilon_{\omega^{p - 3}} A_{p} = 1$.
\end{theorem}
\begin{proof}
This is Corollary 3.8 on page 230 of \cite{kurihara1992some}.
\end{proof}

\begin{corollary}
\label{Corollary:L41L31NontrivialSometimes}
If $q^2 \not\equiv 1 \pmod{p}$, then $[L_{4, 1} : L_{3, 1}] \ge p$.
\end{corollary}
\begin{proof}
Taking $i = p - 3$ and combining
\cref{Theorem:IsotypicGenerator,Theorem:HerbrandRibet,Theorem:Kurihara} yields
$U_{2} \not\in \mathbf{Q}(\mu_{p})^{+ \times p}$, so taking $i = 2$ in
\autoref{Corollary:CyclotomicRibetGenDisjoint} yields $[M_{2} : E] \ge p$, and
we are done by \autoref{Lemma:DisjointCompositeOfMiDegree}.
\end{proof}

\begin{corollary}
\label{Corollary:L41BigEnough}
If $q^2 \not\equiv 1 \pmod{p}$, then $[L_{4, 1} : E] \ge p^{2}$.
\end{corollary}
\begin{proof}
This follows from \autoref{Corollary:L21L11NontrivialAlways} and
\autoref{Corollary:L41L31NontrivialSometimes}.
\end{proof}

\begin{lemma}
\label{Lemma:L31L21L11q3p13}
Suppose that $q = 3$ and $p \in \{ 5, 7, 11, 13 \}$. Then $[L_{3, 1} : L_{2, 1}]
= [L_{2, 1} : E] = p$.
\end{lemma}
\begin{proof}
By \autoref{Corollary:Lk1DegreeUpperBound}, it suffices to check that $L_{3, 1} /
L_{2, 1} / E$ is a tower where each successive step is nontrivial. The bottom
extension $L_{2, 1} / E$ is known to be nontrivial by
\autoref{Lemma:LijFactsL21}, so it suffices to check that $L_{3, 1} / L_{2, 1}$
is nontrivial. For each $p \in \{ 5, 7, 11, 13 \}$, we find a prime $r$ and a
prime $\mathfrak{r}$ of $E$ lying above $r$ such that $\Frob_{\mathfrak{r}} - 1$
kills $\mathcal{J}_{p, q}[(1 - \zeta_{p})^{2}]$ but not $\mathcal{J}_{p, q}[(1 -
\zeta_{p})^{3}]$; using \autoref{Corollary:FrobFixJacobi2}, a computer
calculation shows that we may use the prime $\mathfrak{r}$ specified by the
following table.
\begin{center}
\begin{tabular}{c|cccc}
$p$ & $5$ & $7$ & $11$ & $13$ \\\hline
$\#\mathbf{F}_{\mathfrak{r}}$ & $2^{4}$ & $13^{2}$ & $43^{2}$ & $547$
\end{tabular}
\end{center}
\end{proof}

\subsection{The structure of \texorpdfstring{$J_{p,
q}[\ell^{\infty}]$}{J\_\{p,q\}[ell\^{}infinity]} as a
\texorpdfstring{$G_{\mathbf{Q}} Z$}{G\_QZ}-module}

As before, suppose that $p, q$ are distinct odd primes. 

\begin{definition}
Let $G_{\mathbf{Q}}$ denote the absolute Galois group of $\mathbf{Q}$. 
\end{definition}

\begin{lemma} 
\label{Lemma:GenerateZGZ}
Suppose that $\ell$ is a prime and that $k \ge 1$ is an integer. 
\begin{enumerate}[label=\upshape(\arabic*),
ref=\autoref{Lemma:GenerateZGZ}(\arabic*)]

\item \label{Lemma:GenerateZGZEllNotpqGZ}
Suppose that $\ell \not\in \{ p, q \}$ and that $D \in \mathcal{J}_{p,
q}[\ell^{k}] \setminus \mathcal{J}_{p, q}[\ell^{k - 1}]$. Then $G_{\mathbf{Q}} Z
D$ generates $\mathcal{J}_{p, q}[\ell^{k}]$.

\item \label{Lemma:GenerateZGZEllpqGZ}
Suppose that $\ell \in \{ p, q \}$ and that $D \in \mathcal{J}_{p, q}[(1 -
\zeta_{\ell})^{k}] \setminus \mathcal{J}_{p, q}[(1 - \zeta_{\ell})^{k - 1}]$.
Then $G_{\mathbf{Q}} Z D$ generates $\mathcal{J}_{p, q}[(1 -
\zeta_{\ell})^{k}]$.

\end{enumerate}
\end{lemma}
\begin{proof}  \hfill
\begin{enumerate}[label=\upshape(\arabic*),
ref={the proof of \autoref{Lemma:GenerateZGZ}}(\arabic*)]

\item \label{LemmaProof:GenerateZGZEllNotpqGZ}
Suppose that $k = 1$. \autoref{Corollary:SubringByXi} implies that
$\mathcal{J}_{p, q}[\ell] \simeq R / \ell R$, so by linear algebra,
$\mathcal{J}_{p, q}[\ell] \otimes_{\mathbf{F}_{\ell}}
\overline{\mathbf{F}_{\ell}}$ breaks up into the direct sum of its
one-dimensional eigenspaces for the $\zeta_{n d}$-action. The action of
$G_{\mathbf{Q}}$ on $\mathcal{J}_{p, q}[\ell]$ permutes the eigenspaces
transitively.

We will show that $G_{\mathbf{Q}} Z D$ generates $\mathcal{J}_{p,
q}[\ell] \otimes_{\mathbf{F}_{\ell}} \overline{\mathbf{F}_{\ell}}$ as an
$\overline{\mathbf{F}_{\ell}}$-vector space. Since $D$ is nonzero, there is some
eigenspace for which its projection is nonzero, so since $\ell \nmid p q$, there
exists $r \in \overline{\mathbf{F}_{\ell}}[Z]$ such that $r D$ is a
nonzero eigenvector.  Since $G_{\mathbf{Q}}$ acts on the eigenspaces
transitively, the $G_{\mathbf{Q}}$-orbit of $r D$ hits every eigenspace;
hence, $G_{\mathbf{Q}} Z D$ generates $\mathcal{J}_{p, q}[\ell]
\otimes_{\mathbf{F}_{\ell}} \overline{\mathbf{F}_{\ell}}$ as a
$\overline{\mathbf{F}_{\ell}}$-vector space. This completes the case $k = 1$.

Now suppose that $k \ge 1$. Multiplication by $\ell^{k - 1}$ provides an
isomorphism 
\[
\mu \colon \mathcal{J}_{p, q}[\ell^{k}] / \ell \mathcal{J}_{p, q}[\ell^{k}]
\simeq \mathcal{J}_{p, q}[\ell].
\]
The proof of the $k = 1$ case shows that the image of $G_{\mathbf{Q}} Z
D$ under $\mu$ generates $\mathcal{J}_{p, q}[\ell]$, so $G_{\mathbf{Q}}
Z D$  generates $\mathcal{J}_{p, q}[\ell^{k}] / \ell \mathcal{J}_{p,
q}[\ell^{k}]$. By Nakayama's Lemma, the only subgroup of $\mathcal{J}_{p,
q}[\ell^{k}]$ which generates $\mathcal{J}_{p, q}[\ell^{k}] / \ell
\mathcal{J}_{p, q}[\ell^{k}]$ is $\mathcal{J}_{p, q}[\ell^{k}]$, so we are done.

\item \label{LemmaProof:GenerateZGZEllpqGZ}
Without loss of generality, assume that $\ell = q$. 

Suppose that $k = 1$. By \autoref{Corollary:UniqueRepresentationJ1mZ}, we
may express $D = \sum_{i = 0}^{p - 1} a_{i} [(0, \zeta_{p}^{i}) - \infty]$ for
$a_{i} \in \mathbf{Z} / q \mathbf{Z}$ such that $a_{0} + \cdots + a_{p - 1} =
0$. By applying an appropriate power of $\zeta_{p}$ to $D$, we may assume that
$a_{0} \not\equiv 0 \pmod{q}$. Let $g \in G_{\mathbf{Q}}$ restrict to a
generator of $\Gal(\mathbf{Q}(\mu_{p}) / \mathbf{Q})$. Then 
\begin{align}
(g + g^{2} + \cdots + g^{p - 1}) D &= (p - 1) a_{0} [(0, 1) -
\infty] + \left( \sum_{i = 1}^{p - 1} a_{i} \right) \sum_{j = 1}^{p - 1} [ (0,
\zeta_{p}^{j}) - \infty ] \nonumber\\
&= (p - 1) a_{0} [(0, 1) - \infty] + (-a_{0})\left( - [(0, 1) - \infty] \right)
\nonumber \\
&= p a_{0} [(0, 1) - \infty]. \label{Equation:AlmostNormOfD}
\end{align}
Since $p a_{0} \not\equiv 0 \pmod{q}$, \eqref{Equation:AlmostNormOfD} implies
that $[(0, 1) - \infty]$ lies in the subgroup generated by $G_{\mathbf{Q}} Z_{p}
D$. Applying elements of $Z_{p}$ shows that each $[(0, \zeta_{p}^{i}) - \infty]$
also lies the subgroup generated by $G_{\mathbf{Q}} Z_{p} D$, and these generate
$\mathcal{J}_{p, q}[1 - \zeta_{q}]$.

The proof of the $k \ge 1$ case is similar to the one in
\autoref{Lemma:GenerateZGZEllNotpqGZ}.  \qedhere
\end{enumerate}
\end{proof}

\begin{corollary}
\label{Corollary:IrreducibleGZField}
Suppose that $\ell$ is a prime and $k \ge 1$.
\begin{enumerate}[label=\upshape(\arabic*),
ref=\autoref{Corollary:IrreducibleGZField}(\arabic*)]

\item \label{Corollary:IrreducibleGZFieldNotpq}
Suppose that $\ell \not\in \{ p, q \}$ and $D \in \mathcal{J}_{p, q}[\ell^{k }]
\setminus \mathcal{J}_{p, q}[\ell^{k - 1}]$. Then the Galois closure of $E(D)$
over $\mathbf{Q}$ equals $E(\mathcal{J}_{p, q}[\ell^{k}])$.

\item \label{Corollary:IrreducibleGZFieldpq}
Suppose that $\ell \in \{ p, q \}$ and $D \in \mathcal{J}_{p, q}[(1 -
\zeta_{\ell})^{k}] \setminus\mathcal{J}_{p, q}[(1 - \zeta_{\ell})^{k - 1}]$.
Then the Galois closure of $E(D)$ over $\mathbf{Q}$ equals $E(\mathcal{J}_{p,
q}[(1 - \zeta_{\ell})^{k}])$.
\end{enumerate}
\end{corollary}
\begin{proof}
This follows immediately from \autoref{Lemma:GenerateZGZEllNotpqGZ} and
\autoref{Lemma:GenerateZGZEllpqGZ}.
\end{proof}

\begin{corollary}
\label{Corollary:Lk1Whenq3}
Suppose that $p \ge 5$, $q = 3$, and $k \in [1, p - 1]$. Suppose that $D \in
\mathcal{J}_{p, q}[(1 - \zeta_{p})^{k}] \setminus\mathcal{J}_{p, q}[(1 -
\zeta_{p})^{k - 1}]$.  Then $E(D) = L_{k, 1}$.
\end{corollary}
\begin{proof}
Induct on $k$. The case $k = 1$ follows since $\mathcal{J}[1 - \zeta_{p}]$ is
already defined over $E$.

Now suppose the assertion holds for $k - 1$ and that $D \in \mathcal{J}_{p,
q}[(1 - \zeta_{p})^{k}] \setminus\mathcal{J}_{p, q}[(1 - \zeta_{p})^{k - 1}]$.
By the inductive hypothesis, $L_{k - 1, 1} = E((1 - \zeta_{p}) D)$, so $L_{k -
1, 1} = E((1 - \zeta_{p}) D) \subseteq E(D) \subseteq L_{k, 1}$. By
\autoref{Corollary:Lk1DegreeUpperBound}, either $E(D) = L_{k, 1}$ (in which case
we are done) or that $E(D) = L_{k - 1, 1}$. If $E(D) = L_{k - 1, 1}$, then
$E(D)$ is Galois over $\mathbf{Q}$, so \autoref{Corollary:IrreducibleGZFieldpq}
gives $E(D) = E(\mathcal{J}_{p, q}[(1 - \zeta_{\ell})^{k}]) = L_{k, 1}$, so we
are again done.
\end{proof}

\begin{lemma}
\label{Lemma:HowToKillTorsion}
Let $i \ge 0$ be an integer and $\gamma \in \mathbf{Z}_{p}\left[ \Gal\left(
E(\mathcal{J}[p^{\infty}]) / E \right) \right]$. Suppose that $\gamma - 1$ kills
$\mathcal{J}[(1 - \zeta_{p})^{i}]$. Then
\begin{enumerate}[label=\upshape(\arabic*),
ref=\autoref{Lemma:HowToKillTorsion}(\arabic*)]

\item \label{Lemma:HowToKillTorsionPower}
for any integer $k \ge 0$, $(\gamma - 1)^{k}$ kills $\mathcal{J}_{p, q}[(1 -
\zeta_{p})^{i k}]$;

\item \label{Lemma:HowToKillTorsionGeoSum}
$\gamma^{p - 1} + \gamma^{p - 2} + \dots + 1$ kills $\mathcal{J}[p] =
\mathcal{J}_{p, q}[(1 - \zeta_{p})^{p - 1}]$;

\item \label{Lemma:HowToKillTorsionPpower}
$\gamma^{p} - 1$ kills $\mathcal{J}_{p, q}[(1 - \zeta_{p})^{p - 1 + i}]$.
\end{enumerate}
\end{lemma}
\begin{proof}
Since $\mathcal{J}_{p, q}[p^{\infty}] \supseteq \mathcal{J}[1 - \zeta_{p}]$ and
the field of definition of $\mathcal{J}_{p, q}[1 - \zeta_{p}]$ is
$\mathbf{Q}(\mu_{q})$, 
\[
E(\mathcal{J}_{p, q}[p^{\infty}]) = \mathbf{Q}(\mu_{p}, \mathcal{J}_{p,
q}[p^{\infty}]).
\]
So now this lemma follows from \autoref{Lemma:CreateKillerElements}.
\end{proof}

\section{Torsion points on \texorpdfstring{$\mathcal{C}_{n, d}$}{C\_\{n,d\}}}
\label{Section:TorsionPointsSuperellipticCatalan}

\begin{definition}
Let $P \in \mathcal{C}_{n, d}$ be a torsion point.
\end{definition}

\subsection{Bounding the order of torsion points on
\texorpdfstring{$\mathcal{C}_{n, d}$}{C\_\{n,d\}}}

\begin{corollary}
\label{Corollary:TauConstruction2}
Let $m$ be a positive integer. For every prime $\ell$ dividing $m$, let
$c_{\ell} \in \mathbf{Z}_{\ell}^\times$; if $\ell | n d$, assume that $c_{\ell}
\in 1 + n d \mathbf{Z}_{\ell}$.  Then there exists an element $\tau$ of
$\Gal\left(  E(\mathcal{J}_{n, d}[m^{\infty}]) / E \right)$ such that for each
$\ell$ dividing $m$, $\tau$ acts on $\mathcal{J}[\ell^{\infty}]$ as
multiplication by $c_{\ell}$.
\end{corollary}
\begin{proof}
The assumptions on $c_{\ell}$ imply that there exists an element of
$\gamma \in \Gal\left( E(\mu_{m^{\infty}}) / E \right)$ such that for each prime
$\ell$ dividing $m$, $\gamma$ acts on $\mu_{\ell^{\infty}}$ as multiplication by
$c_{\ell}$. Lift $\gamma$ arbitrarily to $h \in \Gal(E(\mathcal{J}_{n,
d}[m^{\infty}]) / E)$, define $\tau \colonequals h \cdot {}^{\gamma}\! h$, and
apply \autoref{Corollary:TauConstruction1.5} to finish.
\end{proof}

\begin{proposition}
\label{Proposition:2pqOrderTorsionExtension}
Suppose that $\mathcal{C}_{n, d}$ has genus $g > 1$ (i.e., $(n, d) \not\in \{
(2, 3), (3, 2) \}$). Let $m = \lcm(2, n d)$. 
\begin{enumerate}[label=\upshape(\roman*),
ref=\autoref{Proposition:2pqOrderTorsionExtension}(\roman*)]

\item \label{Proposition:2pqOrderTorsionExtensionm}
If $(n, d) \not \in \{ (2, 5), (4, 5), (5, 2), (5, 4) \}$ then $m (P -
\infty) \sim 0$.

\item \label{Proposition:2pqOrderTorsionExtension3m}
If $(n, d) \in \{ (2, 5), (5, 2), (4, 5), (5, 4) \}$ then $3 m (P - \infty) \sim
0$.

\end{enumerate}
\end{proposition}
\begin{proof}
Without loss of generality, assume that $d$ is odd.  Choose an integer $M$ such
that $M(P - \infty) \sim 0$. Assume that $M$ is divisible by $m$. Define
\begin{align*}
\mathcal{R} &\colonequals \{ \text{prime } r \colon r \nmid 2 n d
\text{ and } r | M \}, \\
\mathcal{S} &\colonequals \{ \text{prime } s \colon s \nmid 2 \text{ and } s
\mid n d \}, 
\end{align*}
so that the set of primes dividing $M$ is the disjoint union $\{ 2 \} \cup
\mathcal{R} \cup \mathcal{S}$. For any prime $a$ dividing $M$, let $e_{a}$ be
the largest integer such that $a^{e_{a}} | n d$ and define $D_{a} \in
\mathcal{J}_{n, d}[a^{\infty}]$ such that
\[
[P - \infty] = \sum_{a | M} D_{a}.
\]
By definition of $m$, 
\begin{equation}
\label{Equation:PrimeFactorizationm}
m = 2^{\max \{ 1, e_{2} \}} \left(
\prod_{s \in \mathcal{S}} s^{e_{s}} \right).
\end{equation}
Using \autoref{Corollary:TauConstruction2}, choose $\tau_{1}, \tau_{2},
\tau_{3} \in \Gal(E(\mathcal{J}_{n, d}[M^{\infty}]) / E)$ such that:
\begin{align*}
\tau_{1} &\text{ acts on } \begin{cases}
\mathcal{J}_{n, d}[2^{\infty}] &\text{ as multiplication by } 1 + 2^{\max \{ 1,
e_2 \}} \\
\mathcal{J}_{n, d}[r^{\infty}] &\text{ as multiplication by } 2 \text{ for each
} r \in \mathcal{R} \\
\mathcal{J}_{n, d}[s^{\infty}] &\text{ as multiplication by } 1 + s^{e_{s}}
\text{ for each } s \in \mathcal{S}
\end{cases} \\
\tau_{2} &\text{ acts on } \begin{cases}
\mathcal{J}_{n, d}[2^{\infty}] &\text{ as multiplication by } 1 - 2^{\max \{ 1,
e_2 \}} \\
\mathcal{J}_{n, d}[r^{\infty}] &\text{ as multiplication by } -2 \text{ for each
} r \in \mathcal{R} \\
\mathcal{J}_{n, d}[s^{\infty}] &\text{ as multiplication by } 1 - s^{e_{s}}
\text{ for each }s \in \mathcal{S}
\end{cases} \\
\tau_{3} &\text{ acts on } \begin{cases}
\mathcal{J}_{n, d}[2^{\infty}] &\text{ as multiplication by } 1 \\
\mathcal{J}_{n, d}[r^{\infty}] &\text{ as multiplication by } -1 \text{ for
all } r \in \mathcal{R} \\
\mathcal{J}_{n, d}[s^{\infty}] &\text{ as multiplication by } 1 \text{ for
all }s \in \mathcal{S}
\end{cases}
\end{align*}
By construction, 
\begin{equation}
\label{Equation:Tau1Tau2Tau3P}
\tau_1 P + \tau_2 P \sim \tau_3 P + P
\end{equation}

\

\begin{enumerate}[label=\textbf{Case~\Alph*:}, ref={Case~\Alph*}, leftmargin=*,
itemindent=35pt]
\item $\{ \tau_1 P, \tau_2 P \} \neq \{ \tau_3 P, P\}$

By \eqref{Equation:Tau1Tau2Tau3P}, this means that there exists a map
$\upsilon_{1} \colon \mathcal{C}_{n, d} \to \mathbf{P}^{1}$ of degree $h \le 2$.
Since $d$ is odd, we may apply \autoref{Corollary:FollowFromCS} applied
with $\upsilon_{1}$ and the $y$-map to obtain
\[
(n - 1)(d - 1) / 2 \le (h - 1)(d - 1).
\]
Since $d > 1$, this implies that $n - 1 \le 2 (h - 1)$. Since $h \le 2$ and $n
\ge 2$, this implies that $h = 2$ and $n \in \{ 2, 3 \}$.

\begin{enumerate}[label=\textbf{\theenumi\arabic*:}, ref={\theenumi\arabic*},
leftmargin=*, itemindent=35pt]

\item $n = 3$ and $h = 2$

By \autoref{Corollary:FollowFromCS} applied with $\upsilon_{1}$ and the
$x$-map, we obtain
\[
(3 - 1)(d - 1) / 2 \le (2 - 1)(3 - 1),
\]
which forces $d \le 3$, contradicting the assumption that $n$ and $d$ are
coprime.

\item $n = 2$ and $h = 2$

This curve is hyperelliptic of genus at least 2, so any $2$-to-$1$ map to
$\mathbf{P}^{1}$ must factor through the canonical map. Applying this fact to
$\upsilon_{1}$ yields $\tau_1 P + \tau_2 P \sim 2 \infty$ and $\tau_3 P + P \sim
2 \infty$, so by definition of $\tau_{3}$,
\begin{align}
2 D_{2} &= 0 \label{Equation:A12D2Dead} \\
D_{s} &= 0 &&\text{ for all } s \in \mathcal{S} \label{Equation:A1DsDead}.
\end{align}
Using \autoref{Corollary:TauConstruction2}, choose $\tau_{4}, \tau_{5} \in
\Gal(E(\mathcal{J}_{n, d}[M^{\infty}]) / E)$ such that:
\begin{align*}
\tau_{4} &\text{ acts on } \begin{cases}
\mathcal{J}_{n, d}[2^{\infty}] &\text{ as multiplication by } 1 \\
\mathcal{J}_{n, d}[r^{\infty}] &\text{ as multiplication by } 1 + 3 \text{ for
each } r \in \mathcal{R} \cap \{ 3 \} \\
\mathcal{J}_{n, d}[r^{\infty}] &\text{ as multiplication by } 3 \text{ for each
} r \in \mathcal{R} \setminus \{ 3 \}
\end{cases} \\
\tau_{5} &\text{ acts on } \begin{cases}
\mathcal{J}_{n, d}[2^{\infty}] &\text{ as multiplication by } 1\\
\mathcal{J}_{n, d}[r^{\infty}] &\text{ as multiplication by } 1 - 3 \text{ for
each } r \in \mathcal{R} \cap \{ 3 \} \\
\mathcal{J}_{n, d}[r^{\infty}] &\text{ as multiplication by } -1 \text{ for each
} r \in \mathcal{R} \setminus \{ 3 \}
\end{cases}
\end{align*}
By construction, 
\begin{equation}
\label{Equation:Tau4Tau52P}
\tau_4 P + \tau_5 P \sim 2P.
\end{equation}
\begin{enumerate}[label=\textbf{\theenumii\alph*:}, ref={\theenumii\alph*},
leftmargin=*, itemindent=45pt]
\item $\tau_{4} P \neq P$

If $\tau_{5} P = P$, then \eqref{Equation:Tau4Tau52P} would imply that
$\mathcal{C}_{n, d}$ has a degree 1 map to $\mathbf{P}^{1}$, which contradicts
the assumption that the genus of $\mathcal{C}_{n, d}$ is at least 2. Therefore,
$P \not\in \{ \tau_{4} P, \tau_{5} P \}$, and \eqref{Equation:Tau4Tau52P} gives
a $2$-to-$1$ map to $\mathbf{P}^{1}$. As before, such a map must factor through
the canonical map, so $\tau_{4} P + \tau_{5} P \sim 2 P \sim 2 \infty$. Hence $2
[P - \infty] = 0$, so the conclusion of the proposition holds.

\item $\tau_{4} P = P$

Then by definition of $\tau_{4}$, 
\begin{align}
3 D_{3} &= 0 &&\text{if } 3 \in \mathcal{R}, \label{Equation:A1b3D3Dead} \\
D_{r} &= 0 &&\text{for all } r \in \mathcal{R} \setminus \{ 3 \}
\label{Equation:A1bDrDead}.
\end{align}

Suppose that $3 \not\in \mathcal{R}$. Then we are done because
\eqref{Equation:A12D2Dead}, \eqref{Equation:A1DsDead}, and
\eqref{Equation:A1bDrDead} together imply that $2[P - \infty] = 0$.

Suppose that $3 \in \mathcal{R}$. Then \eqref{Equation:A12D2Dead},
\eqref{Equation:A1DsDead}, \eqref{Equation:A1b3D3Dead}, and
\eqref{Equation:A1bDrDead} together imply that 
\begin{equation}
\label{Equation:A1b6PDead}
6[P - \infty] = 0,
\end{equation}
so using $\iota$ to denote the hyperelliptic involution yields
\begin{equation}
\label{Equation:3P3iP}
3 P \sim 3 \iota P.
\end{equation}
If $P = \iota P$, then $2 [P - \infty] = 0$ and the conclusion of the
proposition holds. If $P \neq \iota P$, then \eqref{Equation:3P3iP} yields a
nonconstant $3$-to-$1$ map $\upsilon_{2} : \mathcal{C}_{n, d} \to
\mathbf{P}^{1}$, so applying \autoref{Corollary:FollowFromCS} with
$\upsilon_{2}$ and the $x$-map yields $(2 - 1)(d - 1) / 2 \le (3 - 1)(2 - 1)$,
forcing $d \le 5$, so by \eqref{Equation:A1b6PDead}, the conclusion of the
proposition holds.
\end{enumerate}
\end{enumerate}

\item $P = \tau_{1} P$

Then by definition of $\tau_{1}$, 
\begin{align}
2^{\max \{ 1, e_{2} \}} D_{2} &= 0, \label{Equation:BmD2Died}\\
D_{r} &= 0 &&\text{for all } r \in \mathcal{R}, \label{Equation:BDrDied} \\
s^{e_{s}} D_{s} &= 0 &&\text{for all } s \in \mathcal{S},
\label{Equation:BmDsDied}
\end{align}
so \eqref{Equation:PrimeFactorizationm}, \eqref{Equation:BmD2Died},
\eqref{Equation:BDrDied}, and \eqref{Equation:BmDsDied} together imply $m D_{2}
= 0$, $m D_{r} = 0$ for all $r \in \mathcal{R}$, and $m D_{s} = 0$ for all $s
\in \mathcal{S}$. We conclude that $m[P - \infty] = 0$.

\item $P = \tau_{2} P$

Then by definition of $\tau_{2}$, 
\begin{align}
2^{\max \{ 1, e_{2} \}} D_{2} &= 0, \label{Equation:CmD2Died}\\
3 D_{3} &= 0 &&\text{for all } r \in \mathcal{R} \cap \{ 3 \},
\label{Equation:C3D3Died} \\
D_{r} &= 0 &&\text{for all } r \in \mathcal{R} \setminus \{ 3 \},
\label{Equation:CDrDied} \\
s^{e_{s}} D_{s} &= 0 &&\text{for all } s \in \mathcal{S}.
\label{Equation:CmDsDied}
\end{align}

\begin{enumerate}[label=\textbf{\theenumi\arabic*:}, ref={\theenumi\arabic*},
leftmargin=*, itemindent=35pt]
\item $3 \not\in \mathcal{R}$
\label{Case:Tau2FixesAnd3NotInR}

Then \eqref{Equation:PrimeFactorizationm}, \eqref{Equation:CmD2Died},
\eqref{Equation:CDrDied}, and \eqref{Equation:CmDsDied} together imply $m D_{2}
= 0$, $m D_{r} = 0$ for all $r \in \mathcal{R}$, and $m D_{s} = 0$ for all $s
\in \mathcal{S}$. We conclude that $m[P - \infty] = 0$.

\item $3 \in \mathcal{R}$

Arguing similarly as in \autoref{Case:Tau2FixesAnd3NotInR} yields 
\begin{equation}
\label{Equation:C23mKillsP}
3 m[P - \infty] = 0. 
\end{equation}
Using \autoref{Corollary:TauConstruction2}, choose $\tau_{6} \in
\Gal(E(\mathcal{J}_{n, d}[M^{\infty}]) / E)$ such that:
\[
\tau_{6} \text{ acts on } \begin{cases}
\mathcal{J}_{n, d}[(2 n d)^{\infty}] &\text{ as multiplication by } 1 \\
\mathcal{J}_{n, d}[3^{\infty}] &\text{ as multiplication by } 2
\end{cases}
\]
By definition, $\tau_{6}$ fixes $D_{2}$ and $D_{s}$ for all $s \in \mathcal{S}$,
so by \eqref{Equation:C3D3Died} and \eqref{Equation:CDrDied}, $\tau_{6}$ must
fix $3 P$, i.e.,
\begin{equation}
\label{Equation:3P3Tau6P}
3 P \sim 3 \tau_{6} P.
\end{equation}

\begin{enumerate}[label=\textbf{\theenumii\alph*:}, ref={\theenumii\alph*},
leftmargin=*, itemindent=45pt]
\item $P = \tau_{6} P$

Since $\tau_{6}$ acts on $D_{3}$ as multiplication by $2$, this forces
$D_{3} = 0$. Combining this with \eqref{Equation:PrimeFactorizationm},
\eqref{Equation:CmD2Died}, \eqref{Equation:CDrDied}, and
\eqref{Equation:CmDsDied} yields $m[P - \infty] = 0$.

\item $P \neq \tau_{6} P$

Then \eqref{Equation:3P3Tau6P} yields a nonconstant $3$-to-$1$ map $\upsilon_{3}
\colon \mathcal{C}_{n, d} \to \mathbf{P}^{1}$. Then $3 \in \mathcal{R}$ implies
$3 \nmid n d$, so $3 \nmid \min \{ n, d \}$, meaning we may apply
\autoref{Corollary:FollowFromCS} to $\upsilon_{3}$ and to whichever of $y \colon
\mathcal{C}_{n, d} \to \mathbf{P}^{1}$, $x \colon \mathcal{C}_{n, d} \to
\mathbf{P}^{1}$ has smaller degree to obtain 
\[
(n - 1)(d - 1) / 2 \le (3 - 1)(\min \{ n, d \} - 1),
\]
which forces $n, d \le 5$. Since $3 \nmid n d$ and $d$ is odd, this implies $(n,
d) \in \{ (2, 5), (4, 5) \}$, so we are done by \eqref{Equation:C23mKillsP}.
\qedhere
\end{enumerate}
\end{enumerate}
\end{enumerate}
\end{proof}

\subsection{Classification of torsion points on \texorpdfstring{$\mathcal{C}_{n,
d}$}{C\_\{n,d\}}}
\label{Section:ClassifyTorsionPointsCnd}

\subsubsection{The case \texorpdfstring{$\mathcal{C}_{p, q}$}{C\_\{p,q\}} for
distinct odd primes \texorpdfstring{$p, q$}{p, q}}
\label{Subsubsection:Non-hyperelliptic}

Let $p$ and $q$ be distinct odd primes.

\begin{definition}
Suppose that $m \ge 1$ is an integer coprime to $p$ and $q$ and that $a, b \ge
0$ are integers. Say that $D \in \mathcal{J}_{p, q}(\mathbf{C})$ is of
\textit{exact order} $(1 - \zeta_{p})^{a} (1 - \zeta_{q})^{b} m$ if
\begin{align*}
D &\in \mathcal{J}_{p, q}[(1 - \zeta_{p})^{a} (1 - \zeta_{q})^{b} m] \\
D &\not\in \mathcal{J}_{p, q}[(1 - \zeta_{p})^{a - 1} (1 - \zeta_{q})^{b} m] \\
D &\not\in \mathcal{J}_{p, q}[(1 - \zeta_{p})^{a} (1 - \zeta_{q})^{b - 1} m],
\end{align*}
and for all $m' | m$ such that $m' \neq m$, we have
\begin{align*}
D &\not\in \mathcal{J}_{p, q}[(1 - \zeta_{p})^{a} (1 - \zeta_{q})^{b} m'].
\end{align*}
\end{definition}

\begin{lemma}
\label{Lemma:ExactOrder11Stabilizer} Suppose that $D \in \mathcal{J}_{n,
d}(\mathbf{C})$.
\begin{enumerate}[label=\upshape(\arabic*),
ref=\autoref{Lemma:ExactOrder11Stabilizer}(\arabic*)]

\item \label{Lemma:ExactOrder11Stabilizer1mzp}
Suppose that $D$ has exact order $(1 - \zeta_{p})$. Then the stabilizer of
$D$ in $Z$ is $Z_{p}$.

\item \label{Lemma:ExactOrder11Stabilizer1mzq}
Suppose that $D$ has exact order $(1 - \zeta_{q})$. Then the stabilizer of
$D$ in $Z$ is $Z_{q}$.

\item \label{Lemma:ExactOrder11Stabilizer1mzp1mzq}
Suppose that $D$ has exact order $(1 - \zeta_{p})(1 - \zeta_{q})$. Then
the stabilizer of $D$ in $Z$ is $\{ 1 \}$.
\end{enumerate}
\end{lemma}
\begin{proof} \hfill
\begin{enumerate}[label=\upshape(\arabic*),
ref={the proof of \autoref{Lemma:ExactOrder11Stabilizer}(\arabic*)}]

\item \label{LemmaProof:ExactOrder11Stabilizer1mzp}
The stabilizer contains $Z_{p}$. If it were any larger, then it would be
$Z$. That would imply that $D$ is fixed by $\zeta_{q}$, so $q D = 0$. Since $D$
is fixed by $\zeta_{p}$, we also have $p D = 0$. Together, these imply that $D =
0$, a contradiction.

\item \label{LemmaProof:ExactOrder11Stabilizer1mzq}
Similar to the proof of the previous part.

\item \label{LemmaProof:ExactOrder11Stabilizer1mzp1mzq}
The stabilizer of $D$ is contained in the stabilizer of $(1 - \zeta_{p})
D$ and $(1 - \zeta_{q}) D$, so the previous two parts imply that it is contained
in $Z_{p} \cap Z_{q} = \{ 1 \}$. \qedhere

\end{enumerate}
\end{proof}

\begin{lemma}
\label{Lemma:Ramification2Torsion}
Suppose that $D \in \mathcal{J}_{p, q}[2] \setminus \{ 0 \}$. 
\begin{enumerate}[label=\upshape(\arabic*),
ref=\autoref{Lemma:Ramification2Torsion}(\arabic*)]

\item \label{Lemma:Ramification2TorsionRamified}
The extension $E(D) / E$ is ramified at some prime above 2.

\item \label{Lemma:Ramification2TorsionNontrivial}
The extension $E(D, \mathcal{J}_{p, q}[p q]) / E(\mathcal{J}_{p, q}[p q])$ is
nontrivial.
\end{enumerate}
\end{lemma}
\begin{proof} \hfill
\begin{enumerate}[label=\upshape(\arabic*),
ref={the proof of \autoref{Lemma:Ramification2Torsion}(\arabic*)}]

\item \label{LemmaProof:Ramification2TorsionRamified}
Suppose for contradiction that $E(D) / E$  is unramified at every prime above
$2$. Since $E / \mathbf{Q}$ is also unramified at every prime above $2$, we see
that $E(D) / \mathbf{Q}$ is unramified at every prime above $2$. Hence
$E(\mathcal{J}_{p, q}[2]) / \mathbf{Q}$, which is the Galois closure of $E(D) /
\mathbf{Q}$ by \autoref{Corollary:IrreducibleGZFieldNotpq}, must also be
unramified at every prime above $2$, so Lemma 1.4 of \cite{gross1978some}
implies that the mod 2 reduction of $\mathcal{J}_{p, q}$ must be ordinary, which
contradicts Lemma 4.2 of \cite{jkedrzejak2016note}.

\item \label{LemmaProof:Ramification2TorsionNontrivial}
Since $\mathcal{C}_{p, q}$ has good reduction at $2$ and $2$ is coprime to
$p q$, the extension $E(\mathcal{J}_{p, q}[p q]) / E$ is unramified at every
prime above 2, so we are done by \autoref{Lemma:Ramification2TorsionRamified}.
\qedhere
\end{enumerate}
\end{proof}

\begin{lemma}
\label{Lemma:2pqOrderTorsion}
$2 p q [P - \infty] = 0$.
\end{lemma}
\begin{proof}
This is a special case of \autoref{Proposition:2pqOrderTorsionExtensionm}.
\end{proof}

\begin{definition}
By \autoref{Lemma:2pqOrderTorsion}, there exist $a \in [0, p -
1]$, $b \in [0, q - 1]$, and $c \in \{ 1, 2 \}$ such that $P$ has exact order
$(1 - \zeta_{p})^{a} (1 - \zeta_{q})^{b} c$. Define $D_{p}$, $D_{q}$, $D_{2}$ to
have exact order $(1 - \zeta_{p})^{a}$, $(1 - \zeta_{q})^{b}$, and $c$
respectively, such that $[P - \infty] = D_{p} + D_{q} + D_{2}$.
\end{definition}

\begin{proposition}
$p q [P - \infty] = 0$.
\end{proposition}
\begin{proof}
Suppose not. Then $D_{2} \neq 0$, so
\autoref{Lemma:Ramification2TorsionNontrivial} implies that there exists a
nontrivial $\tau \in \Gal( \overline{E} / E(\mathcal{J}_{p, q}[p q]))$ which
moves $D_{2}$, so $\tau P \neq P$. Since $\tau$ fixes $D_{p}$ and $D_{q}$, we
see that $2[P - \tau P] = 2 (D_{2} - \tau D_{2}) = 0$, which violates
\autoref{Lemma:NotHyperelliptic}.  
\end{proof}

\begin{lemma}
\label{Lemma:EDpFactsForArgument}
Suppose that $a \ge 2$. Then
\begin{enumerate}[label=\upshape(\arabic*),
ref=\autoref{Lemma:EDpFactsForArgument}(\arabic*)]

\item \label{Lemma:EDpFactsForArgumentEdpEp}
$[E(D_{p}) : E] \ge p$;

\item \label{Lemma:EDpFactsForArgumentEdpq3}
if $q = 3$, then $E(D_{p}) = L_{a, 1}$;

\item \label{Lemma:EDpFactsForArgumentEdpBigq3}
if $q = 3$ and $a \ge 4$, then $[E(D_{p}) : E] \ge p^{2}$.
\end{enumerate}
\end{lemma}
\begin{proof} \hfill
\begin{enumerate}[label=\upshape(\arabic*),
ref={the proof of \autoref{Lemma:EDpFactsForArgument}(\arabic*)}]

\item \label{LemmaProof:EDpFactsForArgumentEdpEp}
Since $E(D_{p})$ is a subfield of $L_{a, 1}$ and \autoref{Lemma:LijFactsLp1}
implies that $[L_{a, 1} : E]$ is a power of $p$, $[E(D_{p}) : E]$ is also a
power of $p$. If $E(D_{p}) = E$, then taking Galois closure of both sides (over
$\mathbf{Q}$) and applying \autoref{Corollary:IrreducibleGZFieldpq} yields
$L_{a, 1} = E$, which contradicts \autoref{Lemma:LijFactsL21}.

\item \label{LemmaProof:EDpFactsForArgumentEdpq3}
This follows from \autoref{Corollary:Lk1Whenq3}.

\item \label{LemmaProof:EDpFactsForArgumentEdpBigq3}
This follows from \autoref{Lemma:EDpFactsForArgumentEdpq3} and
\autoref{Corollary:L41BigEnough}. \qedhere

\end{enumerate}
\end{proof}

\begin{definition}
Let $G_{E}$ denote the absolute Galois group of $E$.
\end{definition}

\begin{lemma} 
\label{Lemma:SFSize}
Suppose that $a, b \ge 1$.
\begin{enumerate}[label=\upshape(\arabic*),
ref=\autoref{Lemma:SFSize}(\arabic*)]

\item \label{Lemma:SFSizeSp}
$\# G_{E} Z P \ge p q [E(D_{p}) : E]$. 

\item \label{Lemma:SFSizeSq}
$\# G_{E} Z P \ge p q [E(D_{q}) : E]$.
\end{enumerate}
\end{lemma}
\begin{proof}
Both parts are similar so we prove the first. Suppose that $h \in
\Gal(\overline{E} / E)$ and $z \in Z$ satisfy $h z P =  P$.  Since the action of
$\Gal(\overline{E} / E)$ and $Z$ commute, we must check that $h$ fixes $D_{p}$
and $z = 1$.

Then $(1 - \zeta_{p})^{a - 1} (1 - \zeta_{p})^{b - 1} P$ has exact order $(1 -
\zeta_{p})(1 - \zeta_{q})$ and is fixed by $h z$. It is also fixed by $h$ since
$E = L_{1, 1}$. Therefore, it is fixed by $z$.
\autoref{Lemma:ExactOrder11Stabilizer1mzp1mzq} implies that $z = 1$, so $h$
fixes $P$, and hence $h$ also fixes $D_{p}$.
\end{proof}

\begin{lemma} \hfill
\label{Lemma:ab2Kill}
\begin{enumerate}[label=\upshape(\arabic*),
ref=\autoref{Lemma:ab2Kill}(\arabic*)]
\item \label{Lemma:ab2Killa2}
Suppose that $a \ge 2$. Then there exists $h \in \Gal\left( \overline{E} /
E \right)$ which moves $P$ such that $h - 1$ kills $\mathcal{J}_{p, q}[(q (1 -
\zeta_{p})]$.

\item \label{Lemma:ab2Killb2}
Suppose that $b \ge 2$. Then there exists $h \in \Gal\left( \overline{E} /
E \right)$ which moves $P$ such that $h - 1$ kills $\mathcal{J}_{p, q}[p (1 -
\zeta_{q})]$.
\end{enumerate}
\end{lemma}
\begin{proof}
Both parts are similar so we show the first. Since $a \ge 2$, the extension 
$E(D_{p}) / E$ is nontrivial by \autoref{Lemma:EDpFactsForArgumentEdpEp} and it
is disjoint from $E(\mathcal{J}_{p, q}[q]) / E$ by
\autoref{Lemma:LijDisjointComposite}. Therefore, there exists $h \in
\Gal(\overline{E} / E)$ which moves $D_{p}$ (and hence, moves $P$) such that $h
- 1$ kills $\mathcal{J}_{p, q}[q]$. Since $E = L_{1, 1}$ by
\autoref{Lemma:LijFactsL11}, we know that $h - 1$ also kills $\mathcal{J}_{p,
q}[1 - \zeta_{p}]$.
\end{proof}

Recall the definitions of $\WM(Q)$ and $\wt(Q)$ in
\autoref{Definition:WeierstrassMonoid} and
\autoref{Definition:WeierstrassWeightPoint}.
\begin{lemma}
\label{Lemma:ManyHighWeightPoints2} \hfill
\begin{enumerate}[label=\upshape(\arabic*),
ref=\autoref{Lemma:ManyHighWeightPoints2}(\arabic*)]
\item \label{Lemma:ManyHighWeightPoints2a2b1}
Suppose that $a \ge 2$ and $b \ge 1$. Then for each $Q \in S_{P}$, we have
$p - 1, p \in \WM(Q)$.

\item \label{Lemma:ManyHighWeightPoints2a1b2}
Suppose that $a \ge 1$ and $b \ge 2$. Then for each $Q \in S_{P}$, we have
$q - 1, q \in \WM(Q)$.
\end{enumerate}
\end{lemma}
\begin{proof}
Both parts are similar so we show the first. Since $\WM(Q) = \WM(P)$ for each $Q
\in S_{P}$, we will show that $p - 1, p \in \WM(P)$.

By \autoref{Lemma:ab2Killa2}, there exists $h \in \Gal(\overline{E} / E)$ which
moves $P$ such that $h - 1$ kills $\mathcal{J}_{p, q}[q (1 - \zeta_{p})]$. By
\autoref{Lemma:HowToKillTorsionPpower}, $h^{p} - 1$ kills $\mathcal{J}_{p, q}[p
q]$, so $h^{p} P = P$ and hence
\begin{equation}
\label{Equation:DistinctImagesh}
h^{i} P \neq P\text{ for }1 \le i \le p - 1.
\end{equation}

Since $h - 1$ kills $\mathcal{J}_{p, q}[q] \ni p P$, we have $p P \sim p (h P)$,
so by \eqref{Equation:DistinctImagesh}, $p \in \WM(P)$. 

Since $h - 1$ kills $\mathcal{J}_{p, q}[1 - \zeta_{p}]$,
\autoref{Lemma:HowToKillTorsionGeoSum} gives that $1 + h + \dots + h^{p - 1}$
kills $\mathcal{J}_{p, q}[p]$, which combined with the fact that $h - 1$ kills
$\mathcal{J}_{p, q}[q]$ shows that $(1 + h + \dots + h^{p - 1}) - p$ kills
$\mathcal{J}_{p, q}[p q] \ni P$, so 
\[
h P + h^{2} P + \cdots + h^{p - 1} P \sim (p - 1) P,
\]
and hence by \eqref{Equation:DistinctImagesh}, $p - 1 \in \WM(P)$.
\end{proof}

\begin{lemma}
\label{Lemma:WeightLowerBound}
Suppose that $Q \in \mathcal{C}_{p, q}(\mathbf{C})$. 
\begin{enumerate}[label=\upshape(\arabic*),
ref=\autoref{Lemma:WeightLowerBound}(\arabic*)]

\item \label{Lemma:WeightLowerBoundp}
If $p - 1, p \in \WM(Q)$, then
\[
\wt(Q) \ge \begin{cases}
2 &\text{if } q = 3 \\
\displaystyle g \left(\frac{q - 3}{2}\right) &\text{if } q \ge 5. 
\end{cases}
\]

\item \label{Lemma:WeightLowerBoundq}
If $q - 1, q \in \WM(Q)$, then
\[
\wt(Q) \ge \begin{cases}
2 &\text{if } p = 3 \\
\displaystyle g \left(\frac{p - 3}{2}\right) &\text{if } p \ge 5. 
\end{cases}
\]

\end{enumerate}
\end{lemma}

\begin{proof}
Both parts are similar so we prove the first. Suppose that the gaps of $Q$ are
$k_{1} < k_{2} < \dots < k_{g}$. Since $\WM(Q)$ is a monoid and we assume that
$p - 1, p \in \WM(Q)$, 
\[
\{ p - 1, p, 2 p - 2, 2 p - 1, 2 p, 3 p - 3, 3 p - 2, 3 p - 1, 3 p, \cdots \}
\subseteq \WM(Q),
\]
so
\begin{equation}
\label{Equation:LowerBoundNormalizedGaps}
k_{i} - i \ge \begin{cases}
0 &\text{for } i \ge 1, \\
2 &\text{for } i \ge p - 1, \\
5 &\text{for } i \ge 2 p - 4, \\
9 &\text{for } i \ge 3 p - 8, \\
\vdots &\qquad \vdots
\end{cases}
\end{equation}
If $q = 3$, then $g = (p - 1)(q - 1) / 2 = p - 1$, so
\eqref{Equation:LowerBoundNormalizedGaps} implies
\[
\wt(Q) = \sum_{i = 1}^{g} (k_{i} - i) \ge 0 + 0 + \cdots + 0 + 2 = 2.
\]
If $q \ge 5$, then weakening the bounds in
\eqref{Equation:LowerBoundNormalizedGaps} yields
\[
k_{i} - i \ge \begin{cases}
0 &\text{for } i \ge 1, \\
2 &\text{for } i \ge p, \\
4 &\text{for } i \ge 2 p - 1, \\
6 &\text{for } i \ge 3 p - 2, \\
\vdots &\qquad \vdots
\end{cases}
\]
so 
\begin{align*}
\wt(Q) &= \sum_{i = 1}^{g} (k_{i} - i) \\
&\ge 0 \cdot (p - 1) + 2 \cdot (p - 1) + \cdots + 2 \left( \frac{g}{p - 1} - 1
\right)(p - 1) \\
&= g \left( \frac{g}{p - 1} - 1 \right) \\
&= g \left( \frac{q - 3}{2} \right) 
\end{align*}
since $g = (p - 1)(q - 1) / 2$. \qedhere
\end{proof}

\begin{proposition}
\label{Proposition:LargeTorsionPointImpossible} \hfill
\begin{enumerate}[label=\upshape(\arabic*),
ref=\autoref{Proposition:LargeTorsionPointImpossible}(\arabic*)]

\item \label{Proposition:LargeTorsionPointImpossiblea2b1}
If $a \ge 2$ and $b \ge 1$, then $q = 3$ and $a \in \{ 2, 3
\}$.

\item \label{Proposition:LargeTorsionPointImpossiblea1b2}
If $a \ge 1$ and $b \ge 2$, then $p = 3$ and $b \in \{ 2, 3
\}$.

\end{enumerate}
\end{proposition}
\begin{proof}
Both parts are similar so we prove the first. Using
\autoref{Theorem:SumWeierstrassWeights}, observe that
\begin{align}
g^{3} - g &= \sum_{Q \in \mathcal{C}_{p, q}(\mathbf{C})} \wt(Q) \nonumber \\
&\ge \sum_{Q \in G_{E} Z P} \wt(Q) \nonumber \\
&\ge \left(\# G_{E} Z P\right) \left(\min_{Q \in G_{E} Z P} \left( \wt(Q)
\right)\right) \nonumber \\
&\ge p q [E(D_{p}) : E] \left(\min_{Q \in G_{E} Z P} \left( \wt(Q)
\right)\right) \label{Equation:g3gSPminWt}
\end{align}
by \autoref{Lemma:SFSizeSp}.

If $q = 3$ and $a \ge 4$, then $g = (p - 1)(q - 1) / 2 = p - 1$,
\autoref{Lemma:WeightLowerBoundp} gives $\min_{Q \in G_{E} Z P} \left( \wt(Q)
\right) \ge 2$, and \autoref{Lemma:EDpFactsForArgumentEdpBigq3} gives $[E(D_{p})
: E] \ge p^{2}$, so by \eqref{Equation:g3gSPminWt}, 
\[
g^{3} - g \ge 3 p (p^{2})(2) > 6 (p - 1)^{3} = 6 g^{3},
\]
which is impossible.

If $q \ge 5$, then $g = (p - 1)(q - 1) / 2$, \autoref{Lemma:WeightLowerBoundp}
gives $\min_{Q \in G_{E} Z P} \left( \wt(Q) \right) \ge g (q - 3) / 2$, and
\autoref{Lemma:EDpFactsForArgumentEdpEp} gives $[E(D_{p}) : E] \ge p$, so by
\eqref{Equation:g3gSPminWt}, 
\[
g^{3} - g \ge p q (p) \left(g \left( \frac{q - 3}{2} \right)\right) > g (p -
1)^{2} \left( \frac{5 (q - 1)^{2}}{16} \right) = \frac{5}{4} g^{3},
\]
which is impossible.
\end{proof}

\begin{lemma}
\label{Lemma:dTorsionSuperellipticGeneral}
Suppose that $Q = (x_{0}, y_{0})$ is a torsion point on a superelliptic curve
$y^{n} = f(x)$ where $f$ is monic, $d \colonequals \deg(f)$ is coprime to $n$,
and $d [Q - \infty] = 0$.  Then there exists $v(x) \in \mathbf{C}[x]$ with $\deg
v < d / n$ such that $\divisor (y - v(x)) = d Q - d \infty$ and $v(x)^{n} = f(x)
- (x - x_{0})^{d}$.
\end{lemma}
\begin{proof}
Let $\mathcal{F}$ be the function field of the curve. Since $d Q \sim d \infty$,
there exists $h \in \mathcal{F}$ such that 
\begin{equation}
\label{Equation:TooSmallTorsionDefinitionOfh}
\divisor(h) = d Q - d \infty. 
\end{equation}
Since $h$ only has poles at $\infty$, $h$ is a polynomial in $x$ and $y$. Since
the pole at $\infty$ has order $d$, it follows (after scaling $h$ by a constant)
that  $h = y - v(x)$ where $\deg(v) < d / n$. The $x$-map provides an inclusion
of function fields $\mathbf{C}(x) \subseteq \mathcal{F}$, so taking the norm of
both sides of \eqref{Equation:TooSmallTorsionDefinitionOfh} from $\mathcal{F}$
to $\mathbf{C}(x)$ yields
\[
\divisor(f(x) - (v(x))^{n}) = d \divisor(x - x_{0}) = \divisor((x -
x_{0})^{d}),
\]
so $f(x) - (v(x))^{n}$ and $(x - x_{0})^{d}$ are the same up to a constant
multiple; since $f$ is monic and $\deg v < d / n$, they are equal.
\end{proof}

\begin{proposition}
\label{Proposition:TooSmallTorsionPointImpossible} \hfill
\begin{enumerate}[label=\upshape(\arabic*),
ref=\autoref{Proposition:TooSmallTorsionPointImpossible}(\arabic*)]

\item \label{Proposition:TooSmallTorsionPointImpossiblea1b1}
If $a, b \le 1$, then $(a, b) \in \{ (0, 0), (0, 1), (1, 0) \}$.

\item \label{Proposition:TooSmallTorsionPointImpossibleminab0}
If $\min \{ a, b \} = 0$, then $(a, b) \in \{ (0, 0), (0, 1), (1, 0) \}$.

\end{enumerate}
\end{proposition}
\begin{proof} \hfill
\begin{enumerate}[label=\upshape(\arabic*),
ref={the proof of
\autoref{Proposition:TooSmallTorsionPointImpossible}(\arabic*)}]

\item \label{PropositionProof:TooSmallTorsionPointImpossiblea1b1}
Then $(1 - \zeta_{p})(1 - \zeta_{q}) [P - \infty] = 0$, which can be
rearranged to yield $P + \zeta_{p} \zeta_{q} P \sim \zeta_{p} P + \zeta_{q} P$,
so by \autoref{Lemma:NotHyperelliptic}, either $P = \zeta_{p} P$ or $P =
\zeta_{q} P$, meaning that either $a$ or $b$ is $0$.  

\item \label{PropositionProof:TooSmallTorsionPointImpossibleminab0}
Without loss of generality, suppose that $a = 0$. Then $q P \sim q \infty$, so
if we let $c$ be the $x$-coordinate of $P$ and define
\begin{equation}
\label{Equation:DefinitionOfLx}
L(x) \colonequals x^{q} + 1 - (x - c)^{q},
\end{equation}
then \autoref{Lemma:dTorsionSuperellipticGeneral} shows that there exists 
$v \in \mathbf{C}[x]$ such that 
\begin{equation}
\label{Equation:WeirdRelationshipQthPower}
L(x) = v(x)^{p}.
\end{equation}
A calculation yields
\begin{align*}
1 &= L - \left( \frac{2 x - c}{q} \right) L' + \left(\frac{x (x - c)}{q (q -
1)}\right) L'' &&\text{(by \eqref{Equation:DefinitionOfLx})} \\
&= v^{p - 2} \left( v^{2} - \left( \frac{p (2 x - c)}{q} \right) v v' +
\frac{p x (x - c)}{q (q - 1)} \left( (p - 1) (v')^{2} + v v'' \right) \right) &&
\text{(by \eqref{Equation:WeirdRelationshipQthPower})},
\end{align*}
so $v^{p - 2}$ divides $1$, implying $v$ is a constant, so the terms with $v'$
and $v''$ disappear and we obtain $v^{p} = 1$, so
\eqref{Equation:WeirdRelationshipQthPower} gives $L(x) = 1$, and then
\eqref{Equation:DefinitionOfLx} yields $c = 0$. Hence $[P - \infty] \in
\mathcal{J}_{p, q}[1 - \zeta_{q}]$, so $b \le 1$. \qedhere

\end{enumerate}
\end{proof}

\begin{theorem}
\label{Theorem:NoETPpqodd}
$P$ is not an exceptional torsion point; i.e., $(a, b) \in \{ (0, 0), (0, 1),
(1, 0) \}$.
\end{theorem}
\begin{proof}
If $\min \{ a, b \} = 0$, then we are done by
\autoref{Proposition:TooSmallTorsionPointImpossibleminab0}.

\begin{enumerate}[label=\textbf{Case~\Alph*:}, ref={Case~\Alph*},
leftmargin=*, itemindent=25pt]

\item \label{Case:NoETPpqoddminabge2}
$\min \{ a, b \} \ge 2$

Then \autoref{Proposition:LargeTorsionPointImpossiblea2b1} implies $q = 3$ and
\autoref{Proposition:LargeTorsionPointImpossiblea1b2} implies $p = 3$, which is
impossible since $p$ and $q$ are distinct odd primes.

\item \label{Case:NoETPpqoddminabeq1}
$\min \{ a, b \} = 1$

Without loss of generality, assume that $b = 1$ and $a \ge 1$. Then
\autoref{Proposition:TooSmallTorsionPointImpossiblea1b1} implies that $a \ge 2$,
so by \autoref{Proposition:LargeTorsionPointImpossiblea2b1}, $q = 3$ and
$a \in \{ 2, 3 \}$. Then $(1 - \zeta_{p})^{3} (1 - \zeta_{3}) [P - \infty] =
0$, which we can rewrite as
\begin{equation}
\label{Equation:Degree8Map}
\zeta_{p}^3 \zeta_{3} P + 3 \zeta_{p}^2 P + 3 \zeta_{p} \zeta_{3} P + P \sim
\zeta_{p}^3 P + 3 \zeta_{p}^2 \zeta_{3} P + 3 \zeta_{p} P + \zeta_{3} P
\end{equation}

\begin{enumerate}[label=\textbf{\theenumi\arabic*:}, ref={\theenumi\arabic*},
leftmargin=*, itemindent=35pt]

\item \label{Case:NoETPpqoddminabeq1SomeIntersection}
$\{ \zeta_{p}^3 \zeta_{3} P, \zeta_{p}^2 P,  \zeta_{p} \zeta_{3} P,  P
\} \cap \{ \zeta_{p}^3 P, \zeta_{p}^2 \zeta_{3} P, \zeta_{p} P, \zeta_{3} P \}
\neq \emptyset$

Then $P$ is fixed by some $z \in Z \setminus \{ 1 \}$, so it is fixed by either
$\zeta_{p}$ or $\zeta_{3}$, which implies that $(a, b) \in \{ (0, 0), (0, 1),
(1, 0) \}$.

\item \label{Case:NoETPpqoddminabeq1NoIntersection}
$\{ \zeta_{p}^3 \zeta_{3} P, \zeta_{p}^2 P,  \zeta_{p} \zeta_{3} P,  P \} \cap
\{ \zeta_{p}^3 P, \zeta_{p}^2 \zeta_{3} P, \zeta_{p} P, \zeta_{3} P \} =
\emptyset$

Then \eqref{Equation:Degree8Map} gives a degree $8$ map $\upsilon \colon
\mathcal{C}_{p, q} \to \mathbf{P}^{1}$, so applying
\autoref{Corollary:FollowFromCS} with $\upsilon$ and the $y$-map yields
\[
(3 - 1)(p - 1) / 2 \le (3 - 1)(8 - 1),
\]
so $p \le 15$; thus, $p \in \{ 5, 7, 11, 13 \}$.  By
\autoref{Lemma:L31L21L11q3p13}, there exists a nontrivial $\gamma \in \Gal(L_{a,
1} / L_{a - 1, 1})$. \autoref{Lemma:EDpFactsForArgumentEdpq3} gives $L_{a, 1} =
E(D_{p})$, so $\gamma$ moves $D_{p}$ and hence $\gamma$ moves $P$. Since $\gamma
- 1$ kills $\mathcal{J}[(1 - \zeta_{p})^{a - 1}]$,
\autoref{Lemma:HowToKillTorsionPower} gives that $(\gamma - 1)^{2}$ kills
$\mathcal{J}[(1 - \zeta_{p})^{2(a - 1)}]$. Also, $\gamma - 1$ kills
$\mathcal{J}[1 - \zeta_{q}]$, so $(\gamma - 1)^{2}$ also kills $\mathcal{J}[1 -
\zeta_{q}]$.  Hence $(\gamma - 1)^{2}$ kills $\mathcal{J}[(1 - \zeta_{p})^{2(a -
1)}(1 - \zeta_{q})]$, and since $2 (a - 1) \ge a$, it kills $P$. Therefore,
\[
\gamma^{2} P + P \sim 2 \gamma P,
\]
so \autoref{Lemma:NotHyperelliptic} implies $P = \gamma P$, contradicting the
fact that $\gamma$ moves $P$. \qedhere

\end{enumerate}
\end{enumerate}
\end{proof}

\subsubsection{The hyperelliptic case}
\label{Subsubsection:Hyperelliptic}

\begin{theorem}[\cite{coleman1986torsion}]
\label{Theorem:C25Torsion}
The set of exceptional torsion points of $\mathcal{C}_{2, 5}$ is the $Z$-orbit
of $(\sqrt[5]{4}, \sqrt{5})$. Each has exact order $(1 - \zeta_{5})^{3}$; in
particular, each is killed by $5$.
\end{theorem}
\begin{proof}
On pages 206--207 of \cite{coleman1986torsion}, Coleman computes the torsion
points of the curve $w^{5} = u(1 - u)$, which is isomorphic to $\mathcal{C}_{2,
5}$. See also \cite{poonen2001computing}.
\end{proof}

\begin{theorem}
\label{Theorem:MostHyperellipticCases}
When $q \ge 7$ is prime, $\mathcal{C}_{2, q}$ has no exceptional torsion points.
\end{theorem}
\begin{proof}
This is Theorem 1.1 of \cite{grant2004cuspidal}, which classifies torsion points
on the isomorphic curve $x^{q} = y(1 - y)$.
\end{proof}

\subsubsection{Some remaining curves} 
\label{Subsubsection:RemainingCurves}

\begin{proposition}
\label{Proposition:RemainingCurvesETP}\hfill
\begin{enumerate}[label=\upshape(\arabic*),
ref=\autoref{Proposition:RemainingCurvesETP}(\arabic*)]

\item \label{Proposition:RemainingCurvesETPNoETP}
For $(n, d) \in \{ (2, 9), (8, 3), (2, 15), (2, 25), (4, 5) \}$,
$\mathcal{C}_{n, d}$ has no exceptional torsion points. 

\item \label{Proposition:RemainingCurvesETPYesETP}
The set of exceptional torsion points of $\mathcal{C}_{4, 3}$ is the $Z$-orbit
of $(2, \sqrt{3})$. Each has exact order $(1 - \zeta_{4})(1 - \zeta_{3})^{2}$;
in particular, each is killed by $12$.
\end{enumerate}
\end{proposition}
\begin{proof}

A computation with \texttt{Magma} yields 
\begin{align*}
&\divisor\left( \zeta_{12}^3 + 2 \zeta_{12}^2 - 2 \zeta_{12} - 1 - 6 \left(
\frac{x^2 - (-2 \zeta_{12}^2 + 1) x y - x - 3 y^2 - (4 \zeta_{12}^2 - 2) y +
1}{y^3 + (6 \zeta_{12}^2 - 3) y^2 - 9 y - 6 \zeta_{12}^2 + 3} \right) \right)\\
&= (1 - \zeta_{4})(1 - \zeta_{3})^{2} (2, \sqrt{3}), \\
&\divisor( (12-4 \sqrt{3} y) x^2 + (18 y^2 - 8 \sqrt{3} y - 6) x + y^4 - 12
\sqrt{3} y^3 + 18 y^2 - 4 \sqrt{3} y + 9)\\
&=12 (2, \sqrt{3}) - 12 \infty,
\end{align*}
so this shows that the point $(2, \sqrt{3})$ is a torsion point of
$\mathcal{C}_{4, 3}$; hence, its $Z$-orbit will also consist of torsion points
of the same order.

Suppose for contradiction that $P$ were an exceptional torsion point of
$\mathcal{C}_{n, d}$ and that $P$ does not lie in the $Z$-orbit of $(2,
\sqrt{3})$ when $(n, d) = (4, 3)$.

\begin{enumerate}[label=\textbf{Case~\Alph*:}, ref={Case~\Alph*},
leftmargin=*, itemindent=25pt]

\item \label{RemainingCurvesETPCoverOf23}
$(n, d) \in \{ (2, 9), (4, 3), (8, 3) \}$

Let $\varphi_{n, d} \colon \mathcal{C}_{n, d} \to \mathcal{C}_{2, 3}$ be defined
by $\varphi_{n, d}(x, y) = (x^{d / 3}, y^{n / 2})$. Define $S_{0} \subseteq
\mathcal{C}_{2, 3}(\overline{\mathbf{Q}})$ as follows: for $(n, d) \in \{ (2,
9), (8, 3) \}$, $S_{0}$ is the union of the $Z$-orbit of $\{ \infty, (0, 1),
(-1, 0)\}$; for $(n, d) = (4, 3)$, $S_{0}$ is the union of the $Z$-orbit of $\{
\infty, (0, 1), (-1, 0), (2, 3) \}$. Our assumptions on $P$ imply $\varphi_{n,
d}(P) \not\in S_{0}$. \autoref{Proposition:2pqOrderTorsionExtensionm} gives $n d
[P - \infty] = 0$, so $\varphi_{n, d}(P) \in C_{2, 3}[n d]$ and hence $P$ must
lie in the finite set $S_{n, d} \colonequals \varphi_{n, d}^{-1}\left(
\mathcal{C}_{2, 3}[n d] \setminus S_{0} \right)$. 

Since $\mathcal{C}_{n, d}$ has good reduction at $71$, let $\mathcal{C}_{n, d,
71}$ be the reduced curve over $\mathbf{F}_{71}$, let $P_{71} \in
\mathcal{C}_{n, d, 71}(\overline{\mathbf{F}_{71}})$ be the reduction of $P$, and
let $S_{n, d, 71} \subseteq \mathcal{C}_{n, d, 71}(\overline{\mathbf{F}_{71}})$
be the reduction of $S_{n, d}$, so $P_{71} \in S_{n, d, 71}$ is such that $n d
P_{71} - n d \infty$ is a principal divisor.  Using division polynomials, we use
\texttt{Magma} to compute $S_{n, d, 71}$ explicitly and find that $S_{n, d, 71}
\subseteq \mathcal{C}_{n, d, 71}(\mathbf{F}_{71^{24}})$.  We use the
\texttt{IsPrincipal} feature of \texttt{Magma} over $\mathbf{F}_{71^{24}}$ to
find that there are no $Q \in S_{n, d, 71}$ such that $n d Q - n d \infty$ is a
principal divisor, so $P_{71}$, and hence $P$, cannot exist.

\item \label{RemainingCurvesETPCoverOf25}
$(n, d) \in (2, 15), (2, 25), (4, 5) \}$

Let $N_{n, d} = n d$ if $(n, d) \in \{ (2, 15), (2, 25) \}$ and let $N_{n, d} =
3 n d$ if $(n, d) \in \{ 4, 5 \}$.  By
\autoref{Proposition:2pqOrderTorsionExtension3m}, $N_{n, d} [P - \infty] = 0$.
Let $\varphi_{n, d} \colon \mathcal{C}_{n, d} \to \mathcal{C}_{2, 5}$ be defined
by $\varphi_{n, d}(x, y) = (x^{d / 5}, y^{n / 2})$ and $\mathcal{T}_{2, 5}$ be
the exceptional torsion points of $\mathcal{C}_{2, 5}$ listed in
\autoref{Theorem:C25Torsion}. As in \autoref{RemainingCurvesETPCoverOf23}, we
see that $P$ lies in the finite set $S_{n, d} \colonequals \varphi_{n,
d}^{-1}(\mathcal{T}_{2, 5})$.  Since $\mathcal{C}_{n, d}$ has good reduction at
$54001$, we can define the reduced curve $\mathcal{C}_{n, d, 54001}$ and the
reductions $P_{54001}$, $S_{n, d, 54001}$ of $P$, $S_{n, d}$ respectively. We use
\texttt{Magma} to compute $S_{n, d, 54001}$ explicitly and find that $S_{n, d,
54001} \subseteq \mathcal{C}_{n, d, 54001}(\mathbf{F}_{54001})$.  We use the
\texttt{IsPrincipal} feature of \texttt{Magma} over $\mathbf{F}_{54001}$ to
find that there are no $Q \in S_{n, d, 54001}$ such that $N_{n, d} Q - N_{n, d}
\infty$ is a principal divisor, so $P_{54001}$, and hence $P$, cannot exist.
\qedhere
\end{enumerate}
\end{proof}

\subsubsection{Main Theorem}

\begin{lemma}
\label{Lemma:NoETPIfQuotientHasNoETP}
Suppose that $n', d'$ are integers such that $n' | n$ and $d' | d$. If
$\mathcal{C}_{n', d'}$ has no exceptional torsion points, then neither does
$\mathcal{C}_{n, d}$.
\end{lemma}
\begin{proof}
The map $\mathcal{C}_{n, d} \to \mathcal{C}_{n', d'}$ given by $(x, y) \mapsto
(x^{d / d'}, y^{n / n'})$ sends exceptional torsion points to exceptional
torsion points.
\end{proof}

\MainTheoremOfPaper

\begin{proof}
Without loss of generality, suppose that $d$ is odd.

Suppose that $n$ is divisible by an odd prime $p$. Let $q$ be an odd prime
dividing $d$. Then by \autoref{Theorem:NoETPpqodd}, $\mathcal{C}_{p, q}$ has no
exceptional torsion points, so \autoref{Lemma:NoETPIfQuotientHasNoETP} implies
that $\mathcal{C}_{n, d}$ has no exceptional torsion points.

So we may assume that that $n = 2^i$ for an integer $i \ge 1$. If $d$ has a
prime factor $q \ge 7$, then \autoref{Theorem:MostHyperellipticCases} implies
that $\mathcal{C}_{2, q}$ has no exceptional torsion points, so
\autoref{Lemma:NoETPIfQuotientHasNoETP} implies that $\mathcal{C}_{n, d}$ has no
exceptional torsion points.

So we may assume that there exist integers $j, k \ge 0$ such that $d = 3^{j}
5^{k}$ and $(j, k) \neq (0, 0)$.

\begin{enumerate}[label=\textbf{Case~\Alph*:}, ref={Case~\Alph*},
leftmargin=*, itemindent=25pt]
\item \label{Case:ETPClassifyjk2}
$j + k \ge 2$

Then $n$ is divisible by $2$ and $d$ is divisible by either $9$, $15$, or $25$,
so we are done by \autoref{Proposition:RemainingCurvesETPNoETP} and
\autoref{Lemma:NoETPIfQuotientHasNoETP}.

\item \label{Case:ETPClassifyjk10}
$(j, k) = (1, 0)$

If $i \ge 3$, then $n$ is divisible by $8$. Since $d = 3$, we are done by
\autoref{Proposition:RemainingCurvesETPNoETP} and
\autoref{Lemma:NoETPIfQuotientHasNoETP}. The case $(n, d) = (4, 3)$ is handled
by \autoref{Proposition:RemainingCurvesETPYesETP}. The case $(n, d) = (2, 3)$ is
\autoref{Theorem:MainTheoremOfPaperC23}.

\item \label{Case:ETPClassifyjk01}
$(j, k) = (0, 1)$

If $i \ge 2$, then $n$ is divisible by $4$. Since $d = 5$, we are done by
\autoref{Proposition:RemainingCurvesETPNoETP} and
\autoref{Lemma:NoETPIfQuotientHasNoETP}. The case $(n, d) = (2, 5)$ is handled
by \autoref{Theorem:C25Torsion}. \qedhere
\end{enumerate}
\end{proof}

\section{Torsion points on a generic superelliptic curve}
\label{Section:TorsionPointsGenericSuperelliptic}

As usual, for any superelliptic curve $y^{n} = (x - a_{1}) \cdots (x - a_{d})$,
the automorphism $\zeta_{n}$ refers to the map given by$(x, y) \mapsto (x,
\zeta_{n} y)$. The points fixed by $\zeta_{n}$ are $\{ (a_{1}, 0), \dots,
(a_{d}, 0), \infty \}$, and they are torsion points whose order divides $n$.

The aim of this section is to prove the following result.
\GenericResult

This extends Theorem 7.1 of \cite{poonen2014most} from $n = 2$ to all $n$. To
prove \autoref{Theorem:GenericResult}, we need a few more results about torsion
points on certain curves.

\subsection{The curves \texorpdfstring{$y^n = x^d + x$}{y\^{}n = x\^{}d + x}}

\begin{proposition}
\label{Proposition:DirectNoTorsion}
Suppose that $n, d \ge 2$ are coprime, $P$ is a torsion point of $y^n = x^d + x$
whose order divides $d$, and $P \neq \infty$. Then $d = 2$ or $(n, d) = (2, 3)$.
\end{proposition}
\begin{proof}
Let the $x$-coordinate of $P$ be $c$. By
\autoref{Lemma:dTorsionSuperellipticGeneral}, there exists $v \in \mathbf{C}[x]$
with $\deg v < d / n$ such that 
\begin{equation}
\label{Equation:DirectNoTorsionDefv}
v(x)^{n} = x^d + x - (x - c)^d.
\end{equation}
Let $x' \colonequals x - c / 2$ and define $u(x) \colonequals v(x + c /
2)$. Using \eqref{Equation:DirectNoTorsionDefv} with $x$ and $c - x$, a
computation yields
\begin{align}
\label{Equation:WeirdRelation}
u(x')^n + (-1)^d u(-x')^n = \left( 1 - (-1)^d \right) x' + \left( 1 + (-1)^d
\right) \frac{c}{2}.
\end{align}

\begin{enumerate}[label=\textbf{Case~\Alph*:}, ref={Case~\Alph*},
leftmargin=*, itemindent=25pt]

\item \label{Case:DirectNoTorsiondEven}
$d$ is even

Suppose for contradiction that $d > 2$. Factoring the left hand side of
\eqref{Equation:WeirdRelation} yields
\[
\prod_{i = 0}^{n - 1} (u(x') + \zeta_{n}^i \cdot \zeta_{2 n} u(-x')) = c.
\]
In particular, $u(x') + \zeta_{2 n} u(-x')$ and $u(x') + \zeta_{2 n} \cdot
\zeta_{n} u(-x')$ are forced to be constants, so $u(x')$ and $u(-x')$ are
constants, so $v(x)$ is constant, so by \eqref{Equation:DirectNoTorsionDefv},
\begin{equation}
x^d + x - (x - c)^d\text{ is constant.} \label{Equation:LIsConstant}
\end{equation}
Since $d > 2$, the $x^{d - 1}$-coefficient of $x^d + x - (x - c)^d$ is $d c$, so
\eqref{Equation:LIsConstant} implies $c = 0$, so $x^d + x - (x - c)^d = x$, but
this contradicts \eqref{Equation:LIsConstant}.

\item \label{Case:DirectNoTorsiondOdd}
$d$ is odd

Factoring the left hand side of \eqref{Equation:WeirdRelation} yields
\begin{equation}
\label{Equation:DirectNoTorsiondOdd}
\prod_{i = 0}^{n - 1} (u(x') - \zeta_{n}^i u(-x')) = 2 x'.
\end{equation}

\begin{enumerate}[label=\textbf{\theenumi\arabic*:}, ref={\theenumi\arabic*},
leftmargin=*, itemindent=35pt]

\item \label{Case:DirectNoTorsiondOddnge3}
$n \ge 3$

Considering the degree of each factor in \eqref{Equation:DirectNoTorsiondOdd}
shows that at least two of them must be constants, which will force $u(x')$ and
$u(-x')$ to be constant, and we can repeat the same argument as in
\autoref{Case:DirectNoTorsiondEven} to get a contradiction. 

\item \label{Case:DirectNoTorsiondOddneq2}
$n = 2$

Then \eqref{Equation:DirectNoTorsiondOdd} becomes
\begin{equation}
\label{Equation:DirectNoTorsiondOddn2}
(u(x') + u(-x'))(u(x') - u(-x')) = 2 x'.
\end{equation}
Since $u(x') + u(-x')$ is an even polynomial and $u(x') - u(-x')$ is
an odd polynomial, \eqref{Equation:DirectNoTorsiondOddn2} forces $u(x') +
u(-x')$ to be constant and $u(x') - u(-x')$ to be a multiple of $x'$. Then $\deg
u = 1$, so $\deg v = 1$. Let $v(x) = a x + b$, so
\eqref{Equation:DirectNoTorsionDefv} gives
\begin{equation}
\label{Equation:WeirdRelationaxbsquare}
(a x + b)^2 = x^d + x - (x - c)^d.
\end{equation}
Considering the coefficient of $x^{d - 1}$, we conclude that either $c = 0$ or
$d = 3$. If $c = 0$, then \eqref{Equation:WeirdRelationaxbsquare} implies that
$x = (a x + b)^2$, which is impossible. So we conclude that $(n, d) = (2, 3)$.
\qedhere

\end{enumerate}
\end{enumerate}
\end{proof}

\subsection{Two curves for which \texorpdfstring{$n + d = 7$}{n + d = 7}}

\begin{proposition}\hfill
\label{Proposition:WeirdCoversOfEC}
\begin{enumerate}[label=\upshape(\arabic*),
ref=\autoref{Proposition:WeirdCoversOfEC}(\arabic*)]

\item \label{Proposition:WeirdCoversOfEC34}
If $P$ is a torsion point on $y^{3} = x^{4} + x^{2} + 1$ with $12[P - \infty] =
0$, then $P$ is fixed by $\zeta_{3}$.

\item \label{Proposition:WeirdCoversOfEC43}
If $P$ is a torsion point on $y^{4} = x^{3} + x^{2} + 1$ with $12[P - \infty] =
0$, then $P$ is fixed by $\zeta_{4}$.

\end{enumerate}
\begin{proof}
Let $\mathcal{C}$ be the curve $y^{3} = x^{4} + x^{2} + 1$, let $E$ be the
elliptic curve $y^3 = x^2 + x + 1$, let $\varphi \colon \mathcal{C} \to E$
be the $2$-to-$1$ map $(x, y) \mapsto (x^{2}, y)$, let $S_{0}$ be the points of
$E$ fixed by $\zeta_{3}$, and suppose for contradiction that $P$ is a torsion
point of $\mathcal{C}$ with $12[P - \infty] = 0$ such that $P$ is not fixed by
$\zeta_{3}$. Then $\varphi(P) \in E[12]$, so $P$ lies in the finite set
$S \colonequals \varphi^{-1}(E[12] \setminus S_{0})$.

Since $\mathcal{C}$ has good reduction at $47$, let $\mathcal{C}_{47}$ be the
reduced curve over $\mathbf{F}_{47}$, let $P_{47} \in
\mathcal{C}_{47}(\overline{\mathbf{F}_{47}})$ be the reduction of $P$, and let
$S_{47} \subseteq \mathcal{C}_{47}(\overline{\mathbf{F}_{47}})$ be the reduction
of $S$, so $P_{47} \in S_{47}$ is such that $12 P_{47} - 12 \infty$ is a
principal divisor.  Using division polynomials, we use \texttt{Magma} to compute
$S_{47}$ explicitly and find that $S_{47} \subseteq
\mathcal{C}_{47}(\mathbf{F}_{47^{4}})$.  We use the \texttt{IsPrincipal} feature
of \texttt{Magma} over $\mathbf{F}_{47^{4}}$ to find that there are no $Q \in
S_{47}$ such that $12 Q - 12 \infty$ is a principal divisor, so $P_{47}$, and
hence $P$, cannot exist.

The curve $y^4 = x^3 + x + 1$ is a $2$-to-$1$ cover of the elliptic curve $y^2 =
x^3 + x + 1$ and the same technique happens to work over $\mathbf{F}_{47^{4}}$
again.
\end{proof}
\end{proposition}

\subsection{Proof of \autoref{Theorem:GenericResult}} 

\begin{enumerate}[label=\textbf{Case~\Alph*:}, ref={Case~\Alph*},
leftmargin=*, itemindent=25pt]
\item \label{Case:GenericResultd2}
$d = 2$

$\mathcal{C}_{2, n}$ is isomorphic over $k$ to $y^n = (x - a_1)(x - a_2)$ via
the isomorphism
\[
(x, y) \in \mathcal{C}_{2, n} \mapsto \left( \frac{(a_2 - a_1) y + (a_1 +
a_2)}{2}, \sqrt[n]{\frac{(a_2-a_1)^2}{4}} x \right) \in \mathscr{C}_{n},
\]
so \autoref{Theorem:MainTheoremOfPaper} 
gives \autoref{Theorem:GenericResultd2n7} and
\autoref{Theorem:GenericResultd2n5}.

\item \label{Case:GenericResultd3}
$d \ge 3$

Suppose that $P$ is a torsion point of $\mathscr{C}_{n}$ of order $m$. Let $M =
\lcm(m, n d)$. Since $\mathscr{J}_{n}[M]$ is a finite \'etale cover of $\Spec
k$, every specialization map will induce an isomorphism on the $M$-torsion of
the jacobian. 

\begin{enumerate}[label=\textbf{\theenumi\arabic*:}, ref={\theenumi\arabic*},
leftmargin=*, itemindent=35pt]

\item \label{Case:GenericResultd3nd2552}
$(n, d) \not \in \{ (3, 4), (4, 3) \}$

Specializing to $\mathcal{C}_{n, d}$ and using
\autoref{Theorem:MainTheoremOfPaper} gives $(1 - \zeta_{n}) [P - \infty] = 0$ or
$d[P - \infty] = 0$. If $d[P - \infty] = 0$, then specializing to $y^{n} = x^{d}
+ x$ and using \autoref{Proposition:DirectNoTorsion} gives $P = \infty$.

\item \label{Case:GenericResultd3nd34}
$(n, d) = (3, 4)$

Specializing to $\mathcal{C}_{3, 4}$ and using
\autoref{Theorem:MainTheoremOfPaper} gives $12 [P - \infty] = 0$. Specializing
to $y^{3} = x^{4} + x^{2} + 1$ and using \autoref{Proposition:WeirdCoversOfEC34}
gives $(1 - \zeta_{3}) [P - \infty] = 0$.

\item \label{Case:GenericResultd3nd43}
$(n, d) = (4, 3)$

Specializing to $\mathcal{C}_{4, 3}$ and using
\autoref{Theorem:MainTheoremOfPaper} gives $12 [P - \infty] = 0$. Specializing
to $y^{4} = x^{3} + x + 1$ and using \autoref{Proposition:WeirdCoversOfEC43}
gives $(1 - \zeta_{4}) [P - \infty] = 0$. \qedhere

\end{enumerate}
\end{enumerate}


\newcommand{\noopsort}[1]{} \newcommand{\printfirst}[2]{#1}
  \newcommand{\singleletter}[1]{#1} \newcommand{\switchargs}[2]{#2#1}

\end{document}